\documentclass[11pt,reqno]{amsart}
\usepackage{amsmath, color}
\usepackage{amssymb}
\usepackage{a4wide}
\usepackage{graphics}
\usepackage{epsfig}

\newtheorem{theorem}{Theorem}[section]
\newtheorem{lemma}[theorem]{Lemma}
\newtheorem{proposition}[theorem]{Proposition}

\newtheorem{remark}[theorem]{Remark}

\newcommand{\R}{\mathbb{R}}
\newcommand{\E}{\mathrm{E}}
\renewcommand{\P}{\mathrm{P}}
\newcommand{\ta}{\theta}

\makeatletter
\@addtoreset{equation}{section}
\makeatother

\allowdisplaybreaks
\usepackage{appendix}
\usepackage{hyperref}
\title{LAN property for an ergodic Ornstein-Uhlenbeck process with Poisson jumps}
\date{\today}
\author[Ngoc Khue Tran]{Ngoc Khue Tran}

\address{Ngoc Khue Tran, Department of Mathematical Sciences - Ritsumeikan University and Japan Science and Technology Agency, 1-1-1 Nojihigashi, Kusatsu, Shiga, 525-8577, Japan and Department of Fundamental Sciences - Pham Van Dong University, Phan Dinh Phung, Quang Ngai City, Quang Ngai, Vietnam}

\email{tnkhueprob@gmail.com}


\begin{document}
\maketitle
\begin{abstract} 
In this paper, we consider an ergodic Ornstein-Uhlenbeck process with jumps driven by a Brownian motion and a compensated Poisson process, whose drift and diffusion coefficients as well as its jump intensity depend on unknown parameters. Considering the process discretely observed at high frequency, we derive the local asymptotic normality property. To obtain this result, Malliavin calculus and Girsanov's theorem are applied to write the log-likelihood ratio in terms of sums of conditional expectations, for which a central limit theorem for triangular arrays can be applied. 
\end{abstract}
	
\section{Introduction}
	
On a complete probability space $(\Omega, \mathcal{F}, \P)$ which will be specified later on, we consider an Ornstein-Uhlenbeck (O-U) process with Poisson jumps $X^{\theta,\sigma,\lambda}=(X_t^{\theta,\sigma,\lambda})_{t\geq 0}$ in $\R$ defined by 
\begin{equation}\label{c2eq1}
X_t^{\theta,\sigma,\lambda}=x_0-\theta\int_0^t X_s^{\theta,\sigma,\lambda}ds +\sigma B_t + N_t -\lambda t,
\end{equation}
where $x_0\in\R $, $B=(B_t)_{t \geq 0}$ is a Brownian motion, $N=(N_t)_{t \geq 0}$ is a Poisson process with intensity $\lambda >0$ independent of $B$, and $(\widetilde{N}_{t}^{\lambda})_{t \geq 0}$ denotes the compensated Poisson process  $\widetilde{N}_{t}^{\lambda}:=N_t-\lambda t$. The parameters $(\theta,\sigma,\lambda)\in \Theta\times \Sigma \times\Lambda$ are unknown, and $\Theta, \Sigma$ and $\Lambda$ are closed intervals of $\R^{\ast}_+$, where $\R^{\ast}_+=\R_+\setminus\{0\}$. Let $\{\widehat{\mathcal{F}}_t\}_{t\geq 0}$ denote the natural filtration generated by $B$ and $N$. We denote by $\P^{\theta,\sigma,\lambda}$ the probability law induced by the $\{\widehat{\mathcal{F}}_t\}_{t\geq 0}$-adapted c\`{a}dl\`{a}g stochastic process $X^{\theta,\sigma,\lambda}$, and by $\E^{\theta,\sigma,\lambda}$ the expectation with respect to (w.r.t.) $\P^{\theta,\sigma,\lambda}$. Let $\overset{\P^{\theta,\sigma,\lambda}}{\longrightarrow}$ and $\overset{\mathcal{L}(\P^{\theta,\sigma,\lambda})}{\longrightarrow}$ denote the convergence in $\P^{\theta,\sigma,\lambda}$-probability and in $\P^{\theta,\sigma,\lambda}$-law, respectively.
	
Since $\theta>0$, $X^{\theta,\sigma,\lambda}$ is ergodic with a unique invariant probability measure $\pi_{\theta,\sigma,\lambda}(dx)$. Furthermore, $\pi_{\theta,\sigma,\lambda}(dx)$ can be calculated explicitly (see \cite[Theorem 17.5 and Corollary 17.9]{SK}), and satisfies $\int_{\mathbb{R}}\vert x\vert^p\pi_{\theta,\sigma,\lambda}(dx)<\infty$, for any $p\geq 0$. By It\^o's formula, $X^{\theta,\sigma,\lambda}$ is given by
\begin{equation}\label{solution}
X_t^{\theta,\sigma,\lambda}=x_0e^{-\theta t}+\sigma\int_0^te^{-\theta (t-s)}dB_s + \int_0^te^{-\theta (t-s)}d\left(N_s -\lambda s\right).
\end{equation}
	
In this paper, we are interested in the estimation of the parameters $(\theta,\sigma,\lambda)\in \Theta\times \Sigma \times\Lambda$ when the observation of the process $X^{\theta,\sigma,\lambda}$ is the discretized path $(X_{k\Delta_n}^{\theta,\sigma,\lambda})_{0\leq k\leq n}$, where $\{\Delta_n\}_{n\in\mathbb{N}^{\ast}}$ is a sequence of time-step sizes. We denote by $\P^{\theta,\sigma,\lambda}_n$ the probability law of the random vector $(X_{k\Delta_n}^{\theta,\sigma,\lambda})_{0\leq k\leq n}$. Let us now introduce the fundamental concept in parametric estimation called the local asymptotic normality (LAN) property which was introduced by Le Cam \cite{LC60} and extended by Jeganathan \cite{JP82} to the local asymptotic mixed normality (LAMN) property. In our context, we say that the sequence $\{\P^{\theta,\sigma,\lambda}_n\}_{(\theta,\sigma,\lambda)\in \Theta\times \Sigma \times\Lambda}$ of parametric statistical models has the LAN property for the likelihood at $(\theta_0,\sigma_0,\lambda_0)\in \Theta\times \Sigma \times\Lambda$ with rate of convergence $(\sqrt{n\Delta_n},\sqrt{n},\sqrt{n\Delta_n})$ and with covariance matrix $\Gamma(\theta_0,\sigma_0,\lambda_0)$ if for any $z=(u,v,w)\in\R^3$,
\begin{align*}
&\log\dfrac{d\P^{\theta_0+\frac{u}{\sqrt{n\Delta_n}},\sigma_0+\frac{v}{\sqrt{n}},\lambda_0+\frac{w}{\sqrt{n\Delta_n}}}_n}{d\P^{\theta_0,\sigma_0,\lambda_0}_n}\left((X_{k\Delta_n}^{\theta_0,\sigma_0,\lambda_0})_{0\leq k\leq n}\right)\\
&\qquad\overset{\mathcal{L}(\P^{\theta_0,\sigma_0,\lambda_0})}{\longrightarrow} z^{\ast}\mathcal{N}\left(0,\Gamma(\theta_0,\sigma_0,\lambda_0)\right)-\dfrac{1}{2}z^{\ast} \Gamma(\theta_0,\sigma_0,\lambda_0)z,
\end{align*}
as $n\to\infty$, where $\mathcal{N}(0,\Gamma(\theta_0,\sigma_0,\lambda_0))$ is a centered $\R^3$-valued Gaussian vector with covariance matrix $\Gamma(\theta_0,\sigma_0,\lambda_0)$, and $\ast$ denotes the transpose. Here, $\Gamma(\theta_0,\sigma_0,\lambda_0)$ is called the asymptotic Fisher information matrix of the parametric statistical models.
	
Recall that one can define the notion of asymptotic efficiency of estimators in terms of classical Cram\'er-Rao lower bound. Another approach to define the asymptotic efficiency of the estimators is to study the lower bound for the asymptotic variance of the estimators via a convolution theorem. In fact, this problem is closely related to the consequence of the LAN property. Precisely, when the LAN property holds at $(\theta_0,\sigma_0,\lambda_0)$ with nonsingular $\Gamma(\theta_0,\sigma_0,\lambda_0)$, convolution and minimax theorems can be applied (see \cite{Haj72}, \cite{CY90}). Indeed, on one hand, the convolution theorem suggests the notion of asymptotically efficient estimators in terms of minimal asymptotic variance. That is, a sequence $\{(\widehat{\theta}_n,\widehat{\sigma}_n,\widehat{\lambda}_n)\}_{n\in\mathbb{N}^{\ast}}$ of unbiased estimators of the parameter $(\theta_0,\sigma_0,\lambda_0)$ is said to be asymptotically efficient at $(\theta_0,\sigma_0,\lambda_0)$ in the sense of H\'ajek-Le Cam convolution theorem if as $n\to\infty$,
$$
\varphi_n^{-1}(\theta_0,\sigma_0,\lambda_0)\left((\widehat{\theta}_n,\widehat{\sigma}_n,\widehat{\lambda}_n)-(\theta_0,\sigma_0,\lambda_0)\right)^{\ast}\overset{\mathcal{L}(\P^{\theta_0,\sigma_0,\lambda_0})}{\longrightarrow}\mathcal{N}\left(0,\Gamma(\theta_0,\sigma_0,\lambda_0)^{-1}\right),
$$
where $\varphi_n(\theta_0,\sigma_0,\lambda_0):=\text{diag}(\sqrt{n\Delta_n},\sqrt{n},\sqrt{n\Delta_n})$ is the diagonal matrix. Notice that a sequence of estimators which is asymptotically efficient in the sense of H\'ajek-Le Cam convolution theorem achieves asymptotically the Cram\'er-Rao lower bound $\Gamma(\theta_0,\sigma_0,\lambda_0)^{-1}$ for the estimation variance. On the other hand, the minimax theorem states that $\Gamma(\theta_0,\sigma_0,\lambda_0)^{-1}$ gives the lower bound for the asymptotic variance of any sequence of unbiased estimators. This justifies the importance of the LAN property in parametric estimation. 

The question of asymptotic efficiency of estimators for ergodic diffusions with jumps was solved e.g. in \cite{M14, SY06}. The estimators in \cite{SY06} are constructed from a contrast function which is based on a discretization of the likelihood function associated to the continuous observations of an ergodic diffusion with jumps whose unknown parameters appear in the drift and diffusion coefficients and in the jump coefficient as well. In \cite{M14}, to estimate the drift parameter of the Ornstein-Uhlenbeck processes driven by a L\'evy process, the estimators are constructed from a discretization of the time-continuous maximum likelihood estimators. These estimators are asymptotically efficient in the sense of H\'ajek-Le Cam convolution theorem (see \cite[Theorem 2.1, Remark 2.2]{SY06}, \cite[Theorem 3.5, Remark 3.6]{M14}).
	
Initiated by Gobet \cite{G01}, Malliavin calculus has recently been used to analyze the log-likelihood ratio of the discrete observation of diffusion processes (with jumps). Concretely, Gobet \cite{G01}, \cite{G02} obtained the LAMN and LAN properties respectively for multidimensional elliptic diffusions and ergodic diffusions on the basis of discrete observations at high frequency. In the presence of jumps, see e.g. A\"it-Sahalia and Jacod \cite{AJ07}, Kawai \cite{K13}, Cl\'ement et {\it al.} \cite{CDG14, CG15}, Kohatsu-Higa et {\it al.} \cite{KNT14, KNT15}. It seems that the validity of the LAN property for general stochastic differential equations (SDEs) with jumps whose unknown parameters appear in the drift and diffusion coefficients and in the jump component as well, has never been addressed in the literature. This is because the behaviour of the transition density, which is not explicit in general, changes strongly due to the presence of jumps. In this paper, we focus on the study of the global parameter estimation for the O-U process \eqref{c2eq1} whose unbounded drift and diffusion coefficients as well as its jump intensity depend on unknown parameters. Recall that in \cite{KNT14}, the global parameter estimation is considered for a L\'evy process with bounded drift coefficient. Recently, \cite{KNT15} has studied a multidimensional SDE with jumps whose unknown parameter appears only in the drift coefficient. Kawai \cite{K13} has dealt with an O-U process with jumps whose jump component does not depend on the unknown parameter. 
	
In physics, the O-U process describes the velocity of a massive Brownian particle under the influence of friction. In mathematical finance, it is widely used to model the evolution of interest rates, currency exchange rates (see e.g. \cite{BN01, DJ06}). The O-U process has a density which is highly tractable, the proof of our result, however, seems to be quite complicated. One of the motivations of writing this paper is to present a methodology which can be used to derive the LAN property for general SDEs with jumps with global parameters. 
	
The objective of this paper is to derive the LAN property for $X^{\theta,\sigma,\lambda}$ based on discrete observations. Towards this aim, our result uses the Girsanov's theorem and the Malliavin calculus w.r.t. the Brownian motion to derive an explicit expression for the logarithm derivatives of the transition density. Moreover, when treating the negligible terms, one difficulty comes from the fact that the conditional expectations are computed under the probability measures $\P^{\theta(\ell),\sigma_0,\lambda_0}$ (or $\P^{\theta_n,\sigma(\ell),\lambda_n}$, or $\P^{\theta_n,\sigma_0,\lambda(\ell)}$) which come from the application of the Malliavin calculus and the Girsanov's theorem, whereas the convergence is considered under $\P^{\theta_0,\sigma_0,\lambda_0}\neq \P^{\theta(\ell),\sigma_0,\lambda_0}$ (or $\neq\P^{\theta_n,\sigma(\ell),\lambda_n}$, or $\neq\P^{\theta_n,\sigma_0,\lambda(\ell)}$). To simplify the exposition, these parameters and corresponding measures will be specified in Section \ref{mainresult} and Subsection \ref{expand}. To this end, when two corresponding diffusion parameters are the same as $\sigma_0$, the Girsanov's theorem in Lemma \ref{c2Girsanov2} can be applied to change the measures. When they are different ($\sigma(\ell)\neq \sigma_0$), we need to condition on all the possible number of jumps and jump times of the Poisson process occurring on each interval $[t_k,t_{k+1}]$ in order that the change of measures can be done via the transition densities conditioned on the number of jumps and jump times (see Lemmas \ref{change}, \ref{lemma6} and \ref{lemma7}). As a consequence, the Gaussian type expression for these transition densities is strongly used. Moreover, our proof uses the large deviation type estimates which should be understood in the sense that we study the behaviour of rare events whose decreasing rate is exponential and polynomial (see Lemma \ref{deviation}). For this, the key argument consists in conditioning on the number of jumps within the conditional expectation and in the conditioning random variable, which expresses the transition density conditioned on the number of jumps. When these two conditionings relate to different jumps, we may use a large deviation principle in the estimate. When they are equal, we use the complementary set in order to apply the large deviation principle. Within all these arguments, the Gaussian type expression for the transition density conditioned on the jump structure is again strongly used. Let us mention here that the exponential decay comes from the fact that the study of the asymptotic behaviour of the transition density leads us to study the behaviour of the transition density under the additional condition on the number of jumps which has to be compared with another transition density with a different number of jumps. This will be seen in the estimate of the terms $M_{0,1,1,p}^{\theta,\sigma,\lambda}$ and $M_{1,1,0,1,p}^{\theta,\sigma,\lambda}$ in the proof of Lemma \ref{deviation}.
	
This paper is organized as follows. In Section 2, we state our main result in Theorem \ref{c2theorem}. Section 3 introduces preliminary results needed for the proof of Theorem \ref{c2theorem}, such as Girsanov's theorem, explicit expressions for the logarithm derivatives of the transition density using Malliavin calculus and Girsanov's theorem, and decompositions of the Skorohod integral. We prove our main result in Section 4. Section 5 is devoted to the conclusion and discussion of some future work. The proofs of some technical propositions and lemmas are presented in Section 6, where some crucial estimates related to the transition density conditioned on the jump structure and the large deviation type estimates are derived.
	
\section{Main result}
\label{mainresult}
For fixed $(\theta_0,\sigma_0,\lambda_0)\in \Theta\times \Sigma \times\Lambda$ and $n\geq 1$, we consider a discrete observation scheme at equidistant times $t_k=k \Delta_n$, $k \in \{0,\ldots,n\}$ of the process $X^{\theta_0,\sigma_0,\lambda_0}$, which is denoted by $X^{n}=(X_{t_0}, X_{t_1},\ldots,X_{t_n})$, where $\Delta_n\leq 1$ for all $n\geq 1$. We assume that the high-frequency, infinite horizon and decreasing rate conditions hold. That is, $\Delta_n\rightarrow 0$, $n\Delta_n\rightarrow\infty$, and $n\Delta_n^2\rightarrow\infty$ as $n\rightarrow\infty$. 
	
Given the process $X^{\theta,\sigma,\lambda}=(X_t^{\theta,\sigma,\lambda})_{t \geq 0}$, we denote by $p_n(\cdot;(\theta,\sigma,\lambda))$ the density of the random vector $(X_{t_0}^{\theta,\sigma,\lambda},X_{t_1}^{\theta,\sigma,\lambda},\ldots,X_{t_n}^{\theta,\sigma,\lambda})$. In particular, $p_n(\cdot;(\theta_0,\sigma_0,\lambda_0))$ denotes the density of the discrete observation $X^n$. For $(u,v,w)\in\R^3$, we set $\theta_n:=\theta_0+\frac{u}{\sqrt{n\Delta_n}}, \sigma_n:=\sigma_0+\frac{v}{\sqrt{n}}, \lambda_n:=\lambda_0+\frac{w}{\sqrt{n\Delta_n}}$. 
	
The main result of this paper is the following LAN property.
\begin{theorem}\label{c2theorem}
The LAN property holds for the likelihood at $(\theta_0,\sigma_0,\lambda_0)$ with rate of convergence $(\sqrt{n\Delta_n},\sqrt{n},\sqrt{n\Delta_n})$ and asymptotic Fisher information matrix $\Gamma(\theta_0,\sigma_0,\lambda_0)$. That is, for all $z=(u,v,w)\in\R^3$, as $n\to\infty$,
\begin{equation*} 
\log\dfrac{p_n\left(X^{n};\left(\theta_n,\sigma_n,\lambda_n\right)\right)}{p_n\left(X^{n};\left(\theta_0,\sigma_0,\lambda_0\right)\right)}\\
\overset{\mathcal{L}(\P^{\theta_0,\sigma_0,\lambda_0})}{\longrightarrow} z^{\ast}\mathcal{N}\left(0,\Gamma(\theta_0,\sigma_0,\lambda_0)\right)-\dfrac{1}{2}z^{\ast} \Gamma(\theta_0,\sigma_0,\lambda_0)z,
\end{equation*}
where $\mathcal{N}(0,\Gamma(\theta_0,\sigma_0,\lambda_0))$ is a centered $\R^3$-valued Gaussian vector with covariance matrix 
$$\Gamma(\theta_0,\sigma_0,\lambda_0):=\dfrac{1}{\sigma_0^2}\begin{pmatrix}\frac{1}{2\theta_0}\left(\sigma_0^2+1\right)&0&0\\0&2&0\\
0&0&1+\frac{\sigma_0^2}{\lambda_0}
\end{pmatrix}.$$
\end{theorem}
\begin{remark}
Our result can be viewed as an extension of the result obtained by Gobet \cite{G02} which is only valid for the one-dimensional linear case with the presence of jumps. Moreover, when the Poisson component is degenerate ($\lambda=0$), we recover the asymptotic Fisher information matrix of ergodic Ornstein-Uhlenbeck process without jumps (see Gobet \cite[Theorem 4.1]{G02}).
\end{remark}
	
\section{Preliminaries}
\label{sec:prelim}
	
In this section, we introduce some preliminary results needed for the proof of Theorem \ref{c2theorem}. Towards this aim, we consider the canonical filtered probability spaces $(\Omega^i, \mathcal{F}^i,\{\mathcal{F}_t^i\}_{t \geq 0}, \P^i)$, $i\in\{1,\ldots,4\}$, associated respectively to each of the four processes $B, N, W$ and $M$, where $W=(W_t)_{t\geq 0}$ is a Brownian motion, and $M=(M_t)_{t\geq 0}$ is a Poisson process with intensity $\lambda$. The processes $(B,N,W,M)$ are mutually independent. Let $(\Omega,\mathcal{F},\{\mathcal{F}_t\}_{t \geq 0},\P)$ be the product filtered probability space of these four canonical spaces. We set $\widehat{\Omega}=\Omega^1\times \Omega^2$, $\widehat{\mathcal{F}}=\mathcal{F}^1\otimes\mathcal{F}^2$, $\widehat{\P}=\P^1\otimes\P^2$, $\widehat{\mathcal{F}}_t=\mathcal{F}_t^1\otimes\mathcal{F}_t^2$, $\widetilde{\Omega}=\Omega^3\times \Omega^4$, $\widetilde{\mathcal{F}}=\mathcal{F}^3\otimes\mathcal{F}^4$, $\widetilde{\P}=\P^3\otimes\P^4$, and $\widetilde{\mathcal{F}}_t=\mathcal{F}_t^3\otimes\mathcal{F}_t^4$. Then, $\Omega=\widehat{\Omega}\times\widetilde{\Omega}$, $\mathcal{F}=\widehat{\mathcal{F}}\otimes\widetilde{\mathcal{F}}$, $\P=\widehat{\P}\otimes\widetilde{\P}$, $\mathcal{F}_t=\widehat{\mathcal{F}}_t\otimes\widetilde{\mathcal{F}}_t$, and $\E=\widehat{\E}\otimes\widetilde{\E}$, where $\E$, $\widehat{\E}$, $\widetilde{\E}$ denote the expectation w.r.t. $\P$, $\widehat{\P}$ and $\widetilde{\P}$, respectively.
	
In order to deal with the likelihood ratio in Theorem \ref{c2theorem}, we use the following decomposition 
\begin{equation}\begin{split}
\label{eq:dec} 
\log\dfrac{p_n\left(X^{n};\left(\theta_n,\sigma_n,\lambda_n\right)\right)}{p_n\left(X^{n};\left(\theta_0,\sigma_0,\lambda_0\right)\right)}&=\log\dfrac{p_n\left(X^{n};\left(\theta_n,\sigma_n,\lambda_n\right)\right)}{p_n\left(X^{n};\left(\theta_n,\sigma_0,\lambda_n\right)\right)}+\log\dfrac{p_n\left(X^{n};\left(\theta_n,\sigma_0,\lambda_n\right)\right)}{p_n\left(X^{n};\left(\theta_n,\sigma_0,\lambda_0\right)\right)}\\
&\qquad+\log\dfrac{p_n\left(X^{n};\left(\theta_n,\sigma_0,\lambda_0\right)\right)}{p_n\left(X^{n};\left(\theta_0,\sigma_0,\lambda_0\right)\right)}.
\end{split}
\end{equation}
For each of the above terms, we use the Markov property, the mean value theorem on the parameter space and then analyze each term, which leads to the logarithm derivatives of the transition density function w.r.t. the parameters. To analyze these logarithm derivatives, we start as in Gobet \cite{G01} by applying the Malliavin calculus integration by parts formula on each $[t_k, t_{k+1}]$ to derive an explicit expression for the logarithm derivatives of the transition density w.r.t. $\theta$ and $\sigma$. To avoid confusion with the observed process $X^{\theta,\sigma,\lambda}$, we introduce an extra probabilistic representation of $X^{\theta,\sigma,\lambda}$ for which the Malliavin calculus is applied. Explicitly, we consider on the same probability space $(\Omega, \mathcal{F}, \P)$ the flow $Y^{\theta,\sigma,\lambda}(s,x)=(Y_t^{\theta,\sigma,\lambda}(s,x), t\geq s)$, $x\in\R$ on the time interval $[s,\infty)$ and with initial condition $Y_{s}^{\theta,\sigma,\lambda}(s,x)=x$ satisfying
\begin{equation}\label{flow}
Y_t^{\theta,\sigma,\lambda}(s,x)=x-\theta\int_{s}^{t}Y_u^{\theta,\sigma,\lambda}(s,x)du+\sigma \left(W_t-W_s\right)+ \widetilde{M}_{t}^{\lambda}-\widetilde{M}_{s}^{\lambda},
\end{equation}
where $(\widetilde{M}_{t}^{\lambda})_{t \geq 0}$ is the compensated Poisson process $\widetilde{M}_{t}^{\lambda}:=M_t-\lambda t$. In particular, we write $Y_t^{\theta,\sigma,\lambda}\equiv Y_t^{\theta,\sigma,\lambda}(0,x_0)$, for all $t\geq 0$. That is, 
\begin{equation}\label{c2eq1rajoute}
Y_t^{\theta,\sigma,\lambda}=x_0-\theta \int_0^t Y_s^{\theta,\sigma,\lambda}ds +\sigma W_t+ \widetilde{M}_{t}^{\lambda}.
\end{equation}
	
We will apply the Malliavin calculus on the Wiener space induced by $W$. Let $D$ and $\delta$ denote the Malliavin derivative and the Skorohod integral w.r.t. $W$ on each interval $[t_k,t_{k+1}]$. We denote by $\mathbb{D}^{1,2}$ the space of random variables differentiable w.r.t. $W$ in the sense of Malliavin, and by $\textnormal{Dom}\ \delta$ the domain of $\delta$. The Malliavin calculus adapted to our model is discussed e.g. in \cite{DJ06, P08}. See Nualart \cite{N} for a detailed exposition of the classical Malliavin calculus and the notations we use in this paper. For any $k \in \{0,...,n-1\}$, $(Y_t^{\theta,\sigma,\lambda}(t_k,x), t\in [t_k,t_{k+1}])$ is differentiable w.r.t. $x$, $\theta$, $\sigma$, and we denote $(\partial_{x}Y_t^{\theta,\sigma,\lambda}(t_k,x), t\in [t_k,t_{k+1}])$, $(\partial_{\theta}Y_t^{\theta,\sigma,\lambda}(t_k,x), t\in [t_k,t_{k+1}])$ and $(\partial_{\sigma}Y_t^{\theta,\sigma,\lambda}(t_k,x), t\in [t_k,t_{k+1}])$, respectively. These processes are the solutions to the linear equations
\begin{align}
&\partial_xY_t^{\theta,\sigma,\lambda}(t_k,x)=1-\theta\int_{t_k}^{t}\partial_xY_u^{\theta,\sigma,\lambda}(t_k,x)du,\label{px}\\
&\partial_{\theta}Y_t^{\theta,\sigma,\lambda}(t_k,x)=-\int_{t_k}^{t}\left(Y_u^{\theta,\sigma,\lambda}(t_k,x)+\theta\partial_{\theta}Y_u^{\theta,\sigma,\lambda}(t_k,x)\right)du,\label{pt}\\
&\partial_{\sigma}Y_t^{\theta,\sigma,\lambda}(t_k,x)=-\theta\int_{t_k}^{t}\partial_{\sigma}Y_u^{\theta,\sigma,\lambda}(t_k,x)du+W_t-W_{t_k}.\label{ps}
\end{align}
For any $t\in [t_k,t_{k+1}]$, the random variables $Y_t^{\theta,\sigma,\lambda}(t_k,x)$, $\partial_xY_t^{\theta,\sigma,\lambda}(t_k,x)$, $(\partial_xY_t^{\theta,\sigma,\lambda}(t_k,x))^{-1}$, $\partial_{\theta}Y_t^{\theta,\sigma,\lambda}(t_k,x)$ and $\partial_{\sigma}Y_t^{\theta,\sigma,\lambda}(t_k,x)$ belong to $\mathbb{D}^{1,2}$. Furthermore, by \cite[Proposition 7]{P08}, the Malliavin derivative $D_sY_t^{\theta,\sigma,\lambda}(t_k,x)$ is given by
$$
D_sY_t^{\theta,\sigma,\lambda}(t_k,x)=\sigma\partial_xY_t^{\theta,\sigma,\lambda}(t_k,x)(\partial_xY_s^{\theta,\sigma,\lambda}(t_k,x))^{-1}{\bf 1}_{[t_k,t]}(s).
$$
	
For all $A\in \widetilde{\mathcal{F}}$ and $x\in\R$, we set $\widetilde{\P}_x^{\theta,\sigma,\lambda}(A)=\widetilde{\E}[{\bf 1}_{A}\vert Y_{t_{k}}^{\theta,\sigma,\lambda}=x]$. We denote by $\widetilde{\E}_x^{\theta,\sigma,\lambda}$ the expectation w.r.t. $\widetilde{\P}_x^{\theta,\sigma,\lambda}$. That is, for all $\widetilde{\mathcal{F}}$-measurable random variable $V$, we have that $\widetilde{\E}_x^{\theta,\sigma,\lambda}[V]=\widetilde{\E}[V\vert Y_{t_{k}}^{\theta,\sigma,\lambda}=x]$. 
	
For any $t>s$, the law of $Y_t^{\theta,\sigma,\lambda}$ conditioned on $Y_s^{\theta,\sigma,\lambda}=x$ admits a positive transition density $p^{\theta,\sigma,\lambda}(t-s,x,y)$, which is differentiable w.r.t. $\theta$, $\sigma$, and $\lambda$. As a consequence of \cite[Proposition 4.1]{G01}, we have the following explicit expression for the logarithm derivatives of the transition density w.r.t. $\theta$ and $\sigma$ in terms of a conditional expectation.
\begin{proposition} \label{c2prop1}
For all $(\theta, \sigma, \lambda)\in\Theta\times \Sigma \times\Lambda$, $k \in \{0,...,n-1\}$, $\beta\in\{\theta, \sigma\}$, and $x, y\in\R$,
\begin{align*}
\dfrac{\partial_{\beta}p^{\theta,\sigma,\lambda}}{p^{\theta,\sigma,\lambda}}\left(\Delta_n,x,y\right)=\dfrac{1}{\Delta_n}\widetilde{\E}_x^{\theta,\sigma,\lambda}\left[\delta\left(\partial_{\beta}Y_{t_{k+1}}^{\theta,\sigma,\lambda}(t_k,x)U^{\theta,\sigma,\lambda}(t_k,x)\right)\big\vert Y_{t_{k+1}}^{\theta,\sigma,\lambda}=y\right],
\end{align*}
where $U^{\theta,\sigma,\lambda}(t_k,x):=(U^{\theta,\sigma,\lambda}_t(t_k,x), t\in[t_k,t_{k+1}])$ with $U^{\theta,\sigma,\lambda}_t(t_k,x):=(D_tY_{t_{k+1}}^{\theta,\sigma,\lambda}(t_k,x))^{-1}$.
\end{proposition}
\begin{proof}
Let $f$ be a continuously differentiable function with compact support. The chain rule of the Malliavin calculus implies that $f'(Y_{t_{k+1}}^{\theta,\sigma,\lambda}(t_k,x))=D_t(f(Y_{t_{k+1}}^{\theta,\sigma,\lambda}(t_k,x)))U^{\theta,\sigma,\lambda}_t(t_k,x)$, for all $(\theta,\sigma,\lambda)\in \Theta\times \Sigma \times\Lambda$ and $t\in [t_k,t_{k+1}]$, where $U^{\theta,\sigma,\lambda}_t(t_k,x):=(D_tY_{t_{k+1}}^{\theta,\sigma,\lambda}(t_k,x))^{-1}$. Moreover, it is easy to check that $\partial_{\beta}Y_{t_{k+1}}^{\theta,\sigma,\lambda}(t_k,x)U^{\theta,\sigma,\lambda}(t_k,x)\in\mathbb{D}^{1,2}(H)$, where $H=L^2([t_k,t_{k+1}],\R)$. By \cite[Proposition 1.3.1]{N}, we have that $\partial_{\beta}Y_{t_{k+1}}^{\theta,\sigma,\lambda}(t_k,x)U^{\theta,\sigma,\lambda}(t_k,x)\in\textnormal{Dom}\ \delta$. Then, using the Malliavin calculus integration by parts formula on $[t_k,t_{k+1}]$, we get that
\begin{align*}
\partial_{\beta}\widetilde{\E}\left[f(Y_{t_{k+1}}^{\theta,\sigma,\lambda}(t_k,x))\right]&=\widetilde{\E}\left[f'(Y_{t_{k+1}}^{\theta,\sigma,\lambda}(t_k,x))\partial_{\beta}Y_{t_{k+1}}^{\theta,\sigma,\lambda}(t_k,x)\right]\\
&=\dfrac{1}{\Delta_n}\widetilde{\E}\left[\int_{t_k}^{t_{k+1}}D_t(f(Y_{t_{k+1}}^{\theta,\sigma,\lambda}(t_k,x)))U^{\theta,\sigma,\lambda}_t(t_k,x)\partial_{\beta}Y_{t_{k+1}}^{\theta,\sigma,\lambda}(t_k,x)dt\right]\\
&=\dfrac{1}{\Delta_n}\widetilde{\E}\left[f(Y_{t_{k+1}}^{\theta,\sigma,\lambda}(t_k,x))\delta\left(\partial_{\beta}Y_{t_{k+1}}^{\theta,\sigma,\lambda}(t_k,x)U^{\theta,\sigma,\lambda}(t_k,x)\right)\right].
\end{align*}
Note that $\delta(V)\equiv\delta(V{\bf 1}_{[t_k, t_{k+1}]}(\cdot))$ for any $V\in \textnormal{Dom}\ \delta$. Using the fact that $p^{\theta,\sigma,\lambda}(\Delta_n,x,y)$ and $\partial_{\beta}p^{\theta,\sigma,\lambda}(\Delta_n,x,y)$ are continuous for $(y,\beta)$, the stochastic flow property $Y_t^{\theta,\sigma,\lambda}=Y_t^{\theta,\sigma,\lambda}(s,Y_s^{\theta,\sigma,\lambda})$ for all $0\leq s\leq t$, and the Markov property of diffusion processes, we have that
\begin{align*}
\partial_{\beta}\widetilde{\E}\left[f(Y_{t_{k+1}}^{\theta,\sigma,\lambda}(t_k,x))\right]=\int_{\R}f(y)\partial_{\beta}p^{\theta,\sigma,\lambda}(\Delta_n,x,y)dy,
\end{align*}
and
\begin{align*}
&\widetilde{\E}\left[f(Y_{t_{k+1}}^{\theta,\sigma,\lambda}(t_k,x))\delta\left(\partial_{\beta}Y_{t_{k+1}}^{\theta,\sigma,\lambda}(t_k,x)U^{\theta,\sigma,\lambda}(t_k,x)\right)\right]\\
&=\widetilde{\E}\left[f(Y_{t_{k+1}}^{\theta,\sigma,\lambda})\delta\left(\partial_{\beta}Y_{t_{k+1}}^{\theta,\sigma,\lambda}(t_k,x)U^{\theta,\sigma,\lambda}(t_k,x)\right)\big\vert Y_{t_{k}}^{\theta,\sigma,\lambda}=x\right]\\
&=\int_{\R}f(y)\widetilde{\E}\left[\delta\left(\partial_{\beta}Y_{t_{k+1}}^{\theta,\sigma,\lambda}(t_k,x)U^{\theta,\sigma,\lambda}(t_k,x)\right)\big\vert Y_{t_{k+1}}^{\theta,\sigma,\lambda}=y,Y_{t_{k}}^{\theta,\sigma,\lambda}=x\right]p^{\theta,\sigma,\lambda}(\Delta_n,x,y)dy,
\end{align*}
which gives the desired result.
\end{proof}
	
For all $k \in \{0,...,n-1\}$, we set $\Delta V_{t_{k+1}}:=V_{t_{k+1}}-V_{t_k}$, for $V\in \{X^{\theta_0,\sigma_0,\lambda_0},B,N,W,M\}$. We next recall Girsanov's theorem on each interval $[t_k, t_{k+1}]$.
\begin{lemma} \label{c2Girsanov2}
For all $\theta, \theta_1, \lambda, \lambda_1, \sigma\in\R^{\ast}_+$ and $k \in \{0,...,n-1\}$, define the measure
\begin{align*}
&\widehat{Q}_k^{\theta_1,\lambda_1,\theta,\lambda,\sigma}(A):=\widehat{\E}\left[{\bf 1}_{A}e^{\int_{t_k}^{t_{k+1}}\frac{(\theta-\theta_1)X_s+\lambda-\lambda_1}{\sigma}dB_s+\frac{1}{2}\int_{t_k}^{t_{k+1}}\left(\frac{(\theta-\theta_1)X_s+\lambda-\lambda_1}{\sigma}\right)^2ds-\Delta N_{t_{k+1}}\log\frac{\lambda}{\lambda_1}+\left(\lambda-\lambda_1\right)\Delta_n}\right],
\end{align*}
for all $A\in \widehat{\mathcal{F}}$. Then $\widehat{Q}_k^{\theta_1,\lambda_1,\theta,\lambda,\sigma}$ is a probability measure and under $\widehat{Q}_k^{\theta_1,\lambda_1,\theta,\lambda,\sigma}$, the process $(B_t^{\widehat{Q}_k^{\theta_1,\lambda_1,\theta,\lambda,\sigma}}:=B_t-\int_{t_k}^{t}\frac{(\theta-\theta_1)X_s+\lambda-\lambda_1}{\sigma}ds, t\in[t_k,t_{k+1}])$ is a Brownian motion, and $(N_t, t\in[t_k,t_{k+1}])$ is a Poisson process with intensity $\lambda_1$. Moreover, all these statements remain valid for the measure $\widetilde{Q}_k^{\theta_1,\lambda_1,\theta,\lambda,\sigma}$ defined by, for all $A\in \widetilde{\mathcal{F}}$,
\begin{align*}
&\widetilde{Q}_k^{\theta_1,\lambda_1,\theta,\lambda,\sigma}(A):=\widetilde{\E}\left[{\bf 1}_{A}e^{\int_{t_k}^{t_{k+1}}\frac{(\theta-\theta_1)Y_s+\lambda-\lambda_1}{\sigma}dW_s+\frac{1}{2}\int_{t_k}^{t_{k+1}}\left(\frac{(\theta-\theta_1)Y_s+\lambda-\lambda_1}{\sigma}\right)^2ds-\Delta M_{t_{k+1}}\log\frac{\lambda}{\lambda_1}+\left(\lambda-\lambda_1\right)\Delta_n}\right].
\end{align*}
\end{lemma}
	
As a consequence, we derive the following explicit expression for the logarithm derivative of the transition density w.r.t. $\lambda$ in the form of conditional expectation. 
\begin{proposition}\label{c2pro2} For all $(\theta, \sigma, \lambda)\in\Theta\times \Sigma \times\Lambda$, $k \in \{0,...,n-1\}$, and $x, y\in\R$,
\begin{align*}
\dfrac{\partial_{\lambda}p^{\theta,\sigma,\lambda}}{p^{\theta,\sigma,\lambda}}(\Delta_n,x,y)=\widetilde{\E}_x^{\theta,\sigma,\lambda}\left[-\dfrac{\Delta W_{t_{k+1}}}{\sigma}+\dfrac{\widetilde{M}_{t_{k+1}}^{\lambda}-\widetilde{M}_{t_k}^{\lambda}}{\lambda}\big\vert Y_{t_{k+1}}^{\theta,\sigma,\lambda}=y\right].
\end{align*}
\end{proposition}
\begin{proof} Let $f$ be a continuously differentiable bounded function. Girsanov's theorem yields
\begin{equation*}\begin{split}
\widetilde{\E}\left[f(Y_{t_{k+1}}^{\theta,\sigma,\lambda}(t_k,x))\right]=\widetilde{\E}\left[f(Y_{t_{k+1}}^{\theta,\sigma,\lambda_1}(t_k,x))\dfrac{d\widetilde{\P}}{d\widetilde{Q}_k^{\theta,\lambda_1,\theta,\lambda,\sigma}}\right].
\end{split}
\end{equation*}
Taking the derivative w.r.t. $\lambda$ in both hand sides of this equality and using Lemma \ref{c2Girsanov2}, 
\begin{align*}
&\partial_{\lambda}\widetilde{\E}\left[f(Y_{t_{k+1}}^{\theta,\sigma,\lambda}(t_k,x))\right]=\widetilde{\E}\left[f(Y_{t_{k+1}}^{\theta,\sigma,\lambda_1}(t_k,x))\partial_{\lambda}\left(\dfrac{d\widetilde{\P}}{d \widetilde{Q}_k^{\theta,\lambda_1,\theta,\lambda,\sigma}}\right)\right]\\
&=\widetilde{\E}\left[f(Y_{t_{k+1}}^{\theta,\sigma,\lambda_1}(t_k,x))\left(-\dfrac{\Delta W_{t_{k+1}}}{\sigma}-\dfrac{\lambda-\lambda_1}{\sigma^2}\Delta_n+\dfrac{\widetilde{M}_{t_{k+1}}^{\lambda}-\widetilde{M}_{t_k}^{\lambda}}{\lambda}\right)\dfrac{d\widetilde{\P}}{d \widetilde{Q}_k^{\theta,\lambda_1,\theta,\lambda,\sigma}}\right]\\
&=\widetilde{\E}\left[f(Y_{t_{k+1}}^{\theta,\sigma,\lambda}(t_k,x))\left(-\dfrac{\Delta W_{t_{k+1}}}{\sigma}+\dfrac{\widetilde{M}_{t_{k+1}}^{\lambda}-\widetilde{M}_{t_k}^{\lambda}}{\lambda}\right)\right].
\end{align*}
Then proceeding similarly as in the proof of Proposition \ref{c2prop1}, the desired result follows.
\end{proof}
	
From the decomposition \ref{eq:dec}, the Markov property and Propositions \ref{c2prop1}, \ref{c2pro2}, the log-likelihood ratio will be represented as sums of conditional expectations involving Skorohod integrals. Therefore, in order to control the convergence of the log-likelihood ratio, one needs to have an appropriate stochastic expansion of the Skorohod integrals $\delta(\partial_{\beta}Y_{t_{k+1}}^{\theta,\sigma,\lambda}(t_k,x)U^{\theta,\sigma,\lambda}(t_k,x))$ appearing in Proposition \ref{c2prop1}, with $\beta\in\{\theta, \sigma\}$, and of the conditional expectations appearing in Proposition \ref{c2pro2}. In fact, these random variables can be decomposed in terms of some function $g(\theta,\sigma,\lambda,\Delta_n,Y_{t_k}^{\theta,\sigma,\lambda},Y_{t_{k+1}}^{\theta,\sigma,\lambda})$ and remainder terms. Then $g(\theta,\sigma,\lambda,\Delta_n,Y_{t_k}^{\theta,\sigma,\lambda},Y_{t_{k+1}}^{\theta,\sigma,\lambda})$ whose conditional expectation can be computed easily corresponds to the main terms that will contribute to the limit of the log-likelihood ratio. Some remainder terms which have to be centered random variables (see \eqref{es1} and \eqref{es4}) will have no contribution in the limit of the log-likelihood ratio. These facts will be seen in Subsections \ref{expand}-\ref{maincontri}. Consequently, we have the following decompositions of the aforementioned Skorohod integrals.
\begin{lemma}\label{delta} For all $(\theta,\sigma,\lambda) \in \Theta\times\Sigma\times\Lambda$, $k \in \{0,...,n-1\}$, and $x\in\R$,
\begin{align*}
\delta\left(\partial_{\theta}Y_{t_{k+1}}^{\theta,\sigma,\lambda}(t_k,x)U^{\theta,\sigma,\lambda}(t_k,x)\right)&=-\Delta_n\sigma^{-2}x\left(Y_{t_{k+1}}^{\theta,\sigma,\lambda}-Y_{t_{k}}^{\theta,\sigma,\lambda}+\Delta_n\theta Y_{t_k}^{\theta,\sigma,\lambda}\right)\\
&\qquad+R_1^{\theta,\sigma,\lambda}+R_2^{\theta,\sigma,\lambda}+R_3^{\theta,\sigma,\lambda}-R_4^{\theta,\sigma,\lambda}-R_5^{\theta,\sigma,\lambda},
\end{align*}
where 
\begin{align*}
R_1^{\theta,\sigma,\lambda}&=\frac{1}{\sigma}\int_{t_k}^{t_{k+1}}\int_{s}^{t_{k+1}}D_s\left(\dfrac{Y_{u}^{\theta,\sigma,\lambda}(t_k,x)}{\partial_xY_{u}^{\theta,\sigma,\lambda}(t_k,x)}\right)\partial_xY_{s}^{\theta,\sigma,\lambda}(t_k,x)duds,\\
R_2^{\theta,\sigma,\lambda}&=-\frac{1}{\sigma}\int_{t_k}^{t_{k+1}}\dfrac{Y_{s}^{\theta,\sigma,\lambda}(t_k,x)}{\partial_xY_{s}^{\theta,\sigma,\lambda}(t_k,x)}ds\int_{t_k}^{t_{k+1}}\left(\partial_xY_{s}^{\theta,\sigma,\lambda}(t_k,x)-\partial_xY_{t_k}^{\theta,\sigma,\lambda}(t_k,x)\right)dW_s,\\
R_3^{\theta,\sigma,\lambda}&=-\frac{1}{\sigma}\int_{t_k}^{t_{k+1}}\left(\dfrac{Y_{s}^{\theta,\sigma,\lambda}(t_k,x)}{\partial_xY_{s}^{\theta,\sigma,\lambda}(t_k,x)}-\dfrac{Y_{t_k}^{\theta,\sigma,\lambda}(t_k,x)}{\partial_xY_{t_k}^{\theta,\sigma,\lambda}(t_k,x)}\right)ds \int_{t_k}^{t_{k+1}}\partial_xY_{t_k}^{\theta,\sigma,\lambda}(t_k,x)dW_s,\\
R_4^{\theta,\sigma,\lambda}&=\Delta_n\sigma^{-2}\theta x\int_{t_k}^{t_{k+1}}\left(Y_{s}^{\theta,\sigma,\lambda}-Y_{t_k}^{\theta,\sigma,\lambda}\right)ds,\ R_5^{\theta,\sigma,\lambda}=-\Delta_n\sigma^{-2}x\left(\widetilde{M}_{t_{k+1}}^{\lambda}-\widetilde{M}_{t_{k}}^{\lambda}\right).
\end{align*}
\end{lemma}
\begin{proof}
From \eqref{px} and It\^o's formula, 
\begin{equation}\label{ipx}\begin{split}
(\partial_xY_{t}^{\theta,\sigma,\lambda}(t_k,x))^{-1}&=1+\theta\int_{t_k}^t(\partial_xY_s^{\theta,\sigma,\lambda}(t_k,x))^{-1}ds,
\end{split}
\end{equation}
which, together with \eqref{pt} and It\^o's formula again, implies that
$$
\dfrac{\partial_{\theta}Y_{t_{k+1}}^{\theta,\sigma,\lambda}(t_k,x)}{\partial_xY_{t_{k+1}}^{\theta,\sigma,\lambda}(t_k,x)}=-\int_{t_k}^{t_{k+1}}\dfrac{Y_{s}^{\theta,\sigma,\lambda}(t_k,x)}{\partial_xY_{s}^{\theta,\sigma,\lambda}(t_k,x)}ds.
$$ 
Then, using  $U^{\theta,\sigma,\lambda}_t(t_k,x)=(D_tY_{t_{k+1}}^{\theta,\sigma,\lambda}(t_k,x))^{-1}=\sigma^{-1}(\partial_xY_{t_{k+1}}^{\theta,\sigma,\lambda}(t_k,x))^{-1}\partial_xY_t^{\theta,\sigma,\lambda}(t_k,x)$ and the product rule \cite[(1.48)]{N}, we obtain that
\begin{align*}
&\delta\left(\partial_{\theta}Y_{t_{k+1}}^{\theta,\sigma,\lambda}(t_k,x)U^{\theta,\sigma,\lambda}(t_k,x)\right)
=-\frac{1}{\sigma}\int_{t_k}^{t_{k+1}}\frac{Y_{s}^{\theta,\sigma,\lambda}(t_k,x)}{\partial_xY_{s}^{\theta,\sigma,\lambda}(t_k,x)}ds \int_{t_k}^{t_{k+1}}\partial_xY_{s}^{\theta,\sigma,\lambda}(t_k,x)dW_s\\
&\qquad+\frac{1}{\sigma}\int_{t_k}^{t_{k+1}}\int_{s}^{t_{k+1}}D_s\left(\dfrac{Y_{u}^{\theta,\sigma,\lambda}(t_k,x)}{\partial_xY_{u}^{\theta,\sigma,\lambda}(t_k,x)}\right)\partial_xY_{s}^{\theta,\sigma,\lambda}(t_k,x)duds.
\end{align*}
		
We next add and subtract the term $\partial_xY_{t_k}^{\theta,\sigma,\lambda}(t_k,x)$ in the second integral, and the term $\frac{Y_{t_k}^{\theta,\sigma,\lambda}(t_k,x)}{\partial_xY_{t_k}^{\theta,\sigma,\lambda}(t_k,x)}$ in the first integral. This, together with  $Y_{t_k}^{\theta,\sigma,\lambda}(t_k,x)=x$, yields 
\begin{equation}\label{e0}
\delta\left(\partial_{\theta}Y_{t_{k+1}}^{\theta,\sigma,\lambda}(t_k,x)U^{\theta,\sigma,\lambda}(t_k,x)\right)
=-\Delta_n\sigma^{-1}x\Delta W_{t_{k+1}}+R_1^{\theta,\sigma,\lambda}+R_2^{\theta,\sigma,\lambda}+R_3^{\theta,\sigma,\lambda}.
\end{equation}
		
On the other hand, by equation \eqref{c2eq1rajoute} we have that
\begin{equation}\label{incrementw}
\sigma\Delta W_{t_{k+1}}=Y_{t_{k+1}}^{\theta,\sigma,\lambda}-Y_{t_{k}}^{\theta,\sigma,\lambda}+\Delta_n\theta Y_{t_k}^{\theta,\sigma,\lambda}+\theta\int_{t_k}^{t_{k+1}}\left(Y_{s}^{\theta,\sigma,\lambda}-Y_{t_k}^{\theta,\sigma,\lambda}\right)ds-\left(\widetilde{M}_{t_{k+1}}^{\lambda}-\widetilde{M}_{t_{k}}^{\lambda}\right).
\end{equation}
This, together with \eqref{e0}, gives the desired result.
\end{proof}
	
\begin{lemma}\label{delta2} For all $(\theta,\sigma,\lambda) \in \Theta\times\Sigma\times\Lambda$, $k \in \{0,...,n-1\}$, and $x\in\R$,
\begin{align*}
&\delta\left(\partial_{\sigma}Y_{t_{k+1}}^{\theta,\sigma,\lambda}(t_k,x)U^{\theta,\sigma,\lambda}(t_k,x)\right)=\dfrac{1}{\sigma^3}\left(Y_{t_{k+1}}^{\theta,\sigma,\lambda}-Y_{t_{k}}^{\theta,\sigma,\lambda}\right)^2-\dfrac{\Delta_n}{\sigma}+H_1^{\theta,\sigma,\lambda}+H_2^{\theta,\sigma,\lambda}+H_3^{\theta,\sigma,\lambda}\\
&\qquad-\dfrac{1}{\sigma^3}\left\{\left(H_4^{\theta,\sigma,\lambda}+H_{5}^{\theta,\sigma,\lambda}\right)^2+2\sigma\Delta W_{t_{k+1}}\left(H_4^{\theta,\sigma,\lambda}+H_{5}^{\theta,\sigma,\lambda}\right)\right\},
\end{align*}
where 
\begin{align*}
H_1^{\theta,\sigma,\lambda}&=-\frac{1}{\sigma}\int_{t_k}^{t_{k+1}}\int_{s}^{t_{k+1}}D_s\left(\dfrac{1}{\partial_xY_u^{\theta,\sigma,\lambda}(t_k,x)}\right)dW_u\partial_xY_s^{\theta,\sigma,\lambda}(t_k,x)ds,\\
H_2^{\theta,\sigma,\lambda}&=\frac{1}{\sigma}\int_{t_k}^{t_{k+1}}\left(\dfrac{1}{\partial_xY_s^{\theta,\sigma,\lambda}(t_k,x)}-\dfrac{1}{\partial_xY_{t_{k}}^{\theta,\sigma,\lambda}(t_k,x)}\right)dW_s\int_{t_k}^{t_{k+1}}\partial_x Y_s^{\theta,\sigma,\lambda}(t_k,x)dW_s,\\
H_3^{\theta,\sigma,\lambda}&=\frac{1}{\sigma}\int_{t_k}^{t_{k+1}}\dfrac{1}{\partial_xY_{t_{k}}^{\theta,\sigma,\lambda}(t_k,x)}dW_s\int_{t_k}^{t_{k+1}}\left(\partial_x Y_s^{\theta,\sigma,\lambda}(t_k,x)-\partial_x Y_{t_k}^{\theta,\sigma,\lambda}(t_k,x)\right)dW_s,\\
H_{4}^{\theta,\sigma,\lambda}&=-\theta\int_{t_k}^{t_{k+1}}Y_{s}^{\theta,\sigma,\lambda}ds,\ \ H_{5}^{\theta,\sigma,\lambda}=\widetilde{M}_{t_{k+1}}^{\lambda}-\widetilde{M}_{t_{k}}^{\lambda}.
\end{align*}
\end{lemma}
\begin{proof}
From \eqref{ps}, \eqref{ipx} and It\^o's formula, 
\begin{align*}
\dfrac{\partial_{\sigma}Y_{t_{k+1}}^{\theta,\sigma,\lambda}(t_k,x)}{\partial_xY_{t_{k+1}}^{\theta,\sigma,\lambda}(t_k,x)}=\int_{t_k}^{t_{k+1}}\dfrac{1}{\partial_xY_s^{\theta,\sigma,\lambda}(t_k,x)}dW_s.
\end{align*}
		
Then, using  $U^{\theta,\sigma,\lambda}_t(t_k,x)=\sigma^{-1}(\partial_xY_{t_{k+1}}^{\theta,\sigma,\lambda}(t_k,x))^{-1}\partial_xY_t^{\theta,\sigma,\lambda}(t_k,x)$ and the product rule \cite[(1.48)]{N}, we obtain that
\begin{align*}
&\delta\left(\partial_{\sigma}Y_{t_{k+1}}^{\theta,\sigma,\lambda}(t_k,x)U^{\theta,\sigma,\lambda}(t_k,x)\right)=\frac{1}{\sigma}\int_{t_k}^{t_{k+1}}\dfrac{1}{\partial_xY_s^{\theta,\sigma,\lambda}(t_k,x)}dW_s\int_{t_k}^{t_{k+1}}\partial_x Y_s^{\theta,\sigma,\lambda}(t_k,x)dW_s\\
&-\frac{1}{\sigma}\int_{t_k}^{t_{k+1}}\left\{\dfrac{1}{\partial_xY_s^{\theta,\sigma,\lambda}(t_k,x)}+\int_{s}^{t_{k+1}}D_s\left(\dfrac{1}{\partial_xY_u^{\theta,\sigma,\lambda}(t_k,x)}\right)dW_u\right\} \partial_xY_s^{\theta,\sigma,\lambda}(t_k,x)ds.
\end{align*}
We next add and subtract the term $
\frac{1}{\partial_xY_{t_k}^{\theta,\sigma,\lambda}(t_k,x)}$ in the first integral, and the term $\partial_xY_{t_k}^{\theta,\sigma,\lambda}(t_k,x)$ in the second integral. This yields 
\begin{equation}\label{H} \begin{split}
\delta\left(\partial_{\sigma}Y_{t_{k+1}}^{\theta,\sigma,\lambda}(t_k,x)U^{\theta,\sigma,\lambda}(t_k,x)\right)&=\dfrac{1}{\sigma}(\Delta W_{t_{k+1}})^2-\dfrac{\Delta_n}{\sigma}+H_1^{\theta,\sigma,\lambda}+H_2^{\theta,\sigma,\lambda}+H_3^{\theta,\sigma,\lambda}.
\end{split}
\end{equation}
On the other hand, by equation \eqref{c2eq1rajoute} we have that
\begin{align*}
\Delta W_{t_{k+1}}=\sigma^{-1}\left(Y_{t_{k+1}}^{\theta,\sigma,\lambda}-Y_{t_{k}}^{\theta,\sigma,\lambda}+\theta\int_{t_k}^{t_{k+1}}Y_{s}^{\theta,\sigma,\lambda}ds-\left(\widetilde{M}_{t_{k+1}}^{\lambda}-\widetilde{M}_{t_{k}}^{\lambda}\right)\right),
\end{align*}
which, together with \eqref{H}, gives the desired result.
\end{proof}
	
We will use the following estimates for the solution to \eqref{flow}.
\begin{lemma}\label{moment3}  
\begin{itemize}\item[\textnormal{(i)}] For any $p\geq 1$ and $(\theta,\sigma,\lambda) \in \Theta\times\Sigma\times\Lambda$, there exists a constant $C_p>0$ such that for all $k \in \{0,...,n-1\}$ and $t\in[t_k,t_{k+1}]$, 
\begin{equation*}
\E\left[\left\vert Y_t^{\theta,\sigma,\lambda}(t_k,x)-Y_{t_k}^{\theta,\sigma,\lambda}(t_k,x)\right\vert^p\big\vert Y_{t_{k}}^{\theta,\sigma,\lambda}(t_k,x)=x\right] \leq C_p\left\vert t-t_k\right\vert^{\frac{p}{2}\wedge 1}\left(1+\vert x\vert^p\right).
\end{equation*}
			
\item[\textnormal{(ii)}] For any function $g$ defined on $\Theta\times\Sigma\times\Lambda\times\R$ with polynomial growth in $x$ uniformly in $(\theta,\sigma,\lambda)$, there exist constants $C, q>0$ such that for $k\in\{0,...,n-1\}$ and $t\in[t_k,t_{k+1}]$, 
$$
\E\left[\left\vert g(\theta,\sigma,\lambda,Y_t^{\theta,\sigma,\lambda}(t_k,x))\right\vert\big\vert Y_{t_{k}}^{\theta,\sigma,\lambda}(t_k,x)=x\right]\leq C\left(1+\vert x\vert^q\right).
$$
Moreover, all these statements remain valid for $X^{\theta,\sigma,\lambda}$.
\end{itemize}
\end{lemma}
	
Using Gronwall's inequality, one can easily check that for any $(\theta,\sigma,\lambda) \in \Theta\times\Sigma\times\Lambda$ and $p\geq 2$, there exist constant $C_p, q>0$ such that for all $k \in \{0,...,n-1\}$ and $t\in[t_k,t_{k+1}]$, 
\begin{align*}
&\E\left[\left\vert \partial_xY_t^{\theta,\sigma,\lambda}(t_k,x)\right\vert^p+\dfrac{1}{\left\vert \partial_xY_t^{\theta,\sigma,\lambda}(t_k,x)\right\vert^p}+\left\vert \partial_{\sigma}Y_t^{\theta,\sigma,\lambda}(t_k,x)\right\vert^p\big\vert Y_{t_{k}}^{\theta,\sigma,\lambda}(t_k,x)=x\right]\\
&\qquad+\sup_{s\in [t_k,t_{k+1}]}\E\left[\left\vert D_sY_t^{\theta,\sigma,\lambda}(t_k,x) \right\vert^p\big\vert Y_{t_{k}}^{\theta,\sigma,\lambda}(t_k,x)=x\right]\leq C_p,\\
&\E\left[\left\vert \partial_{\theta}Y_t^{\theta,\sigma,\lambda}(t_k,x)\right\vert^p\big\vert Y_{t_{k}}^{\theta,\sigma,\lambda}(t_k,x)=x\right]\\
&\qquad+\sup_{s\in [t_k,t_{k+1}]}\E\left[\left\vert D_s\left(\partial_xY_t^{\theta,\sigma,\lambda}(t_k,x)\right)\right \vert^p\big\vert Y_{t_{k}}^{\theta,\sigma,\lambda}(t_k,x)=x\right]\leq C_p\left(1+\vert x\right\vert^q).
\end{align*}
As a consequence, we have the following estimates, which follow easily from \eqref{e0}, \eqref{H}, Lemma \ref{moment3} and properties of the moments of the Brownian motion.
\begin{lemma} \label{estimate}
For any $(\theta,\sigma,\lambda) \in \Theta\times\Sigma\times\Lambda$ and $p\geq 2$, there exist constants $C_p, q>0$ such that for all $k \in \{0,...,n-1\}$,
\begin{align}
&\E\left[R_1^{\theta,\sigma,\lambda}+R_2^{\theta,\sigma,\lambda}+R_3^{\theta,\sigma,\lambda}\big\vert Y_{t_{k}}^{\theta,\sigma,\lambda}(t_k,x)=x\right]=0,\label{es1}\\
&\E\left[\left\vert R_1^{\theta,\sigma,\lambda}+R_2^{\theta,\sigma,\lambda}+R_3^{\theta,\sigma,\lambda}\right\vert^p\big\vert Y_{t_{k}}^{\theta,\sigma,\lambda}(t_k,x)=x\right]\leq C_p\Delta_n^{\frac{3p+1}{2}}\left(1+\vert x\vert^q\right),\label{es2}\\
&\E\left[\left\vert \delta\left(\partial_{\theta}Y_{t_{k+1}}^{\theta,\sigma,\lambda}(t_k,x)U^{\theta,\sigma,\lambda}(t_k,x)\right)\right\vert^{p}\big\vert Y_{t_{k}}^{\theta,\sigma,\lambda}(t_k,x)=x\right]\leq C_p\Delta_n^{\frac{3p}{2}}\left(1+\vert x\vert^q\right),\label{es3}\\
&\E\left[H_1^{\theta,\sigma,\lambda}+H_2^{\theta,\sigma,\lambda}+H_3^{\theta,\sigma,\lambda}\big\vert Y_{t_{k}}^{\theta,\sigma,\lambda}(t_k,x)=x\right]=0,\label{es4}\\
&\E\left[\left\vert H_1^{\theta,\sigma,\lambda}+H_2^{\theta,\sigma,\lambda}+H_3^{\theta,\sigma,\lambda}\right\vert^p\big\vert Y_{t_{k}}^{\theta,\sigma,\lambda}(t_k,x)=x\right]\leq C_p\Delta_n^{p+\frac{1}{2}}\left(1+\vert x\vert^q\right),\label{es6}\\
&\E\left[\left\vert \delta\left(\partial_{\sigma}Y_{t_{k+1}}^{\theta,\sigma,\lambda}(t_k,x)U^{\theta,\sigma,\lambda}(t_k,x)\right)\right\vert^{p}\big\vert Y_{t_{k}}^{\theta,\sigma,\lambda}(t_k,x)=x\right]\leq C_p\Delta_n^{p}\left(1+\vert x\vert^q\right).\label{es5}
\end{align}
\end{lemma}
	
Next, we recall a discrete ergodic theorem.
\begin{lemma}\textnormal{\cite[Lemma 3.1]{G02}}\label{c3ergodic} Consider a differentiable function $g: \R \to \R$, whose derivatives have polynomial growth in $x$. Then, as $n\to\infty$,
\begin{equation*}
\dfrac{1}{n}\sum_{k=0}^{n-1}g(X_{t_k})\overset{\P^{\ta_0,\sigma_0,\lambda_0}}{\longrightarrow}\int_{\mathbb{R}}g(x)\pi_{\theta_0,\sigma_0,\lambda_0}(dx).
\end{equation*}
\end{lemma}
	
For each $n\in\mathbb{N}^{\ast}$, let $(Z_{k,n})_{k\geq 1}$ be a sequence of random variables defined on the filtered probability space $(\Omega, \mathcal{F}, \{\mathcal{F}_t\}_{t\geq 0}, \P)$, and assume that they are $\mathcal{F}_{t_{k+1}}$-measurable, for all $k$.
\begin{lemma}\label{zero} \textnormal{\cite[Lemma 3.2]{G02}} Assume that as $n  \rightarrow \infty$,  
\begin{equation*} 
\textnormal{(i)}\;  \sum_{k=0}^{n-1}\E\left[Z_{k,n}\vert \mathcal{F}_{t_k}\right] \overset{\P}{\longrightarrow} 0, \quad \text{ and } \quad \textnormal{(ii)} \,  \sum_{k=0}^{n-1}\E\left[Z_{k,n}^2\vert \mathcal{F}_{t_k} \right]\overset{\P}{\longrightarrow} 0.
\end{equation*}
Then as $n  \rightarrow \infty$, 
$
\sum_{k=0}^{n-1}Z_{k,n}\overset{\P}{\longrightarrow} 0.
$
\end{lemma}
	
\begin{lemma}\label{zero2} \textnormal{\cite[Lemma 4.1]{J11}}  Assume that as $n  \rightarrow \infty$, $\sum_{k=0}^{n-1}\E\left[\vert Z_{k,n}\vert\vert \mathcal{F}_{t_k}\right] \overset{\P}{\longrightarrow} 0$. Then as $n  \rightarrow \infty$, 
$\sum_{k=0}^{n-1}Z_{k,n}\overset{\P}{\longrightarrow} 0$.
\end{lemma}
	
\section{Proof of Theorem \ref{c2theorem}}
\label{sec:proof}
	
In this section, the proof of Theorem \ref{c2theorem} will be divided into three steps. We begin deriving a stochastic expansion of the log-likelihood ratio using Propositions \ref{c2prop1}, \ref{c2pro2} and Lemmas \ref{delta}, \ref{delta2}. The second step treats the negligible contributions of the expansion. Finally, we apply the central limit theorem for triangular arrays in order to show the LAN property.
	
\subsection{Expansion of the log-likelihood ratio}
\label{expand}
\begin{lemma}\label{expansion} The log-likelihood ratio at $(\theta_0,\sigma_0,\lambda_0)$ can be expressed as
\begin{align*} 
&\log\frac{p_n\left(X^{n};\left(\theta_n,\sigma_n,\lambda_n\right)\right)}{p_n\left(X^{n};\left(\theta_0,\sigma_0,\lambda_0\right)\right)}=\sum_{k=0}^{n-1}\left(\xi_{k,n}+\eta_{k,n}+\beta_{k,n}\right)+\sum_{k=0}^{n-1}\dfrac{u}{\sqrt{n\Delta_n^3}}\int_0^1 \bigg\{ Z^{4,\ell}_{k,n}+Z^{5,\ell}_{k,n}\\
&+\widetilde{\E}_{X_{t_{k}}}^{\theta(\ell),\sigma_0,\lambda_0}\left[R^{\theta(\ell),\sigma_0,\lambda_0}-R_4^{\theta(\ell),\sigma_0,\lambda_0}-R_5^{\theta(\ell),\sigma_0,\lambda_0}\big\vert Y_{t_{k+1}}^{\theta(\ell),\sigma_0,\lambda_0}=X_{t_{k+1}}\right]\bigg\}d\ell+\sum_{k=0}^{n-1}(T_{k,n}-R_{k,n})\\
&+\sum_{k=0}^{n-1}\dfrac{v}{\sqrt{n\Delta_n^2}}\int_0^1\widetilde{\E}_{X_{t_k}}^{\theta_n,\sigma(\ell),\lambda_n}\left[H^{\theta_n,\sigma(\ell),\lambda_n}\big\vert Y^{\theta_n,\sigma(\ell),\lambda_n}_{t_{k+1}}=X_{t_{k+1}}\right]d\ell+\sum_{k=0}^{n-1}\dfrac{v}{\sqrt{n\Delta_n^2}}\int_0^1\dfrac{1}{\sigma(\ell)^3}\\
&\qquad\times\bigg\{\left(H_{6}+H_{7}\right)^2+2\sigma_0\Delta B_{t_{k+1}}\left(H_{6}+H_{7}\right)-\widetilde{\E}_{X_{t_k}}^{\theta_n,\sigma(\ell),\lambda_n}\bigg[\left(H_4^{\theta_n,\sigma(\ell),\lambda_n}+H_{5}^{\theta_n,\sigma(\ell),\lambda_n}\right)^2\\
&\qquad\qquad+2\sigma(\ell)\Delta W_{t_{k+1}}\left(H_4^{\theta_n,\sigma(\ell),\lambda_n}+H_{5}^{\theta_n,\sigma(\ell),\lambda_n}\right)\big\vert Y^{\theta_n,\sigma(\ell),\lambda_n}_{t_{k+1}}=X_{t_{k+1}}\bigg]\bigg\}d\ell,
\end{align*}
where $\theta(\ell):=\theta_0+\dfrac{\ell u}{\sqrt{n\Delta_n}}$,  $\sigma(\ell):=\sigma_0+\dfrac{\ell v}{\sqrt{n}}$, $\lambda(\ell):=\lambda_0+\dfrac{\ell w}{\sqrt{n\Delta_n}}$, and
\begin{align*}
&\xi_{k,n}=-\dfrac{u}{\sigma_0^{2}\sqrt{n\Delta_n}}X_{t_k}\left(\sigma_0\Delta B_{t_{k+1}}+\dfrac{u\Delta_n}{2\sqrt{n\Delta_n}}X_{t_k}\right),\\
&\eta_{k,n}=\dfrac{v}{\sqrt{n\Delta_n^2}}\int_0^1\left(\dfrac{\sigma_0^2}{\sigma(\ell)^3}(\Delta B_{t_{k+1}})^2-\dfrac{\Delta_n}{\sigma(\ell)}\right)d\ell,\\
&\beta_{k,n}=-\dfrac{w}{\sigma_0^2\sqrt{n\Delta_n}}\left(\sigma_0\Delta B_{t_{k+1}}+\dfrac{w\Delta_n}{2\sqrt{n\Delta_n}}+\dfrac{u\Delta_n}{\sqrt{n\Delta_n}}X_{t_k}\right)\notag\\
&\qquad\qquad+\dfrac{w}{\sqrt{n\Delta_n}}\int_0^1\widetilde{\E}_{X_{t_k}}^{\theta_n,\sigma_0,\lambda(\ell)}\left[\dfrac{\widetilde{M}_{t_{k+1}}^{\lambda(\ell)}-\widetilde{M}_{t_k}^{\lambda(\ell)}}{\lambda(\ell)}\big\vert Y_{t_{k+1}}^{\theta_n,\sigma_0,\lambda(\ell)}=X_{t_{k+1}}\right]d\ell,\\
&Z_{k,n}^{4,\ell}=\Delta_n\sigma_0^{-2}\theta_0X_{t_k}\int_{t_k}^{t_{k+1}}\left(X_{s}^{\theta_0,\sigma_0,\lambda_0}-X_{t_k}\right)ds,\ Z_{k,n}^{5,\ell}=-\Delta_n\sigma_0^{-2}X_{t_k}\left(\widetilde{N}_{t_{k+1}}^{\lambda_0}-\widetilde{N}_{t_{k}}^{\lambda_0}\right),\\
&R^{\theta(\ell),\sigma_0,\lambda_0}=R_1^{\theta(\ell),\sigma_0,\lambda_0}+R_2^{\theta(\ell),\sigma_0,\lambda_0}+R_3^{\theta(\ell),\sigma_0,\lambda_0},\\
&T_{k,n}=\dfrac{w}{\sigma_0^2\sqrt{n\Delta_n}}\int_0^1\bigg\{\theta_0\int_{t_k}^{t_{k+1}}\left(X_s^{\theta_0,\sigma_0,\lambda_0}-X_{t_k}\right)ds\\
&\qquad\qquad-\widetilde{\E}_{X_{t_k}}^{\theta_n,\sigma_0,\lambda(\ell)}\left[\theta_n\int_{t_k}^{t_{k+1}}\left(Y_s^{\theta_n,\sigma_0,\lambda(\ell)}-Y_{t_k}^{\theta_n,\sigma_0,\lambda(\ell)}\right)ds\big\vert Y_{t_{k+1}}^{\theta_n,\sigma_0,\lambda(\ell)}=X_{t_{k+1}}\right]\bigg\}d\ell,\\
&R_{k,n}=\dfrac{w}{\sigma_0^2\sqrt{n\Delta_n}}\int_0^1\left(\Delta N_{t_{k+1}}-\widetilde{\E}_{X_{t_k}}^{\theta_n,\sigma_0,\lambda(\ell)}\left[\Delta  M_{t_{k+1}}\big\vert Y^{\theta_n,\sigma_0,\lambda(\ell)}_{t_{k+1}}=X_{t_{k+1}}\right]\right)d\ell,\\
&H^{\theta_n,\sigma(\ell),\lambda_n}=H_1^{\theta_n,\sigma(\ell),\lambda_n}+H_2^{\theta_n,\sigma(\ell),\lambda_n}+H_3^{\theta_n,\sigma(\ell),\lambda_n},\\
&H_{6}=-\theta_0\int_{t_k}^{t_{k+1}}X_{s}^{\theta_0,\sigma_0,\lambda_0}ds,\ \ H_{7}=\widetilde{N}_{t_{k+1}}^{\lambda_0}-\widetilde{N}_{t_{k}}^{\lambda_0}.
\end{align*}
\end{lemma}
\begin{proof}
Recall that the log-likelihood ratio is decomposed as in \eqref{eq:dec}. First, using the Markov property, Proposition \ref{c2prop1}, Lemma \ref{delta} and equation \eqref{c2eq1} for the term $X_{t_{k+1}}-X_{t_{k}}$ that we obtain from the term $Y_{t_{k+1}}^{\theta(\ell),\sigma_0,\lambda_0}-Y_{t_{k}}^{\theta(\ell),\sigma_0,\lambda_0}$ in Lemma \ref{delta} when taking the conditional expectation, we obtain that
\begin{align*} 
&\log\dfrac{p_n\left(X^{n};\left(\theta_n,\sigma_0,\lambda_0\right)\right)}{p_n\left(X^{n};\left(\theta_0,\sigma_0,\lambda_0\right)\right)}=\sum_{k=0}^{n-1}\log\dfrac{p^{\theta_n,\sigma_0,\lambda_0}}{p^{\theta_0,\sigma_0,\lambda_0}}\left(\Delta_n,X_{t_k},X_{t_{k+1}}\right)\\
&=\sum_{k=0}^{n-1}\dfrac{u}{\sqrt{n\Delta_n}}\int_0^1\dfrac{\partial_{\theta}p^{\theta(\ell),\sigma_0,\lambda_0}}{p^{\theta(\ell),\sigma_0,\lambda_0}}\left(\Delta_n,X_{t_k},X_{t_{k+1}}\right)d\ell \\
&=\sum_{k=0}^{n-1}\dfrac{u}{\sqrt{n\Delta_n^3}}\int_0^1\widetilde{\E}_{X_{t_k}}^{\theta(\ell),\sigma_0,\lambda_0}\big[\delta\big(\partial_{\theta}Y_{t_{k+1}}^{\theta(\ell),\sigma_0,\lambda_0}(t_k,X_{t_{k}})U^{\theta(\ell),\sigma_0,\lambda_0}(t_k,X_{t_{k}})\big)\big\vert Y_{t_{k+1}}^{\theta(\ell),\sigma_0,\lambda_0}=X_{t_{k+1}}\big]d\ell\\
&=\sum_{k=0}^{n-1}\xi_{k,n}+\sum_{k=0}^{n-1}\dfrac{u}{\sqrt{n\Delta_n^3}}\int_0^1 \bigg\{ Z^{4,\ell}_{k,n}+Z^{5,\ell}_{k,n}\\
&\qquad+\widetilde{\E}_{X_{t_{k}}}^{\theta(\ell),\sigma_0,\lambda_0}\left[R^{\theta(\ell),\sigma_0,\lambda_0}-R_4^{\theta(\ell),\sigma_0,\lambda_0}-R_5^{\theta(\ell),\sigma_0,\lambda_0}\big\vert Y_{t_{k+1}}^{\theta(\ell),\sigma_0,\lambda_0}=X_{t_{k+1}}\right]\bigg\}d\ell.
\end{align*}
Next, using the Markov property, Proposition \ref{c2prop1}, Lemma \ref{delta2} and  equation \eqref{c2eq1} for the term $X_{t_{k+1}}-X_{t_{k}}$ that we obtain from the term $Y_{t_{k+1}}^{\theta_n,\sigma(\ell),\lambda_n}-Y_{t_{k}}^{\theta_n,\sigma(\ell),\lambda_n}$ in Lemma \ref{delta2} when taking the conditional expectation, we obtain that
\begin{align*}  
&\log\dfrac{p_n\left(X^{n};\left(\theta_n,\sigma_n,\lambda_n\right)\right)}{p_n\left(X^{n};\left(\theta_n,\sigma_0,\lambda_n\right)\right)}=\sum_{k=0}^{n-1}\dfrac{v}{\sqrt{n}}\int_0^1\dfrac{\partial_{\sigma}p^{\theta_n,\sigma(\ell),\lambda_n}}{p^{\theta_n,\sigma(\ell),\lambda_n}}\left(\Delta_n,X_{t_k},X_{t_{k+1}}\right)d\ell\\
&=\sum_{k=0}^{n-1}\dfrac{v}{\sqrt{n\Delta_n^2}}\int_0^1\widetilde{\E}_{X_{t_k}}^{\theta_n,\sigma(\ell),\lambda_n}\big[\delta\big(\partial_{\sigma}Y_{t_{k+1}}^{\theta_n,\sigma(\ell),\lambda_n}(t_k,X_{t_{k}})U^{\theta_n,\sigma(\ell),\lambda_n}(t_k,X_{t_{k}})\big)\vert Y_{t_{k+1}}^{\theta_n,\sigma(\ell),\lambda_n}=X_{t_{k+1}}\big]d\ell\\
&=\sum_{k=0}^{n-1}\eta_{k,n}+\sum_{k=0}^{n-1}\dfrac{v}{\sqrt{n\Delta_n^2}}\int_0^1\widetilde{\E}_{X_{t_k}}^{\theta_n,\sigma(\ell),\lambda_n}\left[H^{\theta_n,\sigma(\ell),\lambda_n}\big\vert Y^{\theta_n,\sigma(\ell),\lambda_n}_{t_{k+1}}=X_{t_{k+1}}\right]d\ell\\
&\qquad+\sum_{k=0}^{n-1}\dfrac{v}{\sqrt{n\Delta_n^2}}\int_0^1\dfrac{1}{\sigma(\ell)^3}\bigg\{\left(H_{6}+H_{7}\right)^2+2\sigma_0\Delta B_{t_{k+1}}\left(H_{6}+H_{7}\right)\\
&\qquad-\widetilde{\E}_{X_{t_k}}^{\theta_n,\sigma(\ell),\lambda_n}\bigg[\left(H_4^{\theta_n,\sigma(\ell),\lambda_n}+H_{5}^{\theta_n,\sigma(\ell),\lambda_n}\right)^2+2\sigma(\ell)\Delta W_{t_{k+1}}\\
&\qquad\qquad\times\left(H_4^{\theta_n,\sigma(\ell),\lambda_n}+H_{5}^{\theta_n,\sigma(\ell),\lambda_n}\right)\big\vert Y^{\theta_n,\sigma(\ell),\lambda_n}_{t_{k+1}}=X_{t_{k+1}}\bigg]\bigg\}d\ell.
\end{align*}
Finally, using Proposition \ref{c2pro2}, equation \eqref{incrementw} with $(\theta_n,\sigma_0,\lambda(\ell))$ instead of $(\theta,\sigma,\lambda)$, and equation \eqref{c2eq1} for the term $X_{t_{k+1}}-X_{t_{k}}$ that we obtain from the term $Y_{t_{k+1}}^{\theta_n,\sigma_0,\lambda(\ell)}-Y_{t_{k}}^{\theta_n,\sigma_0,\lambda(\ell)}$ in \eqref{incrementw} when taking the conditional expectation, we get that
\begin{align*} 
&\log\dfrac{p_n\left(X^{n};\left(\theta_n,\sigma_0,\lambda_n\right)\right)}{p_n\left(X^{n};\left(\theta_n,\sigma_0,\lambda_0\right)\right)}=\sum_{k=0}^{n-1}\dfrac{w}{\sqrt{n\Delta_n}}\int_0^1\dfrac{\partial_{\lambda}p^{\theta_n,\sigma_0,\lambda(\ell)}}{p^{\theta_n,\sigma_0,\lambda(\ell)}}\left(\Delta_n,X_{t_k},X_{t_{k+1}}\right)d\ell\\
&=\sum_{k=0}^{n-1}\dfrac{w}{\sqrt{n\Delta_n}}\int_0^1\widetilde{\E}_{X_{t_k}}^{\theta_n,\sigma_0,\lambda(\ell)}\left[-\dfrac{\Delta W_{t_{k+1}}}{\sigma_0}+\dfrac{\widetilde{M}_{t_{k+1}}^{\lambda(\ell)}-\widetilde{M}_{t_k}^{\lambda(\ell)}}{\lambda(\ell)}\big\vert Y_{t_{k+1}}^{\theta_n,\sigma_0,\lambda(\ell)}=X_{t_{k+1}}\right]d\ell\\&=\sum_{k=0}^{n-1}\beta_{k,n}+\sum_{k=0}^{n-1}(T_{k,n}-R_{k,n}).
\end{align*}
Therefore, we have shown the desired expansion of the log-likelihood ratio. 
\end{proof}
We show that $\xi_{k,n}, \eta_{k,n}, \beta_{k,n}$ are the terms that contribute to the limit, and all the others are negligible. In all what follows, Lemma \ref{c3ergodic} will be used repeatedly without being quoted.

\subsection{Negligible contributions}
\begin{lemma}\label{lemma1} As $n\to\infty$,
\begin{align*}
\sum_{k=0}^{n-1}\dfrac{u}{\sqrt{n\Delta_n^3}}\int_0^1\widetilde{\E}_{X_{t_{k}}}^{\theta(\ell),\sigma_0,\lambda_0}\left[R^{\theta(\ell),\sigma_0,\lambda_0}\big\vert Y_{t_{k+1}}^{\theta(\ell),\sigma_0,\lambda_0}=X_{t_{k+1}}\right]d\ell\overset{\P^{\theta_0,\sigma_0,\lambda_0}}{\longrightarrow}0.
\end{align*}
\end{lemma}
\begin{proof} 
It suffices to show that conditions (i) and (ii) of Lemma \ref{zero} hold under the measure $\P^{\theta_0,\sigma_0,\lambda_0}$. We start showing (i). Applying Girsanov's theorem and \eqref{es1}, we get that
\begin{align*} 
&\sum_{k=0}^{n-1}\dfrac{u}{\sqrt{n\Delta_n^3}}\int_0^1\E\left[\widetilde{\E}_{X_{t_k}}^{\theta(\ell),\sigma_0,\lambda_0}
\left[R^{\theta(\ell),\sigma_0,\lambda_0}\big\vert Y_{t_{k+1}}^{\theta(\ell),\sigma_0,\lambda_0}=
X_{t_{k+1}}\right]\big\vert \widehat{\mathcal{F}}_{t_k}\right]d\ell\\
&=\sum_{k=0}^{n-1}\dfrac{u}{\sqrt{n\Delta_n^3}}\int_0^1\E_{\widehat{Q}_k^{\theta(\ell),\lambda_0,\theta_0,\lambda_0,\sigma_0}}\left[R^{\theta(\ell),\sigma_0,\lambda_0}\dfrac{d\widehat{\P}}{d \widehat{Q}_k^{\theta(\ell),\lambda_0,\theta_0,\lambda_0,\sigma_0}}\big\vert X_{t_k}\right]d\ell=0,
\end{align*}
where we have used the independence between $R^{\theta(\ell),\sigma_0,\lambda_0}$ and $\frac{d\widehat{\P}}{d \widehat{Q}_k^{\theta(\ell),\lambda_0,\theta_0,\lambda_0,\sigma_0}}$, together with $\E_{\widehat{Q}_k^{\theta(\ell),\lambda_0,\theta_0,\lambda_0,\sigma_0}}[\frac{d\widehat{\P}}{d \widehat{Q}_k^{\theta(\ell),\lambda_0,\theta_0,\lambda_0,\sigma_0}}\vert X_{t_k}]=1$. Thus the term appearing in condition (i) of Lemma \ref{zero} actually equals zero. 

Next, applying Jensen's inequality, Girsanov's theorem, and \eqref{es2} with $p=2$, we obtain 
\begin{align*}
&\sum_{k=0}^{n-1}\dfrac{u^2}{n\Delta_n^3}\E\left[\left(\int_0^1\widetilde{\E}_{X_{t_k}}^{\theta(\ell),\sigma_0,\lambda_0}
\left[R^{\theta(\ell),\sigma_0,\lambda_0}\big\vert Y_{t_{k+1}}^{\theta(\ell),\sigma_0,\lambda_0}=X_{t_{k+1}}\right]d\ell\right)^2\big\vert \widehat{\mathcal{F}}_{t_k}\right]\\
& \leq\sum_{k=0}^{n-1}\dfrac{u^2}{n\Delta_n^3}\int_0^1\E_{\widehat{Q}_k^{\theta(\ell),\lambda_0,\theta_0,\lambda_0,\sigma_0}}
\left[\left(R^{\theta(\ell),\sigma_0,\lambda_0}\right)^2\dfrac{d\widehat{\P}}{d \widehat{Q}_k^{\theta(\ell),\lambda_0,\theta_0,\lambda_0,\sigma_0}}\big\vert X_{t_k}\right]d\ell\\
& =\sum_{k=0}^{n-1}\dfrac{u^2}{n\Delta_n^3}\int_0^1\E_{\widehat{Q}_k^{\theta(\ell),\lambda_0,\theta_0,\lambda_0,\sigma_0}}
\left[\left(R^{\theta(\ell),\sigma_0,\lambda_0}\right)^2\big\vert X_{t_k}\right]d\ell\\
&\leq \dfrac{C u^2\sqrt{\Delta_n}}{n}\sum_{k=0}^{n-1}\left(1+\vert X_{t_k}\vert^{q}\right),
\end{align*}
for some constants $C, q>0$, which gives the desired result.
\end{proof}

\begin{lemma}\label{lemma2} As $n\to\infty$,
\begin{align*}
\sum_{k=0}^{n-1}\dfrac{u}{\sqrt{n\Delta_n^3}}\int_0^1\left(Z^{4,\ell}_{k,n}-\widetilde{\E}_{X_{t_{k}}}^{\theta(\ell),\sigma_0,\lambda_0}\left[R_4^{\theta(\ell),\sigma_0,\lambda_0}\big\vert Y_{t_{k+1}}^{\theta(\ell),\sigma_0,\lambda_0}=X_{t_{k+1}}\right]\right)d\ell\overset{\P^{\theta_0,\sigma_0,\lambda_0}}{\longrightarrow}0.
\end{align*}
\end{lemma}
\begin{proof} We rewrite
\begin{align*}
\dfrac{u}{\sqrt{n\Delta_n^3}}\int_0^1\left(Z^{4,\ell}_{k,n}-\widetilde{\E}_{X_{t_{k}}}^{\theta(\ell),\sigma_0,\lambda_0}\left[R_4^{\theta(\ell),\sigma_0,\lambda_0}\big\vert Y_{t_{k+1}}^{\theta(\ell),\sigma_0,\lambda_0}=X_{t_{k+1}}\right]\right)d\ell=M_{k,n,1}+M_{k,n,2},
\end{align*}
where
\begin{align*}
M_{k,n,1}:&=-\dfrac{u^2}{\sigma_0^2n\Delta_n}X_{t_k}\int_0^1\ell\int_{t_k}^{t_{k+1}}\left(X_s^{\theta_0,\sigma_0,\lambda_0}-X_{t_k}\right)dsd\ell,\\
M_{k,n,2}:&=\dfrac{u}{\sigma_0^2\sqrt{n\Delta_n}}X_{t_k}\int_0^1\theta(\ell)\bigg\{\int_{t_k}^{t_{k+1}}\left(X_s^{\theta_0,\sigma_0,\lambda_0}-X_{t_k}\right)ds\\
&\qquad-\widetilde{\E}_{X_{t_{k}}}^{\theta(\ell),\sigma_0,\lambda_0}\left[\int_{t_k}^{t_{k+1}}\left(Y_s^{\theta(\ell),\sigma_0,\lambda_0}-Y_{t_k}^{\theta(\ell),\sigma_0,\lambda_0}\right)ds\big\vert Y_{t_{k+1}}^{\theta(\ell),\sigma_0,\lambda_0}=X_{t_{k+1}}\right]\bigg\}d\ell.
\end{align*}

First, using Lemma \ref{moment3}(i), we get that
\begin{align*}
\sum_{k=0}^{n-1}\E\left[\left\vert M_{k,n,1}\right\vert\big\vert \widehat{\mathcal{F}}_{t_k}\right]\leq \dfrac{Cu^2\sqrt{\Delta_n}}{n}\sum_{k=0}^{n-1}\left(1+\vert X_{t_k}\vert^{q}\right),
\end{align*}
for some constants $C, q>0$. Therefore, by Lemma \ref{zero2}, $\sum_{k=0}^{n-1}M_{k,n,1}\overset{\P^{\theta_0,\sigma_0,\lambda_0}}{\longrightarrow}0$ as $n\to\infty$. 

Now, we treat $M_{k,n,2}$. Using Girsanov's theorem, we have that
\begin{align*}
&\sum_{k=0}^{n-1}\E\left[M_{k,n,2}\big\vert\widehat{\mathcal{F}}_{t_k}\right]=\dfrac{u}{\sigma_0^2\sqrt{n\Delta_n}}\sum_{k=0}^{n-1}X_{t_k}\int_0^1\theta(\ell)\bigg\{\E\left[\int_{t_k}^{t_{k+1}}\left(X_s^{\theta_0,\sigma_0,\lambda_0}-X_{t_k}\right)ds\big\vert X_{t_k}\right]\\
&\qquad-\E_{\widehat{Q}_k^{\theta(\ell),\lambda_0,\theta_0,\lambda_0,\sigma_0}}\left[\int_{t_k}^{t_{k+1}}\left(Y_s^{\theta(\ell),\sigma_0,\lambda_0}-Y_{t_k}^{\theta(\ell),\sigma_0,\lambda_0}\right)ds\dfrac{d\widehat{\P}}{d \widehat{Q}_k^{\theta(\ell),\lambda_0,\theta_0,\lambda_0,\sigma_0}}\big\vert X_{t_{k}}\right]\bigg\}d\ell\\
&=0,
\end{align*}
where we use the independence between $\int_{t_k}^{t_{k+1}}(Y_s^{\theta(\ell),\sigma_0,\lambda_0}-Y_{t_k}^{\theta(\ell),\sigma_0,\lambda_0})ds$ and $\frac{d\widehat{\P}}{d \widehat{Q}_k^{\theta(\ell),\lambda_0,\theta_0,\lambda_0,\sigma_0}}$, together with $\E_{\widehat{Q}_k^{\theta(\ell),\lambda_0,\theta_0,\lambda_0,\sigma_0}}[\frac{d\widehat{\P}}{d \widehat{Q}_k^{\theta(\ell),\lambda_0,\theta_0,\lambda_0,\sigma_0}}\vert X_{t_k}]=1$.

Next, we proceed as in the proof of Lemma \ref{lemma1} to get that for some constants $C, q>0$,
\begin{align*}
\sum_{k=0}^{n-1}\E\left[M_{k,n,2}^2\big\vert\widehat{\mathcal{F}}_{t_k}\right]\leq\dfrac{Cu^2\Delta_n^2}{n}\sum_{k=0}^{n-1}\left(1+\vert X_{t_k}\vert^{q}\right).
\end{align*}
Therefore, by Lemma \ref{zero}, $\sum_{k=0}^{n-1}M_{k,n,2}\overset{\P^{\theta_0,\sigma_0,\lambda_0}}{\longrightarrow}0$ as $n\to\infty$. Thus, the result follows.
\end{proof}

\begin{lemma}\label{lemma3} As $n\to\infty$,
\begin{align*}
\sum_{k=0}^{n-1}\dfrac{u}{\sqrt{n\Delta_n^3}}\int_0^1\left(Z^{5,\ell}_{k,n}-\widetilde{\E}_{X_{t_{k}}}^{\theta(\ell),\sigma_0,\lambda_0}\left[R_5^{\theta(\ell),\sigma_0,\lambda_0}\big\vert Y_{t_{k+1}}^{\theta(\ell),\sigma_0,\lambda_0}=X_{t_{k+1}}\right]\right)d\ell\overset{\P^{\theta_0,\sigma_0,\lambda_0}}{\longrightarrow}0.
\end{align*}
\end{lemma}
\begin{proof} 
First, applying Girsanov's theorem,
\begin{align*}
&\E\left[Z^{5,\ell}_{k,n}-\widetilde{\E}_{X_{t_{k}}}^{\theta(\ell),\sigma_0,\lambda_0}\left[R_5^{\theta(\ell),\sigma_0,\lambda_0}\big\vert Y_{t_{k+1}}^{\theta(\ell),\sigma_0,\lambda_0}=X_{t_{k+1}}\right]\big\vert\widehat{\mathcal{F}}_{t_k}\right]\\
&=\Delta_n\sigma_0^{-2}X_{t_{k}}\E_{\widehat{Q}_k^{\theta(\ell),\lambda_0,\theta_0,\lambda_0,\sigma_0}}\left[\left(\widetilde{M}_{t_{k+1}}^{\lambda_0}-\widetilde{M}_{t_{k}}^{\lambda_0}\right)\dfrac{d\widehat{\P}}{d \widehat{Q}_k^{\theta(\ell),\lambda_0,\theta_0,\lambda_0,\sigma_0}}\big\vert X_{t_k}\right]=0,
\end{align*}
where we use the independence between $\widetilde{M}_{t_{k+1}}^{\lambda_0}-\widetilde{M}_{t_{k}}^{\lambda_0}$ and $\frac{d\widehat{\P}}{d \widehat{Q}_k^{\theta(\ell),\lambda_0,\theta_0,\lambda_0,\sigma_0}}$. Thus the term appearing in condition (i) of Lemma \ref{zero} actually equals zero. Next, multiplying the random variable inside the expectation by ${\bf 1}_{\widehat{J}_{0,k}} +{\bf 1}_{\widehat{J}_{1,k}}+ {\bf 1}_{\widehat{J}_{\geq 2,k}}$, and applying Lemma \ref{deviation}, we get that for $n$ large enough, for any $\alpha\in(0,\frac{1}{2})$, $p_1>1$, $q_1>1$ with $p_1q_1<2$, and $\mu_1\in(1,2)$,
\begin{align*}
&\dfrac{u^2}{n\Delta_n^3}\sum_{k=0}^{n-1}\E\left[\left(\int_0^1\left(Z^{5,\ell}_{k,n}-\widetilde{\E}_{X_{t_{k}}}^{\theta(\ell),\sigma_0,\lambda_0}\left[R_5^{\theta(\ell),\sigma_0,\lambda_0}\big\vert Y_{t_{k+1}}^{\theta(\ell),\sigma_0,\lambda_0}=X_{t_{k+1}}\right]\right)d\ell\right)^2\big\vert\widehat{\mathcal{F}}_{t_k}\right]\\
&\leq \dfrac{u^2}{\sigma_0^4n\Delta_n}\sum_{k=0}^{n-1}X_{t_k}^2\int_0^1\E\left[\left(\Delta N_{t_{k+1}}-\widetilde{\E}_{X_{t_k}}^{\theta(\ell),\sigma_0,\lambda_0}\left[\Delta M_{t_{k+1}}\big\vert Y^{\theta(\ell),\sigma_0,\lambda_0}_{t_{k+1}}=X_{t_{k+1}}\right]\right)^2\big\vert X_{t_k}\right]d\ell\\
&=\dfrac{u^2}{\sigma_0^4n\Delta_n}\sum_{k=0}^{n-1}X_{t_k}^2\int_0^1\left(M_{0,1}^{\theta(\ell),\sigma_0,\lambda_0}+M_{1,1}^{\theta(\ell),\sigma_0,\lambda_0}+M_{\geq 2,1}^{\theta(\ell),\sigma_0,\lambda_0}\right)d\ell\\
&\leq \left(\Delta_n^{-1}e^{-C_1\Delta_n^{2\alpha-1}}+\Delta_n^{\frac{2}{p_1q_1}-1}+\Delta_n^{\frac{2}{\mu_1}-1}\right)\dfrac{Cu^2}{n}\sum_{k=0}^{n-1}X_{t_k}^2,
\end{align*}
for some constants $C, C_1>0$. This converges to zero in $\P^{\theta_0,\sigma_0,\lambda_0}$-probability as $n\to\infty$. Note that the counting events $\widehat{J}_{0,k}$, $\widehat{J}_{1,k}$, $\widehat{J}_{\geq 2,k}$ and the terms $M_{0,1}^{\theta(\ell),\sigma_0,\lambda_0}$, $M_{1,1}^{\theta(\ell),\sigma_0,\lambda_0}$, $M_{\geq 2,1}^{\theta(\ell),\sigma_0,\lambda_0}$ are defined in Subsections \ref{condensity} and \ref{largeest}. Thus, by Lemma \ref{zero}, the desired result follows.
\end{proof}

\begin{lemma}\label{lemma4} As $n\to\infty$, $\sum_{k=0}^{n-1}T_{k,n}\overset{\P^{\theta_0,\sigma_0,\lambda_0}}{\longrightarrow}0$.
\end{lemma}
\begin{proof} We write $T_{k,n}=T_{k,n,1}+T_{k,n,2}$, where
\begin{align*}
T_{k,n,1}:&=-\dfrac{uw}{\sigma_0^2n\Delta_n}\int_0^1\int_{t_k}^{t_{k+1}}\left(X_s^{\theta_0,\sigma_0,\lambda_0}-X_{t_k}\right)dsd\ell,\\
T_{k,n,2}:&=\dfrac{\theta_nw}{\sigma_0^2\sqrt{n\Delta_n}}\int_0^1\bigg\{\int_{t_k}^{t_{k+1}}\left(X_s^{\theta_0,\sigma_0,\lambda_0}-X_{t_k}\right)ds\\
&\qquad-\widetilde{\E}_{X_{t_k}}^{\theta_n,\sigma_0,\lambda(\ell)}\left[\int_{t_k}^{t_{k+1}}\left(Y_s^{\theta_n,\sigma_0,\lambda(\ell)}-Y_{t_k}^{\theta_n,\sigma_0,\lambda(\ell)}\right)ds\big\vert Y_{t_{k+1}}^{\theta_n,\sigma_0,\lambda(\ell)}=X_{t_{k+1}}\right]\bigg\}d\ell.
\end{align*}

In the same way the terms $M_{k,n,1}$ and $M_{k,n,2}$ of Lemma \ref{lemma2} are treated, we obtain that $\sum_{k=0}^{n-1}T_{k,n,1}\overset{\P^{\theta_0,\sigma_0,\lambda_0}}{\longrightarrow}0$, and $\sum_{k=0}^{n-1}T_{k,n,2}\overset{\P^{\theta_0,\sigma_0,\lambda_0}}{\longrightarrow}0$ as $n\to\infty$. Thus, the result follows.
\end{proof}

In all what follows, we set $U_k:=\Delta N_{t_{k+1}}-\widetilde{\E}_{X_{t_k}}^{\theta_n,\sigma_0,\lambda(\ell)}[\Delta M_{t_{k+1}}\vert Y_{t_{k+1}}^{\theta_n,\sigma_0,\lambda(\ell)}=X_{t_{k+1}}]$.
\begin{lemma}\label{lemma5} As $n\to\infty$, $\sum_{k=0}^{n-1}R_{k,n}\overset{\P^{\theta_0,\sigma_0,\lambda_0}}{\longrightarrow}0$.
\end{lemma}
\begin{proof} 
Using $\E[\Delta N_{t_{k+1}}\vert X_{t_k}]=\lambda_0\Delta_n$, $\Delta M_{t_{k+1}}=\widetilde{M}_{t_{k+1}}^{\lambda(\ell)}-\widetilde{M}_{t_k}^{\lambda(\ell)}+\lambda(\ell)\Delta_n$, and Lemma \ref{c2lemma8}, 
\begin{align*}
\sum_{k=0}^{n-1}\E\left[R_{k,n}\vert \widehat{\mathcal{F}}_{t_k}\right]&=\dfrac{w}{\sigma_0^2\sqrt{n\Delta_n}}\sum_{k=0}^{n-1}\int_0^1\E\left[U_k\vert X_{t_k}\right]d\ell\\
&=\dfrac{w}{\sigma_0^2\sqrt{n\Delta_n}}\sum_{k=0}^{n-1}\int_0^1\left(\lambda_0\Delta_n-\lambda_0\Delta_n\right)d\ell=0.
\end{align*}
Thus the term appearing in condition (i) of Lemma \ref{zero} actually equals zero.

Next, proceeding as in Lemma \ref{lemma3} by multiplying the random variable inside the expectation by ${\bf 1}_{\widehat{J}_{0,k}} +{\bf 1}_{\widehat{J}_{1,k}}+ {\bf 1}_{\widehat{J}_{\geq 2,k}}$, and applying Lemma \ref{deviation}, we get that for $n$ large enough, for any $\alpha\in(0,\frac{1}{2})$, $p_1>1$, $q_1>1$ with $p_1q_1<2$, and $\mu_1\in(1,2)$,
\begin{align*} 
\sum_{k=0}^{n-1}\E\left[R_{k,n}^2\vert \widehat{\mathcal{F}}_{t_k}\right]&\leq\sum_{k=0}^{n-1}\dfrac{w^2}{\sigma_0^4n\Delta_n}\int_0^1\E\left[U_k^2\vert X_{t_k}\right]d\ell\\
&=\sum_{k=0}^{n-1}\dfrac{w^2}{\sigma_0^4n\Delta_n}\int_0^1\left(M_{0,1}^{\theta_n,\sigma_0,\lambda(\ell)}+M_{1,1}^{\theta_n,\sigma_0,\lambda(\ell)}+M_{\geq 2,1}^{\theta_n,\sigma_0,\lambda(\ell)}\right)d\ell\\
&\leq Cw^2\left(\Delta_n^{-1}e^{-C_1\Delta_n^{2\alpha-1}}+\Delta_n^{\frac{2}{p_1q_1}-1}+\Delta_n^{\frac{2}{\mu_1}-1}\right),
\end{align*}
for some constants $C, C_1>0$, which converges to zero in $\P^{\theta_0,\sigma_0,\lambda_0}$-probability as $n\to\infty$. Thus, by Lemma \ref{zero}, the result follows.
\end{proof}

\begin{lemma}\label{lemma6} As $n\to\infty$,
\begin{align*} 
\sum_{k=0}^{n-1}\dfrac{v}{\sqrt{n\Delta_n^2}}\int_0^1\widetilde{\E}_{X_{t_k}}^{\theta_n,\sigma(\ell),\lambda_n}\left[H^{\theta_n,\sigma(\ell),\lambda_n}\big\vert Y^{\theta_n,\sigma(\ell),\lambda_n}_{t_{k+1}}=X_{t_{k+1}}\right]d\ell\overset{\P^{\theta_0,\sigma_0,\lambda_0}}{\longrightarrow}0.
\end{align*}
\end{lemma}
\begin{proof} 
For $\mu\in\{1,2\}$, we set $h_{\mu}(X_{t_{k+1}}):=\widetilde{\E}_{X_{t_k}}^{\theta_n,\sigma(\ell),\lambda_n}[(H^{\theta_n,\sigma(\ell),\lambda_n})^{\mu}\big\vert Y^{\theta_n,\sigma(\ell),\lambda_n}_{t_{k+1}}=X_{t_{k+1}}]$. First, using  Lemma \ref{change}, \eqref{es4}, Lemma \ref{lemma15}, Jensen's inequality, and \eqref{es6}, for $n$ large enough,
\begin{align*}
&\left\vert\sum_{k=0}^{n-1}\dfrac{v}{\sqrt{n\Delta_n^2}}\int_0^1\E\left[\widetilde{\E}_{X_{t_k}}^{\theta_n,\sigma(\ell),\lambda_n}\left[H^{\theta_n,\sigma(\ell),\lambda_n}\big\vert Y^{\theta_n,\sigma(\ell),\lambda_n}_{t_{k+1}}=X_{t_{k+1}}\right]\big\vert \widehat{\mathcal{F}}_{t_k}\right]d\ell\right\vert\\
&=\bigg\vert\sum_{k=0}^{n-1}\dfrac{ve^{w\sqrt{\frac{\Delta_n}{n}}}}{\sqrt{n\Delta_n^2}}\int_0^1\bigg\{\E\left[h_1(X_{t_{k+1}}^{\theta_n,\sigma(\ell),\lambda_n}){\bf{1}}_{\widehat{J}_{0,k}}\Big(\frac{q_{(0)}^{\theta_0,\sigma_0,\lambda_0}}{q_{(0)}^{\theta_n,\sigma(\ell),\lambda_n}}-1\Big)\big\vert X_{t_k}^{\theta_n,\sigma(\ell),\lambda_n}=X_{t_k}\right]\\
&+\sum_{j=1}^{\infty}\frac{\lambda_0^j}{\lambda_n^j}\int_{\Sigma_k^j}\E\left[h_1(X_{t_{k+1}}^{\theta_n,\sigma(\ell),\lambda_n}){\bf{1}}_{\{\widehat{J}_{j,k},s_1,.,s_j\}}\Big(\frac{q_{(j)}^{\theta_0,\sigma_0,\lambda_0}}{q_{(j)}^{\theta_n,\sigma(\ell),\lambda_n}}-1\Big)\big\vert X_{t_k}^{\theta_n,\sigma(\ell),\lambda_n}=X_{t_k}\right]ds_1\cdot ds_j\bigg\}d\ell\bigg\vert\\
&\leq C\Delta_n^{\frac{1}{2\max\{\overline{q}_1,\widetilde{q}_1,q_1\}}}\bigg(1+\sum_{j=1}^{\infty}\frac{\lambda_0^j}{\lambda_n^j}\frac{\Delta_n^j}{j!}\bigg\{\left(\frac{\lambda_n}{\lambda_0}\right)^{\frac{j}{\overline{q}_2}}(e^{-\lambda_0\Delta_n}\lambda_0^j)^{\frac{1}{\overline{p}_1 \overline{q}_2}}(e^{-\lambda_n\Delta_n}\lambda_n^j)^{\frac{1}{\overline{q}_3}}\\
&\qquad+\left(\frac{\lambda_n}{\lambda_0}\right)^{\frac{j}{\widetilde{q}_2}}(e^{-\lambda_0\Delta_n}\lambda_0^j)^{\frac{1}{\widetilde{p}_1 \widetilde{q}_2}}(e^{-\lambda_n\Delta_n}\lambda_n^j)^{\frac{1}{\widetilde{q}_3}}+\int_0^1\Big(\frac{\lambda_n}{\lambda_0+\frac{wh}{\sqrt{n\Delta_n}}}\Big)^{\frac{j}{q_2}}\\
&\qquad\times\Big(e^{-(\lambda_0+\frac{wh}{\sqrt{n\Delta_n}})\Delta_n}(\lambda_0+\frac{wh}{\sqrt{n\Delta_n}})^j\Big)^{\frac{1}{p_1 q_2}}(e^{-\lambda_n\Delta_n}\lambda_n^j)^{\frac{1}{q_3}}dh\bigg\}\bigg)\dfrac{\vert v\vert}{n}\sum_{k=0}^{n-1}\left(1+\vert X_{t_k}\vert^q\right),
\end{align*}
for some $C, q>0$, where $(\overline{q}_1,\overline{q}_2,\overline{q}_3)$, $(\widetilde{q}_1,\widetilde{q}_2,\widetilde{q}_3)$, $(q_1,q_2,q_3)$, $\overline{p}_1, \widetilde{p}_1, p_1$ are as in Lemma \ref{lemma15}. Here  $q_{(j)}^{\theta,\sigma,\lambda}(t-s,x,y)$ are defined in Subsection \ref{condensity},  $\frac{q_{(0)}^{\theta_0,\sigma_0,\lambda_0}}{q_{(0)}^{\theta_n,\sigma(\ell),\lambda_n}}\equiv\frac{q_{(0)}^{\theta_0,\sigma_0,\lambda_0}}{q_{(0)}^{\theta_n,\sigma(\ell),\lambda_n}}(\Delta_n,X_{t_k}^{\theta_n,\sigma(\ell),\lambda_n},X^{\theta_n,\sigma(\ell),\lambda_n}_{t_{k+1}})$, $\frac{q_{(j)}^{\theta_0,\sigma_0,\lambda_0}}{q_{(j)}^{\theta_n,\sigma(\ell),\lambda_n}}\equiv\frac{q_{(j)}^{\theta_0,\sigma_0,\lambda_0}}{q_{(j)}^{\theta_n,\sigma(\ell),\lambda_n}}(\Delta_n,X_{t_k}^{\theta_n,\sigma(\ell),\lambda_n},X^{\theta_n,\sigma(\ell),\lambda_n}_{t_{k+1}};s_1,\ldots,s_j)$. Then using the finiteness of the sum w.r.t. $j$, the right hand side converges to zero in $\P^{\theta_0,\sigma_0,\lambda_0}$-probability as $n\to\infty$.

Next, applying Jensen's and H\"older's inequalities with $p_1, p_2$ conjugate, together with Lemmas \ref{change}, \ref{lemma12}, and \eqref{es6}, we get that for $n$ large enough, and some constants $C, q>0$,
\begin{align*}
&\sum_{k=0}^{n-1}\dfrac{v^2}{n\Delta_n^2}\E\left[\left(\int_0^1\widetilde{\E}_{X_{t_k}}^{\theta_n,\sigma(\ell),\lambda_n}\left[H^{\theta_n,\sigma(\ell),\lambda_n}\big\vert Y^{\theta_n,\sigma(\ell),\lambda_n}_{t_{k+1}}=X_{t_{k+1}}\right]d\ell\right)^2\big\vert \widehat{\mathcal{F}}_{t_k}\right]\\
&\leq\sum_{k=0}^{n-1}\dfrac{v^2}{n\Delta_n^2}e^{w\sqrt{\frac{\Delta_n}{n}}}\int_0^1\bigg\{\E\left[h_2(X_{t_{k+1}}^{\theta_n,\sigma(\ell),\lambda_n}){\bf{1}}_{\widehat{J}_{0,k}}\frac{q_{(0)}^{\theta_0,\sigma_0,\lambda_0}}{q_{(0)}^{\theta_n,\sigma(\ell),\lambda_n}}\big\vert X_{t_k}^{\theta_n,\sigma(\ell),\lambda_n}=X_{t_k}\right]\\
&+\sum_{j=1}^{\infty}\left(\frac{\lambda_0}{\lambda_n}\right)^j\int_{\Sigma_k^j}\E\left[h_2(X_{t_{k+1}}^{\theta_n,\sigma(\ell),\lambda_n}){\bf{1}}_{\{\widehat{J}_{j,k},s_1,\ldots,s_j\}}\frac{q_{(j)}^{\theta_0,\sigma_0,\lambda_0}}{q_{(j)}^{\theta_n,\sigma(\ell),\lambda_n}}\big\vert X_{t_k}^{\theta_n,\sigma(\ell),\lambda_n}=X_{t_k}\right]ds_1\cdots ds_j\bigg\}d\ell\\
&\leq\sum_{k=0}^{n-1}\dfrac{Cv^2}{n\Delta_n^2}\int_0^1\left(\E\left[\left\vert h_2(X_{t_{k+1}}^{\theta_n,\sigma(\ell),\lambda_n})\right\vert^{p_1}\big\vert X_{t_k}^{\theta_n,\sigma(\ell),\lambda_n}=X_{t_k}\right]\right)^{\frac{1}{p_1}}\\
&\qquad\times\bigg\{\left(\E\left[{\bf{1}}_{\widehat{J}_{0,k}}\Big(\frac{q_{(0)}^{\theta_0,\sigma_0,\lambda_0}}{q_{(0)}^{\theta_n,\sigma(\ell),\lambda_n}}\Big)^{p_2}\big\vert X_{t_k}^{\theta_n,\sigma(\ell),\lambda_n}=X_{t_k}\right]\right)^{\frac{1}{p_2}}\\
&+\sum_{j=1}^{\infty}\left(\frac{\lambda_0}{\lambda_n}\right)^j\int_{\Sigma_k^j}\left(\E\left[{\bf{1}}_{\{\widehat{J}_{j,k},s_1,\ldots,s_j\}}\Big(\frac{q_{(j)}^{\theta_0,\sigma_0,\lambda_0}}{q_{(j)}^{\theta_n,\sigma(\ell),\lambda_n}}\Big)^{p_2}\big\vert X_{t_k}^{\theta_n,\sigma(\ell),\lambda_n}=X_{t_k}\right]\right)^{\frac{1}{p_2}}ds_1\cdots ds_j\bigg\}d\ell\\
&\leq\sum_{k=0}^{n-1}\dfrac{Cv^2}{n\Delta_n^2}\int_0^1\big(\E[\vert H^{\theta_n,\sigma(\ell),\lambda_n}\vert^{2p_1}\big\vert Y_{t_k}^{\theta_n,\sigma(\ell),\lambda_n}=X_{t_k}]\big)^{\frac{1}{p_1}}d\ell\Big\{1+\sum_{j=1}^{\infty}\big(\frac{\lambda_0}{\lambda_n}\big)^j\frac{\Delta_n^j}{j!}(e^{-\lambda_n\Delta_n}\lambda_n^j)^{\frac{1}{p_2}}\Big\}\\
&\leq C\Delta_n^{\frac{1}{2p_1}}\dfrac{v^2}{n}\sum_{k=0}^{n-1}\left(1+\vert X_{t_k}\vert^q\right),
\end{align*}
which converges to zero in $\P^{\theta_0,\sigma_0,\lambda_0}$-probability as $n\to\infty$. Here we have used the finiteness of the sum w.r.t. $j$. Thus, by Lemma \ref{zero}, the result desired follows.
\end{proof}

\begin{lemma}\label{lemma7} As $n\to\infty$,
\begin{align*}
&\sum_{k=0}^{n-1}\dfrac{v}{\sqrt{n\Delta_n^2}}\int_0^1\dfrac{1}{\sigma(\ell)^3}\bigg\{\left(H_{6}+H_{7}\right)^2+2\sigma_0\Delta B_{t_{k+1}}\left(H_{6}+H_{7}\right)\\
&\qquad-\widetilde{\E}_{X_{t_k}}^{\theta_n,\sigma(\ell),\lambda_n}\bigg[\left(H_4^{\theta_n,\sigma(\ell),\lambda_n}+H_{5}^{\theta_n,\sigma(\ell),\lambda_n}\right)^2+2\sigma(\ell)\Delta W_{t_{k+1}}\\
&\qquad\qquad\times\left(H_4^{\theta_n,\sigma(\ell),\lambda_n}+H_{5}^{\theta_n,\sigma(\ell),\lambda_n}\right)\big\vert Y^{\theta_n,\sigma(\ell),\lambda_n}_{t_{k+1}}=X_{t_{k+1}}\bigg]\bigg\}d\ell\overset{\P^{\theta_0,\sigma_0,\lambda_0}}{\longrightarrow}0.
\end{align*}
\end{lemma}
\begin{proof}
We split the term inside the conditional expectation as
\begin{align*}
&\left(H_4^{\theta_n,\sigma(\ell),\lambda_n}+H_{5}^{\theta_n,\sigma(\ell),\lambda_n}\right)^2+2\sigma(\ell)\Delta W_{t_{k+1}}\left(H_4^{\theta_n,\sigma(\ell),\lambda_n}+H_{5}^{\theta_n,\sigma(\ell),\lambda_n}\right)=K_{n,\ell}+F_{n,\ell},
\end{align*}
where
\begin{align*}
K_{n,\ell}:&=\theta_n^2\left(\int_{t_k}^{t_{k+1}}Y_{s}^{\theta_n,\sigma(\ell),\lambda_n}ds\right)^2-2\theta_n\int_{t_k}^{t_{k+1}}Y_{s}^{\theta_n,\sigma(\ell),\lambda_n}ds\left(\sigma(\ell)\Delta W_{t_{k+1}}+\widetilde{M}_{t_{k+1}}^{\lambda_n}-\widetilde{M}_{t_{k}}^{\lambda_n}\right),\\
F_{n,\ell}:&=\left(\widetilde{M}_{t_{k+1}}^{\lambda_n}-\widetilde{M}_{t_{k}}^{\lambda_n}\right)^2+2\sigma(\ell)\Delta W_{t_{k+1}}\left(\widetilde{M}_{t_{k+1}}^{\lambda_n}-\widetilde{M}_{t_{k}}^{\lambda_n}\right).
\end{align*}

Similarly, we split $\left(H_{6}+H_{7}\right)^2+2\sigma_0\Delta B_{t_{k+1}}\left(H_{6}+H_{7}\right)=K_0+F_0$, where
\begin{align*}
K_0:&=\theta_0^2\left(\int_{t_k}^{t_{k+1}}X_{s}^{\theta_0,\sigma_0,\lambda_0}ds\right)^2-2\theta_0\int_{t_k}^{t_{k+1}}X_{s}^{\theta_0,\sigma_0,\lambda_0}ds\left(\sigma_0\Delta B_{t_{k+1}}+\widetilde{N}_{t_{k+1}}^{\lambda_0}-\widetilde{N}_{t_{k}}^{\lambda_0}\right),\\
F_0:&=\left(\widetilde{N}_{t_{k+1}}^{\lambda_0}-\widetilde{N}_{t_{k}}^{\lambda_0}\right)^2+2\sigma_0\Delta B_{t_{k+1}}\left(\widetilde{N}_{t_{k+1}}^{\lambda_0}-\widetilde{N}_{t_{k}}^{\lambda_0}\right).
\end{align*}

Therefore, in order to prove Lemma \ref{lemma7}, it suffices to show that as $n\to\infty$,
\begin{align}
&\sum_{k=0}^{n-1}\dfrac{v}{\sqrt{n\Delta_n^2}}\int_0^1\dfrac{1}{\sigma(\ell)^3}\left(K_0-\widetilde{\E}_{X_{t_k}}^{\theta_n,\sigma(\ell),\lambda_n}\left[K_{n,\ell}\big\vert Y^{\theta_n,\sigma(\ell),\lambda_n}_{t_{k+1}}=X_{t_{k+1}}\right]\right)d\ell\overset{\P^{\theta_0,\sigma_0,\lambda_0}}{\longrightarrow} 0,\label{K}\\
&\sum_{k=0}^{n-1}\dfrac{v}{\sqrt{n\Delta_n^2}}\int_0^1\dfrac{1}{\sigma(\ell)^3}\left(F_0-\widetilde{\E}_{X_{t_k}}^{\theta_n,\sigma(\ell),\lambda_n}\left[F_{n,\ell}\big\vert Y^{\theta_n,\sigma(\ell),\lambda_n}_{t_{k+1}}=X_{t_{k+1}}\right]\right)d\ell\overset{\P^{\theta_0,\sigma_0,\lambda_0}}{\longrightarrow} 0 \label{F}.
\end{align}

We first show \eqref{K} by proving that conditions (i) and (ii) of Lemma \ref{zero} hold under $\P^{\theta_0,\sigma_0,\lambda_0}$. We set $g(X_{t_{k+1}}):=\widetilde{\E}_{X_{t_k}}^{\theta_n,\sigma(\ell),\lambda_n}[K_{n,\ell}\big\vert Y^{\theta_n,\sigma(\ell),\lambda_n}_{t_{k+1}}=X_{t_{k+1}}]$. Applying Lemma \ref{change} to the conditional expectation, we get that
\begin{equation*}
\E\left[K_0-\widetilde{\E}_{X_{t_k}}^{\theta_n,\sigma(\ell),\lambda_n}\left[K_{n,\ell}\big\vert Y^{\theta_n,\sigma(\ell),\lambda_n}_{t_{k+1}}=X_{t_{k+1}}\right]\big\vert \widehat{\mathcal{F}}_{t_k}\right]=S_1+S_2,
\end{equation*}
where
\begin{align*}
&S_1:=-e^{w\sqrt{\frac{\Delta_n}{n}}}\bigg\{\E\left[g(X_{t_{k+1}}^{\theta_n,\sigma(\ell),\lambda_n}){\bf{1}}_{\widehat{J}_{0,k}}\left(\frac{q_{(0)}^{\theta_0,\sigma_0,\lambda_0}}{q_{(0)}^{\theta_n,\sigma(\ell),\lambda_n}}-1\right)\big\vert X_{t_k}^{\theta_n,\sigma(\ell),\lambda_n}=X_{t_k}\right]\\
&+\sum_{j=1}^{\infty}\frac{\lambda_0^j}{\lambda_n^j}\int_{\Sigma_k^j}\E\left[g(X_{t_{k+1}}^{\theta_n,\sigma(\ell),\lambda_n}){\bf{1}}_{\{\widehat{J}_{j,k},s_1,.,s_j\}}\Big(\frac{q_{(j)}^{\theta_0,\sigma_0,\lambda_0}}{q_{(j)}^{\theta_n,\sigma(\ell),\lambda_n}}-1\Big)\big\vert X_{t_k}^{\theta_n,\sigma(\ell),\lambda_n}=X_{t_k}\right]ds_1\cdots ds_j\bigg\},\\
&S_2:=\E\left[K_0\big\vert X_{t_k}\right]-\E\left[K_{n,\ell}\big\vert Y_{t_k}^{\theta_n,\sigma(\ell),\lambda_n}=X_{t_k}\right].
\end{align*}

Proceeding as in Lemma \ref{lemma6}, we obtain that
\begin{align*}
&\left\vert\sum_{k=0}^{n-1}\dfrac{v}{\sqrt{n\Delta_n^2}}\int_0^1\dfrac{1}{\sigma(\ell)^3}S_1d\ell\right\vert\leq C\Delta_n^{\frac{1}{2\max\{\overline{q}_1,\widetilde{q}_1,q_1\}}}\bigg(1+ \sum_{j=1}^{\infty}\frac{\lambda_0^j}{\lambda_n^j}\frac{\Delta_n^j}{j!}\bigg\{\left(\frac{\lambda_n}{\lambda_0}\right)^{\frac{j}{\overline{q}_2}}(e^{-\lambda_0\Delta_n}\lambda_0^j)^{\frac{1}{\overline{p}_1 \overline{q}_2}}\\
&\qquad\times(e^{-\lambda_n\Delta_n}\lambda_n^j)^{\frac{1}{\overline{q}_3}}+\left(\frac{\lambda_n}{\lambda_0}\right)^{\frac{j}{\widetilde{q}_2}}(e^{-\lambda_0\Delta_n}\lambda_0^j)^{\frac{1}{\widetilde{p}_1 \widetilde{q}_2}}(e^{-\lambda_n\Delta_n}\lambda_n^j)^{\frac{1}{\widetilde{q}_3}}+\int_0^1\Big(\frac{\lambda_n}{\lambda_0+\frac{wh}{\sqrt{n\Delta_n}}}\Big)^{\frac{j}{q_2}}\\
&\qquad\times\Big(e^{-(\lambda_0+\frac{wh}{\sqrt{n\Delta_n}})\Delta_n}(\lambda_0+\frac{wh}{\sqrt{n\Delta_n}})^j\Big)^{\frac{1}{p_1 q_2}}(e^{-\lambda_n\Delta_n}\lambda_n^j)^{\frac{1}{q_3}}dh\bigg\}\bigg)\dfrac{\vert v\vert}{n}\sum_{k=0}^{n-1}\left(1+\vert X_{t_k}\vert^q\right),
\end{align*}
for some constants $C, q>0$, where $(\overline{q}_1,\overline{q}_2,\overline{q}_3)$, $(\widetilde{q}_1,\widetilde{q}_2,\widetilde{q}_3)$, $(q_1,q_2,q_3)$, and $\overline{p}_1, \widetilde{p}_1, p_1$ are as in Lemma \ref{lemma15}. This converges to zero in $\P^{\theta_0,\sigma_0,\lambda_0}$-probability as $n\to\infty$. On the other hand, 
\begin{align*}
\E\left[K_0\big\vert X_{t_k}\right]&=2\theta_0^2\int_{t_k}^{t_{k+1}}\int_{t_k}^{s}\E\left[X_{s}^{\theta_0,\sigma_0,\lambda_0}X_{u}^{\theta_0,\sigma_0,\lambda_0}\big\vert X_{t_k}\right]duds\\
&\qquad-2\theta_0\int_{t_k}^{t_{k+1}}\E\left[X_{s}^{\theta_0,\sigma_0,\lambda_0}\left(\sigma_0\left(B_{s}-B_{t_{k}}\right)+\widetilde{N}_{s}^{\lambda_0}-\widetilde{N}_{t_{k}}^{\lambda_0}\right)\big\vert X_{t_k}\right]ds\\
&=-2\theta_0\int_{t_k}^{t_{k+1}}\E\left[X_{s}^{\theta_0,\sigma_0,\lambda_0}\left(X_s^{\theta_0,\sigma_0,\lambda_0}-X_{t_k}\right)\big\vert X_{t_k}\right]ds,
\end{align*}
which implies that $S_{2}=-2\theta_0S_{2,1}-S_{2,2}$, where
\begin{align*}
S_{2,1}:&=\int_{t_k}^{t_{k+1}}\bigg(\E\left[X_{s}^{\theta_0,\sigma_0,\lambda_0}\left(X_s^{\theta_0,\sigma_0,\lambda_0}-X_{t_k}\right)\big\vert X_{t_k}\right]\\
&\qquad-\E\left[Y_{s}^{\theta_n,\sigma(\ell),\lambda_n}\left(Y_s^{\theta_n,\sigma(\ell),\lambda_n}-Y_{t_k}^{\theta_n,\sigma(\ell),\lambda_n}\right)\big\vert Y_{t_k}^{\theta_n,\sigma(\ell),\lambda_n}=X_{t_k}\right]\bigg)ds,\\
S_{2,2}:&=2\frac{u}{\sqrt{n\Delta_n}}\int_{t_k}^{t_{k+1}}\E\left[Y_{s}^{\theta_n,\sigma(\ell),\lambda_n}\left(Y_s^{\theta_n,\sigma(\ell),\lambda_n}-Y_{t_k}^{\theta_n,\sigma(\ell),\lambda_n}\right)\big\vert Y_{t_k}^{\theta_n,\sigma(\ell),\lambda_n}=X_{t_k}\right]ds.
\end{align*}

Setting $h(X_s^{\theta_n,\sigma(\ell),\lambda_n}):=X_{s}^{\theta_n,\sigma(\ell),\lambda_n}(X_s^{\theta_n,\sigma(\ell),\lambda_n}-X_{t_k}^{\theta_n,\sigma(\ell),\lambda_n})$, $\widehat{J}_{j,t_k,s}:=\{N_{s}-N_{t_{k}}=j\}$, $\Sigma_{t_k,s}^j:=\{(s_1,\ldots,s_j): t_k<s_1<\cdots<s_j<s\}$, and proceeding as in Lemma \ref{change}, we have 
\begin{align*}
&S_{2,1}=e^{w\sqrt{\frac{\Delta_n}{n}}}\int_{t_k}^{t_{k+1}}\bigg\{\E\Big[h(X_s^{\theta_n,\sigma(\ell),\lambda_n}){\bf{1}}_{\widehat{J}_{0,t_k,s}}\Big(\frac{q_{(0)}^{\theta_0,\sigma_0,\lambda_0}}{q_{(0)}^{\theta_n,\sigma(\ell),\lambda_n}}-1\Big)\big\vert X_{t_k}^{\theta_n,\sigma(\ell),\lambda_n}=X_{t_k}\Big]+\sum_{j=1}^{\infty}\frac{\lambda_0^j}{\lambda_n^j}\\
&\times\int_{\Sigma_{t_k,s}^j}\E\left[h(X_s^{\theta_n,\sigma(\ell),\lambda_n}){\bf{1}}_{\{\widehat{J}_{j,t_k,s},s_1,\ldots,s_j\}}\Big(\frac{q_{(j)}^{\theta_0,\sigma_0,\lambda_0}}{q_{(j)}^{\theta_n,\sigma(\ell),\lambda_n}}-1\Big)\big\vert X_{t_k}^{\theta_n,\sigma(\ell),\lambda_n}=X_{t_k}\right]ds_1\cdots ds_j\bigg\}ds,
\end{align*}
where we denote $\frac{q_{(0)}^{\theta_0,\sigma_0,\lambda_0}}{q_{(0)}^{\theta_n,\sigma(\ell),\lambda_n}}\equiv\frac{q_{(0)}^{\theta_0,\sigma_0,\lambda_0}}{q_{(0)}^{\theta_n,\sigma(\ell),\lambda_n}}(s-t_k,X_{t_k}^{\theta_n,\sigma(\ell),\lambda_n},X^{\theta_n,\sigma(\ell),\lambda_n}_{s})$, $\frac{q_{(j)}^{\theta_0,\sigma_0,\lambda_0}}{q_{(j)}^{\theta_n,\sigma(\ell),\lambda_n}}\equiv\frac{q_{(j)}^{\theta_0,\sigma_0,\lambda_0}}{q_{(j)}^{\theta_n,\sigma(\ell),\lambda_n}}(s-t_k,X_{t_k}^{\theta_n,\sigma(\ell),\lambda_n},X^{\theta_n,\sigma(\ell),\lambda_n}_{s};s_1,\ldots,s_j)$.
Then proceeding as in Lemma \ref{lemma6}, we obtain that
\begin{align*}
&\sum_{k=0}^{n-1}\dfrac{v}{\sqrt{n\Delta_n^2}}\int_0^1\dfrac{1}{\sigma(\ell)^3}S_{2,1}d\ell\overset{\P^{\theta_0,\sigma_0,\lambda_0}}{\longrightarrow} 0.
\end{align*}

Next, adding and subtracting $Y_{t_k}^{\theta_n,\sigma(\ell),\lambda_n}$ and use \eqref{c2eq1rajoute} to get $S_{2,2}=S_{2,2,1}+S_{2,2,2}$, where
\begin{align*}
S_{2,2,1}:&=2\frac{u}{\sqrt{n\Delta_n}}\int_{t_k}^{t_{k+1}}\E\left[\left(Y_{s}^{\theta_n,\sigma(\ell),\lambda_n}-Y_{t_k}^{\theta_n,\sigma(\ell),\lambda_n}\right)^2\big\vert Y_{t_k}^{\theta_n,\sigma(\ell),\lambda_n}=X_{t_k}\right]ds,\\
S_{2,2,2}:&=-2\theta_n\frac{u}{\sqrt{n\Delta_n}}X_{t_k}\int_{t_k}^{t_{k+1}}\int_{t_k}^{s}\E\left[Y_u^{\theta_n,\sigma(\ell),\lambda_n}\big\vert Y_{t_k}^{\theta_n,\sigma(\ell),\lambda_n}=X_{t_k}\right]duds.
\end{align*}
Therefore, using Lemma \ref{moment3}, we get that 
\begin{align*}
&\left\vert\sum_{k=0}^{n-1}\dfrac{v}{\sqrt{n\Delta_n^2}}\int_0^1\dfrac{1}{\sigma(\ell)^3}S_{2,2}d\ell\right\vert\leq C\vert uv\vert\dfrac{\sqrt{\Delta_n}}{n}\sum_{k=0}^{n-1}\left(1+\vert X_{t_k}\vert^q\right),
\end{align*}
for some constants $C, q>0$, which concludes that as $n\to\infty$,
\begin{align*}
\sum_{k=0}^{n-1}\dfrac{v}{\sqrt{n\Delta_n^2}}\int_0^1\dfrac{1}{\sigma(\ell)^3}S_2d\ell\overset{\P^{\theta_0,\sigma_0,\lambda_0}}{\longrightarrow} 0.
\end{align*}
This finishes the proof of Lemma \ref{zero} (i). 

Next, proceeding as in the proof of Lemma \ref{lemma6} by applying Jensen's and H\"older's inequalities with $p_1, p_2$ conjugate, and Lemmas \ref{change}, \ref{lemma12}, we get that for $n$ large enough,
\begin{align*}
&\sum_{k=0}^{n-1}\dfrac{v^2}{n\Delta_n^2}\E\left[\left(\int_0^1\dfrac{1}{\sigma(\ell)^3}\left(K_0-\widetilde{\E}_{X_{t_k}}^{\theta_n,\sigma(\ell),\lambda_n}\left[K_{n,\ell}\big\vert Y^{\theta_n,\sigma(\ell),\lambda_n}_{t_{k+1}}=X_{t_{k+1}}\right]\right)d\ell\right)^2\big\vert \widehat{\mathcal{F}}_{t_k}\right]\\
&\leq\sum_{k=0}^{n-1}\dfrac{2 v^2}{n\Delta_n^2}\int_0^1\dfrac{1}{\sigma(\ell)^6}\left\{\E\left[K_0^2\big\vert X_{t_k}\right]+\E\bigg[\widetilde{\E}_{X_{t_k}}^{\theta_n,\sigma(\ell),\lambda_n}\left[K_{n,\ell}^2\big\vert Y^{\theta_n,\sigma(\ell),\lambda_n}_{t_{k+1}}=X_{t_{k+1}}\right]\big\vert X_{t_k}\bigg]\right\}d\ell\\
&\leq Cv^2\sqrt{\Delta_n}+C\Delta_n^{\frac{1}{2p_1}}\frac{v^2}{n}\sum_{k=0}^{n-1}\left(1+\vert X_{t_k}\vert^q\right),
\end{align*}
for some constants $C, q>0$, which converges to zero in $\P^{\theta_0,\sigma_0,\lambda_0}$-probability as $n\to\infty$. This finishes the proof of \eqref{K}.

Finally, it remains to treat \eqref{F}. For this, using equations \eqref{c2eq1rajoute} and \eqref{c2eq1} to rewrite the Brownian increments $\Delta W_{t_{k+1}}$, $\Delta B_{t_{k+1}}$ in terms of the process increments $\Delta Y^{\theta_n,\sigma(\ell),\lambda_n}_{t_{k+1}}$ and $\Delta X^{\theta_0,\sigma_0,\lambda_0}_{t_{k+1}}$, Poisson increments, and drift terms, it suffices to show that the following terms converge to zero in $\P^{\theta_0,\sigma_0,\lambda_0}$-probability as $n\to\infty$:
\begin{align}
&\sum_{k=0}^{n-1}\dfrac{v}{\sqrt{n}}\int_0^1\dfrac{1}{\sigma(\ell)^3}\left(\Delta N_{t_{k+1}}-\widetilde{\E}_{X_{t_k}}^{\theta_n,\sigma(\ell),\lambda_n}\left[\Delta M_{t_{k+1}}\big\vert Y^{\theta_n,\sigma(\ell),\lambda_n}_{t_{k+1}}=X_{t_{k+1}}\right]\right)d\ell,\label{F1}\\
&\sum_{k=0}^{n-1}\dfrac{v}{\sqrt{n}}\int_0^1\dfrac{X_{t_{k}}}{\sigma(\ell)^3}\left(\Delta N_{t_{k+1}}-\widetilde{\E}_{X_{t_k}}^{\theta_n,\sigma(\ell),\lambda_n}\left[\Delta M_{t_{k+1}}\big\vert Y^{\theta_n,\sigma(\ell),\lambda_n}_{t_{k+1}}=X_{t_{k+1}}\right]\right)d\ell,\label{F2}\\
&\sum_{k=0}^{n-1}\dfrac{v}{\sqrt{n\Delta_n^2}}\int_0^1\dfrac{1}{\sigma(\ell)^3}\left((\Delta N_{t_{k+1}})^2-\widetilde{\E}_{X_{t_k}}^{\theta_n,\sigma(\ell),\lambda_n}\left[(\Delta M_{t_{k+1}})^2\big\vert Y^{\theta_n,\sigma(\ell),\lambda_n}_{t_{k+1}}=X_{t_{k+1}}\right]\right)d\ell,\label{F3}\\
&\sum_{k=0}^{n-1}\dfrac{v}{\sqrt{n\Delta_n^2}}\int_0^1\dfrac{\Delta X_{t_{k+1}}}{\sigma(\ell)^3}\left(\Delta N_{t_{k+1}}-\widetilde{\E}_{X_{t_k}}^{\theta_n,\sigma(\ell),\lambda_n}\left[\Delta M_{t_{k+1}}\big\vert Y^{\theta_n,\sigma(\ell),\lambda_n}_{t_{k+1}}=X_{t_{k+1}}\right]\right)d\ell,\label{F4}\\
&\sum_{k=0}^{n-1}\dfrac{v\theta_0}{\sqrt{n\Delta_n^2}}\int_0^1\dfrac{d\ell}{\sigma(\ell)^3}\int_{t_k}^{t_{k+1}}(X_{s}^{\theta_0,\sigma_0,\lambda_0}-X_{t_k})ds(\widetilde{N}_{t_{k+1}}^{\lambda_0}-\widetilde{N}_{t_{k}}^{\lambda_0}),\label{F5}\\
&\sum_{k=0}^{n-1}\dfrac{v\theta_n}{\sqrt{n\Delta_n^2}}\int_0^1\dfrac{1}{\sigma(\ell)^3}\widetilde{\E}_{X_{t_k}}^{\theta_n,\sigma(\ell),\lambda_n}\bigg[\int_{t_k}^{t_{k+1}}(Y_{s}^{\theta_n,\sigma(\ell),\lambda_n}-Y_{t_k}^{\theta_n,\sigma(\ell),\lambda_n})ds\notag\\
&\qquad\times(\widetilde{M}_{t_{k+1}}^{\lambda_n}-\widetilde{M}_{t_{k}}^{\lambda_n})\big\vert Y^{\theta_n,\sigma(\ell),\lambda_n}_{t_{k+1}}=X_{t_{k+1}}\bigg]d\ell.\label{F6}
\end{align} 
We then proceed similarly as in the proof of Lemmas \ref{lemma3} and \ref{lemma5} using the arguments of large deviation type estimates established in Lemma \ref{deviation} in order to deal with \eqref{F1}-\eqref{F4}. Moreover, due to $\Delta_n$ appearing in the denominator of \eqref{F3}-\eqref{F4}, we need to use in this case the decomposition ${\bf 1}_{\widehat{J}_{0,k}} +{\bf 1}_{\widehat{J}_{1,k}}+ {\bf 1}_{\widehat{J}_{2,k}}+ {\bf 1}_{\widehat{J}_{\geq 3,k}}$ instead of ${\bf 1}_{\widehat{J}_{0,k}} +{\bf 1}_{\widehat{J}_{1,k}}+ {\bf 1}_{\widehat{J}_{\geq 2,k}}$, where $\widehat{J}_{\geq 3,k}:=\{\Delta N_{t_{k+1}}\geq 3\}$. On the other hand, when treating \eqref{F5}-\eqref{F6}, the decreasing rate condition $n\Delta_n^2\to 0$ as $n\to\infty$ is needed for showing Lemma \ref{zero} (i). Thus, the desired proof is now completed.
\end{proof}

\subsection{Main contributions: LAN property}
\label{maincontri}
\begin{lemma}\label{main} Let $\Gamma(\theta_0,\sigma_0,\lambda_0)$ be defined in Theorem \ref{c2theorem}. Then as $n\to\infty$,
\begin{align*}
\sum_{k=0}^{n-1}\left(\xi_{k,n}+\eta_{k,n}+\beta_{k,n}\right)\overset{\mathcal{L}(\P^{\theta_0,\sigma_0,\lambda_0})}{\longrightarrow} z^{\ast}\mathcal{N}\left(0,\Gamma(\theta_0,\sigma_0,\lambda_0)\right)-\dfrac{1}{2}z^{\ast} \Gamma(\theta_0,\sigma_0,\lambda_0)z.
\end{align*}
\end{lemma}
\begin{proof}
Applying \cite[Lemma 4.3]{J11} to $\zeta_{k,n}=\xi_{k,n}+\eta_{k,n}+\beta_{k,n}$, it suffices to show that
\begin{align}
&\sum_{k=0}^{n-1}\E\left[\xi_{k,n}+\eta_{k,n}+\beta_{k,n}\vert \widehat{\mathcal{F}}_{t_k}\right]\overset{\P^{\theta_0,\sigma_0,\lambda_0}}{\longrightarrow}-\dfrac{u^2}{2}\dfrac{1}{2\theta_0}\left(1+\dfrac{1}{\sigma_0^2}\right)-\dfrac{v^2}{2}\dfrac{2}{\sigma_0^2}-\dfrac{w^2}{2\sigma_0^2}\left(1+\dfrac{\sigma_0^2}{\lambda_0}\right), \label{c2eql1}\\
&\sum_{k=0}^{n-1}\left(\E\left[\xi_{k,n}^2\vert \widehat{\mathcal{F}}_{t_k}\right]-\left(\E\left[\xi_{k,n}\vert\widehat{\mathcal{F}}_{t_k}\right]\right)^2\right)\overset{\P^{\theta_0,\sigma_0,\lambda_0}}{\longrightarrow}\dfrac{u^2}{2\theta_0}\left(1+\dfrac{1}{\sigma_0^2}\right), \label{c2eql2}\\
&\sum_{k=0}^{n-1}\left(\E\left[\eta_{k,n}^2\vert \widehat{\mathcal{F}}_{t_k}\right]-\left(\E\left[\eta_{k,n}\vert \widehat{\mathcal{F}}_{t_k}\right]\right)^2\right)\overset{\P^{\theta_0,\sigma_0,\lambda_0}}{\longrightarrow}v^2\dfrac{2}{\sigma_0^2}, \label{c2eql5}\\
&\sum_{k=0}^{n-1}\left(\E\left[\beta_{k,n}^2\vert \widehat{\mathcal{F}}_{t_k}\right]-\left(\E\left[\beta_{k,n}\vert \widehat{\mathcal{F}}_{t_k}\right]\right)^2\right)\overset{\P^{\theta_0,\sigma_0,\lambda_0}}{\longrightarrow}\dfrac{w^2}{\sigma_0^2}\left(1+\dfrac{\sigma_0^2}{\lambda_0}\right), \label{c2eql8}\\
&\sum_{k=0}^{n-1}\left(\E\left[\xi_{k,n}\eta_{k,n}\vert \widehat{\mathcal{F}}_{t_k}\right]-\E\left[\xi_{k,n}\vert \widehat{\mathcal{F}}_{t_k}\right]\E\left[\eta_{k,n}\vert \widehat{\mathcal{F}}_{t_k}\right]\right)\overset{\P^{\theta_0,\sigma_0,\lambda_0}}{\longrightarrow}0, \label{c2eql10}\\
&\sum_{k=0}^{n-1}\left(\E\left[\xi_{k,n}\beta_{k,n}\vert \widehat{\mathcal{F}}_{t_k}\right]-\E\left[\xi_{k,n}\vert \widehat{\mathcal{F}}_{t_k}\right]\E\left[\beta_{k,n}\vert \widehat{\mathcal{F}}_{t_k}\right]\right)\overset{\P^{\theta_0,\sigma_0,\lambda_0}}{\longrightarrow}0, \label{c2eql11}\\
&\sum_{k=0}^{n-1}\left(\E\left[\eta_{k,n}\beta_{k,n}\vert \widehat{\mathcal{F}}_{t_k}\right]-\E\left[\eta_{k,n}\vert \widehat{\mathcal{F}}_{t_k}\right]\E\left[\beta_{k,n}\vert \widehat{\mathcal{F}}_{t_k}\right]\right)\overset{\P^{\theta_0,\sigma_0,\lambda_0}}{\longrightarrow}0,\label{c2eql12}\\
&\sum_{k=0}^{n-1}\E\left[\xi_{k,n}^4+\eta_{k,n}^4+\beta_{k,n}^4\vert \widehat{\mathcal{F}}_{t_k}\right]\overset{\P^{\theta_0,\sigma_0,\lambda_0}}{\longrightarrow}0. \label{c2eql3}
\end{align}
\vskip 5pt
{\it Proof of \eqref{c2eql1}.} Using $\E[\Delta B_{t_{k+1}}\vert \widehat{\mathcal{F}}_{t_k}]=0$ and Lemma \ref{c3ergodic}, as $n \rightarrow \infty$,
\begin{align*}
\sum_{k=0}^{n-1}\E\left[\xi_{k,n}\vert \widehat{\mathcal{F}}_{t_k}\right]=-\dfrac{u^2}{2\sigma_0^2}\dfrac{1}{n}\sum_{k=0}^{n-1}X_{t_k}^2&\overset{\P^{\theta_0,\sigma_0,\lambda_0}}{\longrightarrow}-\dfrac{u^2}{2\sigma_0^2}\int_{\R}x^2\pi_{\theta_0,\sigma_0,\lambda_0}(dx).
\end{align*}

On the other hand, using the ergodicity of $X^{\theta_0,\sigma_0,\lambda_0}$ and It\^o's formula, we get that
\begin{equation}\label{e1}\begin{split}
\int_{\R}x^2\pi_{\theta_0,\sigma_0,\lambda_0}(dx)=\lim_{t\to+\infty}\E\left[\left(X_t^{\theta_0,\sigma_0,\lambda_0}\right)^2\right]=\dfrac{1}{2\theta_0}\left(\sigma_0^2+1\right).
\end{split}
\end{equation}

Next, since $\E[(\Delta B_{t_{k+1}})^2\vert \widehat{\mathcal{F}}_{t_k}]=\Delta_n$, we have that as $n \rightarrow \infty$,
\begin{align*}
&\sum_{k=0}^{n-1}\E\left[\eta_{k,n}\vert \widehat{\mathcal{F}}_{t_k}\right]=\dfrac{v}{\sqrt{n}}\sum_{k=0}^{n-1}\int_0^1\dfrac{\sigma_0^2-\sigma(\ell)^2}{\sigma(\ell)^3}d\ell\longrightarrow-\dfrac{v^2}{2}\dfrac{2}{\sigma_0^2}.
\end{align*}

By Lemma \ref{c2lemma8}, we get that as $n\to\infty$,
\begin{align*}
\sum_{k=0}^{n-1}\E\left[\beta_{k,n}\vert \widehat{\mathcal{F}}_{t_k}\right]=-\dfrac{w^2}{2\sigma_0^2}-\dfrac{uw}{\sigma_0^2}\dfrac{1}{n}\sum_{k=0}^{n-1}X_{t_k}-w^2\int_0^1\dfrac{\ell}{\lambda(\ell)}d\ell
\overset{\P^{\theta_0,\sigma_0,\lambda_0}}{\longrightarrow}-\dfrac{w^2}{2\sigma_0^2}-\dfrac{w^2}{2\lambda_0}.
\end{align*}
Here, we have used Lemma \ref{c3ergodic}, the ergodicity of $X^{\theta_0,\sigma_0,\lambda_0}$, and \eqref{solution} to get that as $n \rightarrow \infty$,
\begin{equation}\label{ergodic2}\begin{split}
\dfrac{1}{n}\sum_{k=0}^{n-1}X_{t_k}\overset{\P^{\theta_0,\sigma_0,\lambda_0}}{\longrightarrow}\int_{\R}x\pi_{\theta_0,\sigma_0,\lambda_0}(dx)=\lim_{t\to+\infty}\E\left[X_t^{\theta_0,\sigma_0,\lambda_0}\right]=0.
\end{split}
\end{equation}
Thus, we have shown \eqref{c2eql1}.
\vskip 5pt
{\it Proof of \eqref{c2eql2}.} Observe that as $n \rightarrow \infty$,
\begin{align*}
\sum_{k=0}^{n-1}\left(\E\left[\xi_{k,n}\vert\widehat{\mathcal{F}}_{t_k}\right]\right)^2=\dfrac{u^4}{4\sigma_0^4}\dfrac{1}{n^2}\sum_{k=0}^{n-1}X_{t_k}^4\overset{\P^{\theta_0,\sigma_0,\lambda_0}}{\longrightarrow}0.
\end{align*}

Using properties of the moments of the Brownian motion, Lemma \ref{c3ergodic} and \eqref{e1}, we get 
\begin{align*}
\sum_{k=0}^{n-1}\E\left[\xi_{k,n}^2\vert\widehat{\mathcal{F}}_{t_k}\right]=\dfrac{u^2}{\sigma_0^4n}\sum_{k=0}^{n-1}\left(\sigma_0^2X_{t_k}^2+\dfrac{u^2X_{t_k}^4}{4n}\right)\overset{\P^{\theta_0,\sigma_0,\lambda_0}}{\longrightarrow}\dfrac{u^2}{2\theta_0}\left(1+\dfrac{1}{\sigma_0^2}\right),
\end{align*}
which concludes \eqref{c2eql2}.

\vskip 5pt
{\it Proof of \eqref{c2eql5}.} First, observe that for some constant $C>0$,
\begin{align*}
\sum_{k=0}^{n-1}\left(\E\left[\eta_{k,n}\vert \widehat{\mathcal{F}}_{t_k}\right]\right)^2&=\dfrac{v^2}{n}\sum_{k=0}^{n-1}\left(\int_0^1\dfrac{\sigma_0^2-\sigma(\ell)^2}{\sigma(\ell)^3}d\ell\right)^2 \leq \frac{C}{n}.
\end{align*}

On the other hand, since  $\E[(\Delta B_{t_{k+1}})^2\vert \widehat{\mathcal{F}}_{t_k}]=\Delta_n$ and $\E[(\Delta B_{t_{k+1}})^4\vert \widehat{\mathcal{F}}_{t_k}]=3\Delta_n^2$, we get
\begin{align*}
\sum_{k=0}^{n-1}\E\left[\eta_{k,n}^2\vert \widehat{\mathcal{F}}_{t_k}\right]&=\dfrac{v^2}{n\Delta_n^2}\sum_{k=0}^{n-1}\bigg\{\left(\int_0^1\dfrac{\sigma_0^2}{\sigma(\ell)^3}d\ell\right)^2\E\left[(\Delta B_{t_{k+1}})^4\vert \widehat{\mathcal{F}}_{t_k}\right]+\left(\int_0^1\dfrac{\Delta_n}{\sigma(\ell)}d\ell\right)^2\\
&\qquad-2\int_0^1\dfrac{\Delta_n}{\sigma(\ell)}d\ell\int_0^1\dfrac{\sigma_0^2}{\sigma(\ell)^3}d\ell\E\left[(\Delta B_{t_{k+1}})^2\vert \widehat{\mathcal{F}}_{t_k}\right]\bigg\}\\
&\to v^2\dfrac{2}{\sigma_0^2}, 
\end{align*}
as $n \rightarrow \infty$. This concludes \eqref{c2eql5}.

\vskip 5pt
{\it Proof of \eqref{c2eql8}.} First, using \eqref{e1} and Lemma \ref{c2lemma8}, we get that as $n\to\infty$,
\begin{align*}
\sum_{k=0}^{n-1}\left(\E\left[\beta_{k,n}\vert \widehat{\mathcal{F}}_{t_k}\right]\right)^2=\sum_{k=0}^{n-1}\left(-\dfrac{w^2}{2\sigma_0^2n}-\dfrac{uw}{\sigma_0^2n}X_{t_k}-\dfrac{w^2}{n}\int_0^1\dfrac{\ell}{\lambda(\ell)}d\ell\right)^2\overset{\P^{\theta_0,\sigma_0,\lambda_0}}{\longrightarrow} 0.
\end{align*}

Next, we write $\sum_{k=0}^{n-1}\E[\beta_{k,n}^2\vert\widehat{\mathcal{F}}_{t_k}]=S_{n,1}+S_{n,2}-2S_{n,3}$,
where
\begin{align*}
S_{n,1}:&=\dfrac{w^2}{\sigma_0^4n\Delta_n}\sum_{k=0}^{n-1}\E\left[\left(\sigma_0\Delta B_{t_{k+1}}+\dfrac{w\Delta_n}{2\sqrt{n\Delta_n}}+\dfrac{u\Delta_n}{\sqrt{n\Delta_n}}X_{t_k}\right)^2\big\vert X_{t_k}\right],\\
S_{n,2}:&=\dfrac{w^2}{n\Delta_n}\sum_{k=0}^{n-1}\E\left[\left(\int_0^1\dfrac{1}{\lambda(\ell)}\widetilde{\E}_{X_{t_k}}^{\theta_n,\sigma_0,\lambda(\ell)}\left[\widetilde{M}_{t_{k+1}}^{\lambda(\ell)}-\widetilde{M}_{t_k}^{\lambda(\ell)}\big\vert Y_{t_{k+1}}^{\theta_n,\sigma_0,\lambda(\ell)}=X_{t_{k+1}}\right]d\ell\right)^2\big\vert X_{t_k}\right],\\
S_{n,3}:&=\dfrac{w^2}{\sigma_0^2n\Delta_n}\sum_{k=0}^{n-1}\E\bigg[\left(\sigma_0\Delta B_{t_{k+1}}+\dfrac{w\Delta_n}{2\sqrt{n\Delta_n}}+\dfrac{u\Delta_n}{\sqrt{n\Delta_n}}X_{t_k}\right)\\
&\qquad \times\int_0^1\dfrac{1}{\lambda(\ell)}\widetilde{\E}_{X_{t_k}}^{\theta_n,\sigma_0,\lambda(\ell)}\left[\widetilde{M}_{t_{k+1}}^{\lambda(\ell)}-\widetilde{M}_{t_k}^{\lambda(\ell)}\big\vert Y_{t_{k+1}}^{\theta_n,\sigma_0,\lambda(\ell)}=X_{t_{k+1}}\right]d\ell\big\vert X_{t_k}\bigg].
\end{align*}

Using properties of the moments of the Brownian motion, \eqref{e1} and \eqref{ergodic2}, we get that $S_{n,1}\overset{\P^{\theta_0,\sigma_0,\lambda_0}}{\longrightarrow}\frac{w^2}{\sigma_0^2}$ as $n\to\infty$.

Since $\widetilde{M}_{t_{k+1}}^{\lambda(\ell)}-\widetilde{M}_{t_k}^{\lambda(\ell)}=\Delta M_{t_{k+1}}-\lambda(\ell)\Delta_n$, we write 
$
S_{n,2}=S_{n,2,1}-2S_{n,2,2}+w^2\Delta_n,
$
where
\begin{align*}
S_{n,2,1}:&=\dfrac{w^2}{n\Delta_n}\sum_{k=0}^{n-1}\E\left[\left(\int_0^1\dfrac{1}{\lambda(\ell)}\widetilde{\E}_{X_{t_k}}^{\theta_n,\sigma_0,\lambda(\ell)}\left[\Delta M_{t_{k+1}}\big\vert Y_{t_{k+1}}^{\theta_n,\sigma_0,\lambda(\ell)}=X_{t_{k+1}}\right]d\ell\right)^2\big\vert X_{t_k}\right],\\
S_{n,2,2}:&=\dfrac{w^2}{n}\sum_{k=0}^{n-1}\int_0^1\dfrac{1}{\lambda(\ell)}\E\left[\widetilde{\E}_{X_{t_k}}^{\theta_n,\sigma_0,\lambda(\ell)}\left[\Delta M_{t_{k+1}}\big\vert Y_{t_{k+1}}^{\theta_n,\sigma_0,\lambda(\ell)}=X_{t_{k+1}}\right]\big\vert X_{t_k}\right]d\ell.
\end{align*}

Observe that Lemma \ref{c2lemma8} yields $S_{n,2,2}\overset{\P^{\theta_0,\sigma_0,\lambda_0}}{\longrightarrow}0$. Moreover, adding and subtracting the term $\Delta N_{t_{k+1}}$ inside the integral, we have $S_{n,2,1}=S_{n,2,1,1}+S_{n,2,1,2}-2S_{n,2,1,3}$, where
\begin{align*}
&S_{n,2,1,1}:=\dfrac{w^2}{n\Delta_n}\sum_{k=0}^{n-1}\E\left[\left(\int_0^1\dfrac{1}{\lambda(\ell)}\Delta N_{t_{k+1}}d\ell\right)^2\big\vert X_{t_k}\right],\\
&S_{n,2,1,2}:=\dfrac{w^2}{n\Delta_n}\sum_{k=0}^{n-1}\E\left[\left(\int_0^1\dfrac{1}{\lambda(\ell)}U_kd\ell\right)^2\big\vert X_{t_k}\right],\\
&S_{n,2,1,3}:=\dfrac{w^2}{n\Delta_n}\sum_{k=0}^{n-1}\E\left[\int_0^1\dfrac{1}{\lambda(\ell)}\Delta N_{t_{k+1}}d\ell\int_0^1\dfrac{1}{\lambda(\ell)}U_kd\ell\big\vert X_{t_k}\right].
\end{align*}

Proceeding as in the proof of Lemma \ref{lemma5}, one can show that $S_{n,2,1,2}$ and $S_{n,2,1,3}$ converge to zero in $\P^{\theta_0,\sigma_0,\lambda_0}$-probability as $n\to\infty$. Moreover, since $\E[(\Delta N_{t_{k+1}})^2\vert X_{t_k}]=\lambda_0\Delta_n+(\lambda_0\Delta_n)^2$, we deduce that $S_{n,2,1,1}\overset{\P^{\theta_0,\sigma_0,\lambda_0}}{\longrightarrow}\frac{w^2}{\lambda_0}$,
which implies that $S_{n,2}\overset{\P^{\theta_0,\sigma_0,\lambda_0}}{\longrightarrow}\frac{w^2}{\lambda_0}$ as $n\to\infty$.

Next, we show that 
$S_{n,3}\overset{\P^{\theta_0,\sigma_0,\lambda_0}}{\longrightarrow}0$ as $n\to\infty$. Using Lemma \ref{c2lemma8}, it suffices to show that
\begin{align*}
&S_{n,3,1}=\dfrac{w^2}{\sigma_0n\Delta_n}\sum_{k=0}^{n-1}\int_0^1\dfrac{1}{\lambda(\ell)}\E\left[\Delta B_{t_{k+1}}\widetilde{\E}_{X_{t_k}}^{\theta_n,\sigma_0,\lambda(\ell)}\left[\widetilde{M}_{t_{k+1}}^{\lambda(\ell)}-\widetilde{M}_{t_k}^{\lambda(\ell)}\big\vert Y_{t_{k+1}}^{\theta_n,\sigma_0,\lambda(\ell)}=X_{t_{k+1}}\right]\big\vert X_{t_k}\right]d\ell
\end{align*}
converges to zero in $\P^{\theta_0,\sigma_0,\lambda_0}$-probability. For this, using the independence between $B$ and $N$, and Cauchy-Schwarz inequality, we get that
\begin{align*}
\vert S_{n,3,1}\vert&=\dfrac{w^2}{\sigma_0n\Delta_n}\left\vert\sum_{k=0}^{n-1}\int_0^1\dfrac{1}{\lambda(\ell)}\E\left[\Delta B_{t_{k+1}}U_k\vert X_{t_k}\right]d\ell\right\vert\\
&\leq \dfrac{w^2}{\sigma_0n\sqrt{\Delta_n}}\sum_{k=0}^{n-1}\int_0^1\dfrac{1}{\vert\lambda(\ell)\vert}\left(\E\left[U_k^2\vert X_{t_k}\right]\right)^{1/2}d\ell,
\end{align*}
which converges to zero in $\P^{\theta_0,\sigma_0,\lambda_0}$-probability as $n\to\infty$ by proceeding as in Lemma \ref{lemma5}.

Consequently, the proof of \eqref{c2eql8} is now completed.

\vskip 5pt
{\it Proof of \eqref{c2eql10}.} Observe that
\begin{align*}
\sum_{k=0}^{n-1}\E\left[\xi_{k,n}\eta_{k,n}\vert \widehat{\mathcal{F}}_{t_k}\right]=\sum_{k=0}^{n-1}\E\left[\xi_{k,n}\vert \widehat{\mathcal{F}}_{t_k}\right]\E\left[\eta_{k,n}\vert \widehat{\mathcal{F}}_{t_k}\right]=\dfrac{u^2v^2}{2\sigma_0^2n^2}\sum_{k=0}^{n-1}X_{t_k}^2\int_0^1\dfrac{\sigma_0+\sigma(\ell)}{\sigma(\ell)^3}\ell d\ell.
\end{align*}
Therefore, $\sum_{k=0}^{n-1}(\E[\xi_{k,n}\eta_{k,n}\vert \widehat{\mathcal{F}}_{t_k}]-\E[\xi_{k,n}\vert \widehat{\mathcal{F}}_{t_k}]\E[\eta_{k,n}\vert \widehat{\mathcal{F}}_{t_k}])=0$, for all $n\geq 1$.
\vskip 5pt
{\it Proof of \eqref{c2eql11}.} Using again Lemma \ref{c2lemma8} and Lemma \ref{c3ergodic}, we get that as $n\to\infty$,
\begin{align*}
\sum_{k=0}^{n-1}\E\left[\xi_{k,n}\vert \widehat{\mathcal{F}}_{t_k}\right]\E\left[\beta_{k,n}\vert \widehat{\mathcal{F}}_{t_k}\right]&=\dfrac{u^2w^2}{4\sigma_0^4n^2}\sum_{k=0}^{n-1}X_{t_k}^2+\dfrac{u^3w}{2\sigma_0^4n^2}\sum_{k=0}^{n-1}X_{t_k}^3+\dfrac{u^2w^2}{2\sigma_0^2n}\int_0^1\dfrac{\ell}{\lambda(\ell)}d\ell\\
&\overset{\P^{\theta_0,\sigma_0,\lambda_0}}{\longrightarrow}0.
\end{align*}

Moreover, basic computations yield 
\begin{align*}
&\sum_{k=0}^{n-1}\E\left[\xi_{k,n}\beta_{k,n}\vert \widehat{\mathcal{F}}_{t_k}\right]=\dfrac{uw}{\sigma_0^2n}\sum_{k=0}^{n-1}X_{t_k}+\dfrac{u^2w^2}{4\sigma_0^4n^2}\sum_{k=0}^{n-1}X_{t_k}^2+\dfrac{u^3w}{2\sigma_0^4n^2}\sum_{k=0}^{n-1}X_{t_k}^3\\
&-\dfrac{uw}{\sigma_0n\Delta_n}\sum_{k=0}^{n-1}X_{t_k}\int_0^1\dfrac{1}{\lambda(\ell)} \E\left[\Delta B_{t_{k+1}}\widetilde{\E}_{X_{t_k}}^{\theta_n,\sigma_0,\lambda(\ell)}\left[\widetilde{M}_{t_{k+1}}^{\lambda(\ell)}-\widetilde{M}_{t_k}^{\lambda(\ell)}\big\vert Y_{t_{k+1}}^{\theta_n,\sigma_0,\lambda(\ell)}=X_{t_{k+1}}\right]\big\vert X_{t_k}\right]d\ell\\
&+\dfrac{u^2w^2}{2\sigma_0^2n^2}\sum_{k=0}^{n-1}X_{t_k}^2\int_0^1\dfrac{\ell}{\lambda(\ell)}d\ell,
\end{align*}
which converges to zero in $\P^{\theta_0,\sigma_0,\lambda_0}$-probability as $n\to\infty$. Here, we have used Lemma \ref{c2lemma8}, Lemma \ref{c3ergodic}, \eqref{e1}, \eqref{ergodic2}, and proceeded as for the term $S_{n,3,1}$.

\vskip 5pt
{\it Proof of \eqref{c2eql12}.} Using again Lemma \ref{c2lemma8} and \eqref{ergodic2}, 
\begin{align*}
\sum_{k=0}^{n-1}\E\left[\eta_{k,n}\vert \widehat{\mathcal{F}}_{t_k}\right]\E\left[\beta_{k,n}\vert \widehat{\mathcal{F}}_{t_k}\right]&=\dfrac{w^2v^2}{2\sigma_0^2n}\int_0^1\dfrac{\sigma_0+\sigma(\ell)}{\sigma(\ell)^3}\ell d\ell+\dfrac{uwv^2}{\sigma_0^2n^2}\sum_{k=0}^{n-1}X_{t_k}\int_0^1\dfrac{\sigma_0+\sigma(\ell)}{\sigma(\ell)^3}\ell d\ell\\
&\qquad+\dfrac{w^2v^2}{n}\int_0^1\dfrac{\sigma_0+\sigma(\ell)}{\sigma(\ell)^3}\ell d\ell\int_0^1\dfrac{\ell}{\lambda(\ell)}d\ell,
 \end{align*}
which converges to zero in $\P^{\theta_0,\sigma_0,\lambda_0}$-probability as $n\to\infty$. Next,
\begin{align*}
&\sum_{k=0}^{n-1}\E\left[\eta_{k,n}\beta_{k,n}\vert \widehat{\mathcal{F}}_{t_k}\right]=\dfrac{w^2v^2}{2\sigma_0^2n}\int_0^1\dfrac{\sigma_0+\sigma(\ell)}{\sigma(\ell)^3}\ell d\ell+\dfrac{uwv^2}{\sigma_0^2n^2}\sum_{k=0}^{n-1}X_{t_k}\int_0^1\dfrac{\sigma_0+\sigma(\ell)}{\sigma(\ell)^3}\ell d\ell\\
&+\dfrac{w^2v}{\sqrt{n}}\int_0^1\dfrac{d\ell}{\sigma(\ell)}\int_0^1\dfrac{\ell}{\lambda(\ell)}d\ell+\dfrac{uw\sigma_0^2}{n\Delta_n\sqrt{\Delta_n}}\int_0^1\dfrac{d\ell}{\sigma(\ell)^3}\\
&\times\sum_{k=0}^{n-1}\int_0^1\dfrac{1}{\lambda(\ell)} \E\left[(\Delta B_{t_{k+1}})^2\widetilde{\E}_{X_{t_k}}^{\theta_n,\sigma_0,\lambda(\ell)}\left[\widetilde{M}_{t_{k+1}}^{\lambda(\ell)}-\widetilde{M}_{t_k}^{\lambda(\ell)}\big\vert Y_{t_{k+1}}^{\theta_n,\sigma_0,\lambda(\ell)}=X_{t_{k+1}}\right]\big\vert X_{t_k}\right]d\ell,
\end{align*}
which converges to zero in $\P^{\theta_0,\sigma_0,\lambda_0}$-probability as $n\to\infty$. Here, we have used Lemma \ref{c2lemma8}, \eqref{ergodic2}, and proceeded as for the term $S_{n,3,1}$.

\vskip 5pt
{\it Proof of \eqref{c2eql3}.} Using Lemma \ref{c3ergodic}, as $n\to\infty$,
\begin{align*}
&\sum_{k=0}^{n-1}\E\left[\xi_{k,n}^4\vert \widehat{\mathcal{F}}_{t_k}\right]\leq\dfrac{8u^4}{\sigma_0^8n^2}\sum_{k=0}^{n-1}X_{t_k}^4\left(3\sigma_0^4+\dfrac{u^4X_{t_k}^4}{16n^2}\right)\overset{\P^{\theta_0,\sigma_0,\lambda_0}}{\longrightarrow}0.
\end{align*}

Next, it is easy to check that for some constant $C>0$,
$$
\sum_{k=0}^{n-1}\E\left[\eta_{k,n}^4\vert \widehat{\mathcal{F}}_{t_k}\right]\leq \dfrac{C}{n}.
$$

Finally, applying Jensen's inequality and Lemma \ref{c3ergodic}, we get that
\begin{align*}
&\sum_{k=0}^{n-1}\E\left[\beta_{k,n}^4\vert \widehat{\mathcal{F}}_{t_k}\right]\leq \dfrac{8w^4}{n^2\Delta_n^2\sigma_0^8}\sum_{k=0}^{n-1}\E\left[\left(\sigma_0\Delta B_{t_{k+1}}+\dfrac{w\Delta_n}{2\sqrt{n\Delta_n}}+\dfrac{u\Delta_n}{\sqrt{n\Delta_n}}X_{t_k}\right)^4\big\vert X_{t_k}\right]\\
&\qquad+\dfrac{8w^4}{n^2\Delta_n^2}\sum_{k=0}^{n-1}\int_0^1\dfrac{1}{\lambda(\ell)^4}\E\left[\widetilde{\E}_{X_{t_k}}^{\theta_n,\sigma_0,\lambda(\ell)}\left[\left(\widetilde{M}_{t_{k+1}}^{\lambda(\ell)}-\widetilde{M}_{t_k}^{\lambda(\ell)}\right)^4\big\vert Y_{t_{k+1}}^{\theta_n,\sigma_0,\lambda(\ell)}=X_{t_{k+1}}\right]\big\vert X_{t_k}\right]d\ell,
\end{align*}
which converges to zero in $\P^{\theta_0,\sigma_0,\lambda_0}$-probability as $n\to\infty$, since $\E[(\Delta B_{t_{k+1}})^4\vert X_{t_k}]=3\Delta_n^2$ and 
$$
\E\left[\widetilde{\E}_{X_{t_k}}^{\theta_n,\sigma_0,\lambda(\ell)}\left[\left(\widetilde{M}_{t_{k+1}}^{\lambda(\ell)}-\widetilde{M}_{t_k}^{\lambda(\ell)}\right)^4\big\vert Y_{t_{k+1}}^{\theta_n,\sigma_0,\lambda(\ell)}=X_{t_{k+1}}\right]\big\vert X_{t_k}\right]\leq
C\Delta_n,
$$
for some constant $C>0$ and $n$ large enough, using the same arguments as in Lemma \ref{c2lemma8}.
\end{proof}

Consequently, from Lemmas \ref{expansion}-\ref{main}, the proof of Theorem \ref{c2theorem} is now completed.

\section{Conclusion}
Considering the SDEs whose jump part is characterized by the stable laws makes the problem simpler to handle thanks to the semi-explicitness of the density of such processes (see A\"it-Sahalia and Jacod \cite{AJ07}, and Cl\'ement and Gloter \cite{CG15}). In our context, we have shown that the Malliavin calculus is a powerful tool for the stochastic analysis of the log-likelihood ratio of diffusions with jumps. Besides, we need to condition on the jump structure and use large deviation type results. We believe that the argument we introduced here can be extended to more general cases where the transition density is not explicit, with further arguments. In fact, in \cite{KNT15} we have used the same methodology presented here to treat a multidimensional  ergodic diffusion with jumps whose unknown parameter appears only in the drift coefficient. However, there is an extension of the result of this paper that we should think about in our future work. As we mentioned in the Introduction, the case of general SDEs with jumps whose unknown parameters appear in the drift and diffusion coefficients and in the jump component remains an open and difficult problem. This issue will be treated in future research.

\section{Appendix}

\subsection{Transition density conditioned on the jump structure}
\label{condensity}

For any $t>s$ and $j\geq 0$, we denote by $q_{(j)}^{\theta,\sigma,\lambda}(t-s,x,y)$ the transition density of $X_t^{\theta,\sigma,\lambda}$ conditioned on $X_s^{\theta,\sigma,\lambda}=x$ and $N_t-N_s=j$. The convolution formula for the sum of independent random variables yields
\begin{equation}\label{density}\begin{split}
p^{\theta,\sigma,\lambda}(t-s,x,y)=\sum_{j=0}^{\infty}q_{(j)}^{\theta,\sigma,\lambda}(t-s,x,y)e^{-\lambda (t-s)}\frac{(\lambda(t-s))^j}{j!}.
\end{split}
\end{equation}
First, using equation \eqref{solution}, we have that
\begin{align}
&q_{(0)}^{\theta,\sigma,\lambda}(t-s,x,y)=\P\left(X_t^{\theta,\sigma,\lambda}=y\vert X_s^{\theta,\sigma,\lambda}=x, N_t-N_s=0\right)\notag\\
&=\P\left(\sigma\int_s^te^{-\theta (t-u)}dB_u=y-xe^{-\theta (t-s)}+\frac{\lambda}{\theta}(1-e^{-\theta(t-s)})\vert X_s^{\theta,\sigma,\lambda}=x, N_t-N_s=0\right)\notag\\
&=\dfrac{1}{\sqrt{\frac{\pi}{\theta}\sigma^2(1-e^{-2\theta(t-s)})}}\exp\left\{-\dfrac{\left(y-xe^{-\theta(t-s)}+\frac{\lambda}{\theta}(1-e^{-\theta(t-s)})\right)^2}{\frac{1}{\theta}\sigma^2(1-e^{-2\theta(t-s)})}\right\}.\label{q0}
\end{align}
For any $j\geq 1$ and $s<s_1<\cdots<s_j<t$, we denote by $q_{(j)}^{\theta,\sigma,\lambda}(t-s,x,y;s_1,\ldots,s_j)$ the transition density of $X_t^{\theta,\sigma,\lambda}$ conditioned on $X_s^{\theta,\sigma,\lambda}=x$, $N_t-N_s=j$ and $(T_1^{s,t}=s_1, \ldots, T_j^{s,t}=s_j)$, where $T_1^{s,t}, T_2^{s,t},\ldots$ with $s<T_1^{s,t}<T_2^{s,t}<\cdots<t$ are the jump times on $[s,t]$ of the Poisson process $N$. Given $N_t-N_s=j$, by conditioning on the jump times $(T_1^{s,t},\ldots,T_j^{s,t})$ which are then distributed as the order statistics of $j$ independent uniform random variables on $[s,t]$, we have 
\begin{align}
&q_{(j)}^{\theta,\sigma,\lambda}(t-s,x,y)=\P\left(X_t^{\theta,\sigma,\lambda}=y\vert X_s^{\theta,\sigma,\lambda}=x, N_t-N_s=j\right)\notag\\
&=\int_{\{s<s_1<\cdots<s_j<t\}}\P\left(X_t^{\theta,\sigma,\lambda}=y\vert X_s^{\theta,\sigma,\lambda}=x, N_t-N_s=j, T_1^{s,t}=s_1, \ldots, T_j^{s,t}=s_j\right)\notag\\
&\qquad\times\P\left(T_1^{s,t}\in ds_1,\ldots,T_j^{s,t}\in ds_j\vert X_s^{\theta,\sigma,\lambda}=x, N_t-N_s=j\right)\notag\\
&=\dfrac{j!}{(t-s)^j}\int_{\{s<s_1<\cdots<s_j<t\}}q_{(j)}^{\theta,\sigma,\lambda}(t-s,x,y;s_1,\ldots,s_j)ds_1\cdots ds_j.\label{qjf}
\end{align}
Moreover, using again equation \eqref{solution}, definition of stochastic integral w.r.t. Poisson process and setting $\Sigma:=\frac{1}{\theta}\sigma^2(1-e^{-2\theta(t-s)})$, we get that
\begin{align}
&q_{(j)}^{\theta,\sigma,\lambda}(t-s,x,y;s_1,\ldots,s_j)=\P\left(X_t^{\theta,\sigma,\lambda}=y\vert X_s^{\theta,\sigma,\lambda}=x, N_t-N_s=j, T_1^{s,t}=s_1, \ldots, T_j^{s,t}=s_j\right)\notag\\
&=\P\bigg(\sigma\int_s^te^{-\theta (t-u)}dB_u+\int_s^te^{-\theta (t-u)}dN_u=y-xe^{-\theta (t-s)}\notag\\
&\qquad+\frac{\lambda}{\theta}(1-e^{-\theta(t-s)})\vert X_s^{\theta,\sigma,\lambda}=x, N_t-N_s=j, T_1^{s,t}=s_1, \ldots, T_j^{s,t}=s_j\bigg)\notag\\
&=\P\bigg(\sigma\int_s^te^{-\theta (t-u)}dB_u=y-xe^{-\theta (t-s)}+\frac{\lambda}{\theta}(1-e^{-\theta(t-s)})\notag\\
&\qquad-\left(e^{-\theta (t-s_1)}+\cdots+e^{-\theta (t-s_j)}\right)\vert X_s^{\theta,\sigma,\lambda}=x, N_t-N_s=j, T_1^{s,t}=s_1, \ldots, T_j^{s,t}=s_j\bigg)\notag\\
&=\dfrac{1}{\sqrt{\pi\Sigma}}\exp\left\{-\dfrac{\left(y-xe^{-\theta(t-s)}+\frac{\lambda}{\theta}(1-e^{-\theta(t-s)})-(e^{-\theta (t-s_1)}+\cdots+e^{-\theta (t-s_j)})\right)^2}{\Sigma}\right\}.\label{qj}
\end{align}

For $k \in \{0,...,n-1\}$, consider the events $\widehat{J}_{j,k}:=\{\Delta N_{t_{k+1}}=j\}$ and $\widetilde{J}_{j,k}:=\{\Delta M_{t_{k+1}}=j\}$. By conditioning on all the possible number of jumps and jump times of the Poisson process occurring on $[t_k,t_{k+1}]$, we have the following change of measures via the transition density conditioned on the number of jumps and jump times.
\begin{lemma}\label{change} Let $f$ be any bounded function. For any $k \in \{0,...,n-1\}$, $(\theta,\sigma,\lambda)\in \Theta\times\Sigma\times\Lambda$, 
\begin{equation*} \begin{split}
\E\left[f(X_{t_{k+1}})\vert X_{t_{k}}\right]=\E\left[f(X_{t_{k+1}}^{\theta,\sigma,\lambda})D(\Delta_n,X_{t_k}^{\theta,\sigma,\lambda},X^{\theta,\sigma,\lambda}_{t_{k+1}})\big\vert X_{t_k}^{\theta,\sigma,\lambda}=X_{t_k}\right],
\end{split}
\end{equation*}
where, setting $\Sigma_k^j:=\{(s_1,\ldots,s_j): t_k<s_1<\cdots<s_j<t_{k+1}\}$,
\begin{align*}
&D(\Delta_n,X_{t_k}^{\theta,\sigma,\lambda},X^{\theta,\sigma,\lambda}_{t_{k+1}}):=e^{(\lambda-\lambda_0)\Delta_n}\bigg({\bf{1}}_{\widehat{J}_{0,k}}\frac{q_{(0)}^{\theta_0,\sigma_0,\lambda_0}}{q_{(0)}^{\theta,\sigma,\lambda}}(\Delta_n,X_{t_k}^{\theta,\sigma,\lambda},X^{\theta,\sigma,\lambda}_{t_{k+1}})\\
&\qquad+\sum_{j=1}^{\infty}\left(\frac{\lambda_0}{\lambda}\right)^j\int_{\Sigma_k^j}{\bf{1}}_{\{\widehat{J}_{j,k},s_1,\ldots,s_j\}}\frac{q_{(j)}^{\theta_0,\sigma_0,\lambda_0}}{q_{(j)}^{\theta,\sigma,\lambda}}(\Delta_n,X_{t_k}^{\theta,\sigma,\lambda},X^{\theta,\sigma,\lambda}_{t_{k+1}};s_1,\ldots,s_j)ds_1\cdots ds_j\bigg),
\end{align*}
and $\{\widehat{J}_{j,k},s_1,\ldots,s_j\}:=\{\Delta N_{t_{k+1}}=j, T_1^{t_k,t_{k+1}}=s_1, \ldots, T_j^{t_k,t_{k+1}}=s_j\}$.
\end{lemma}
\begin{proof}
Observe that
\begin{align*}
&\E\left[f(X_{t_{k+1}})\vert X_{t_{k}}\right]=\E\left[{\bf{1}}_{\widehat{J}_{0,k}}f(X_{t_{k+1}})\vert X_{t_{k}}\right]+\sum_{j=1}^{\infty}\int_{\Sigma_k^j}\E\left[{\bf{1}}_{\{\widehat{J}_{j,k},s_1,\ldots,s_j\}}f(X_{t_{k+1}})\vert X_{t_{k}}\right]ds_1\cdots ds_j\\
&=\int_{\R}f(y)q_{(0)}^{\theta_0,\sigma_0,\lambda_0}(\Delta_n,X_{t_k},y)dye^{-\lambda_0\Delta_n}\\
&\qquad+\sum_{j=1}^{\infty}\int_{\Sigma_k^j}\int_{\R}f(y)q_{(j)}^{\theta_0,\sigma_0,\lambda_0}(\Delta_n,X_{t_k},y;s_1,\ldots,s_j)dye^{-\lambda_0\Delta_n}\frac{(\lambda_0\Delta_n)^j}{j!}\frac{j!}{\Delta_n^j}ds_1\cdots ds_j\\
&=e^{(\lambda-\lambda_0)\Delta_n}\bigg(\int_{\R}f(y)\frac{q_{(0)}^{\theta_0,\sigma_0,\lambda_0}}{q_{(0)}^{\theta,\sigma,\lambda}}q_{(j)}^{\theta,\sigma,\lambda}(\Delta_n,X_{t_k},y)dye^{-\lambda\Delta_n}\\
&+\sum_{j=1}^{\infty}\left(\frac{\lambda_0}{\lambda}\right)^j\int_{\Sigma_k^j}\int_{\R}f(y)\frac{q_{(j)}^{\theta_0,\sigma_0,\lambda_0}}{q_{(j)}^{\theta,\sigma,\lambda}}q_{(j)}^{\theta,\sigma,\lambda}(\Delta_n,X_{t_k},y;s_1,\ldots,s_j)dye^{-\lambda\Delta_n}\frac{(\lambda\Delta_n)^j}{j!}\frac{j!}{\Delta_n^j}ds_1\cdots ds_j\bigg),
\end{align*}
which gives the desired result.
\end{proof}

Next, we have the first crucial estimate.
\begin{lemma}\label{lemma12} Let $(\theta,\sigma,\lambda), (\overline{\theta},\overline{\sigma},\overline{\lambda})\in \Theta\times\Sigma\times\Lambda$, and set $\sigma_1:=\frac{\overline{\sigma}^2}{\overline{\theta}}(1-e^{-2\overline{\theta}\Delta_n})$ and $\sigma_2:=\frac{\sigma^2}{\theta}(1-e^{-2\theta\Delta_n})$. Then there exists a constant $C>0$ such that for all $k \in \{0,...,n-1\}$ and $j\geq 1$, the following estimates
\begin{align*}
&\E\left[{\bf{1}}_{\{\widehat{J}_{j,k},s_1,\ldots,s_j\}}\left(\frac{q_{(j)}^{\overline{\theta},\overline{\sigma},\overline{\lambda}}}{q_{(j)}^{\theta,\sigma,\lambda}}(\Delta_n,X_{t_k}^{\theta,\sigma,\lambda},X_{t_{k+1}}^{\theta,\sigma,\lambda};s_1,\ldots,s_j)\right)^{p}\big\vert X_{t_k}^{\theta,\sigma,\lambda}=x\right]\leq Ce^{-\lambda\Delta_n}\lambda^j,\\
&\E\left[{\bf{1}}_{\widehat{J}_{0,k}}\left(\frac{q_{(0)}^{\overline{\theta},\overline{\sigma},\overline{\lambda}}}{q_{(0)}^{\theta,\sigma,\lambda}}(\Delta_n,X_{t_k}^{\theta,\sigma,\lambda},X_{t_{k+1}}^{\theta,\sigma,\lambda})\right)^{p}\big\vert X_{t_k}^{\theta,\sigma,\lambda}=x\right]\leq C,
\end{align*}
hold for any $p>1$ if $\sigma_2\geq\sigma_1$, and for any $p\in (1,\frac{\sigma_1}{\sigma_1-\sigma_2})$ if $\sigma_2<\sigma_1$. Moreover, this statement remains valid for $Y^{\theta,\sigma,\lambda}$.
\end{lemma}
\begin{proof}
Using \eqref{qj}, we get that for any $p>1$ if $\sigma_2\geq\sigma_1$, and $p\in (1,\frac{\sigma_1}{\sigma_1-\sigma_2})$ if $\sigma_2<\sigma_1$,
\begin{align*}
&\E\left[{\bf{1}}_{\{\widehat{J}_{j,k},s_1,\ldots,s_j\}}\left(\frac{q_{(j)}^{\overline{\theta},\overline{\sigma},\overline{\lambda}}}{q_{(j)}^{\theta,\sigma,\lambda}}(\Delta_n,X_{t_k}^{\theta,\sigma,\lambda},X_{t_{k+1}}^{\theta,\sigma,\lambda};s_1,\ldots,s_j)\right)^{p}\big\vert X_{t_k}^{\theta,\sigma,\lambda}=x\right]\\
&=e^{-\lambda\Delta_n}\frac{(\lambda\Delta_n)^j}{j!}\frac{j!}{\Delta_n^j}\int_{\R}q_{(j)}^{\overline{\theta},\overline{\sigma},\overline{\lambda}}(\Delta_n,x,y;s_1,\ldots,s_j)^{p}q_{(j)}^{\theta,\sigma,\lambda}(\Delta_n,x,y;s_1,\ldots,s_j)^{1-p}dy\\
&=e^{-\lambda\Delta_n}\lambda^j\int_{\R}\dfrac{1}{(\pi\sigma_1)^{\frac{p}{2}}}\dfrac{1}{(\pi\sigma_2)^{\frac{1-p}{2}}}\\
&\quad\times\exp\left\{-p\frac{\left(y-xe^{-\overline{\theta}\Delta_n}+\frac{\overline{\lambda}}{\overline{\theta}}(1-e^{-\overline{\theta}\Delta_n})-(e^{-\overline{\theta}(t_{k+1}-s_1)}+\cdots+e^{-\overline{\theta} (t_{k+1}-s_j)})\right)^2}{\sigma_1}\right\}\\
&\quad\times\exp\left\{-(1-p)\frac{\left(y-xe^{-\theta\Delta_n}+\frac{\lambda}{\theta}(1-e^{-\theta \Delta_n})-(e^{-\theta (t_{k+1}-s_1)}+\cdots+e^{-\theta (t_{k+1}-s_j)})\right)^2}{\sigma_2}\right\}dy\\
&\leq Ce^{-\lambda\Delta_n}\lambda^j,
\end{align*}
for some constant $C>0$, where we use the fact that the $dy$ integral is finite since $\frac{p}{\sigma_1}+\frac{1-p}{\sigma_2}>0$ for any $p>1$ if $\sigma_2\geq\sigma_1$, and $p\in (1,\frac{\sigma_1}{\sigma_1-\sigma_2})$ if $\sigma_2<\sigma_1$. This shows the first inequality. Using \eqref{q0} and the same arguments as above, we conclude the second inequality.
\end{proof}
As in Propositions \ref{c2prop1} and \ref{c2pro2}, we have the following explicit expression for the logarithm derivatives of the transition density conditioned on the number of jumps and jump times.
\begin{lemma}\label{lemma13} For all $(\theta,\sigma,\lambda) \in \Theta\times\Sigma\times\Lambda$, $k \in \{0,...,n-1\}$, $j\geq 1$, $\beta\in\{\theta, \sigma\}$, and $x,y\in\R$,
\begin{align*}
&\dfrac{\partial_{\beta}q_{(j)}^{\theta,\sigma,\lambda}}{q_{(j)}^{\theta,\sigma,\lambda}}(\Delta_n,x,y;s_1,\ldots,s_j)=\dfrac{1}{\Delta_n}\widetilde{\E}_{x}^{\theta,\sigma,\lambda}\left[\delta\left(\partial_{\beta}Y_{t_{k+1}}^{\theta,\sigma,\lambda}(t_k,x)U^{\theta,\sigma,\lambda}(t_k,x)\right)\big\vert Y_{t_{k+1}}^{\theta,\sigma,\lambda}=y,\widetilde{J}_{j,k},s_1,\ldots,s_j\right],\\
&\dfrac{\partial_{\lambda}q_{(j)}^{\theta,\sigma,\lambda}}{q_{(j)}^{\theta,\sigma,\lambda}}(\Delta_n,x,y;s_1,\ldots,s_j)=\widetilde{\E}_{x}^{\theta,\sigma,\lambda}\left[-\dfrac{\Delta W_{t_{k+1}}}{\sigma}+\dfrac{\widetilde{M}_{t_{k+1}}^{\lambda}-\widetilde{M}_{t_k}^{\lambda}}{\lambda}\big\vert Y_{t_{k+1}}^{\theta,\sigma,\lambda}=y,\widetilde{J}_{j,k},s_1,\ldots,s_j\right],\\
&\dfrac{\partial_{\beta}q_{(0)}^{\theta,\sigma,\lambda}}{q_{(0)}^{\theta,\sigma,\lambda}}(\Delta_n,x,y)=\dfrac{1}{\Delta_n}\widetilde{\E}_{x}^{\theta,\sigma,\lambda}\left[\delta\left(\partial_{\beta}Y_{t_{k+1}}^{\theta,\sigma,\lambda}(t_k,x)U^{\theta,\sigma,\lambda}(t_k,x)\right)\big\vert Y_{t_{k+1}}^{\theta,\sigma,\lambda}=y,\widetilde{J}_{0,k}\right],\\
&\dfrac{\partial_{\lambda}q_{(0)}^{\theta,\sigma,\lambda}}{q_{(0)}^{\theta,\sigma,\lambda}}(\Delta_n,x,y)=\widetilde{\E}_{x}^{\theta,\sigma,\lambda}\left[-\dfrac{\Delta W_{t_{k+1}}}{\sigma}+\dfrac{\widetilde{M}_{t_{k+1}}^{\lambda}-\widetilde{M}_{t_k}^{\lambda}}{\lambda}\big\vert Y_{t_{k+1}}^{\theta,\sigma,\lambda}=y,\widetilde{J}_{0,k}\right],
\end{align*}
where the process $U^{\theta,\sigma,\lambda}(t_k,x)=(U^{\theta,\sigma,\lambda}_t(t_k,x), t\in[t_k,t_{k+1}])$ is defined in Proposition \ref{c2prop1}.
\end{lemma}
\begin{proof} Let $f$ be a continuously differentiable function with compact support. The chain rule of the Malliavin calculus implies that $
f'(Y_{t_{k+1}}^{\theta,\sigma,\lambda}(t_k,x))=D_t(f(Y_{t_{k+1}}^{\theta,\sigma,\lambda}(t_k,x)))U^{\theta,\sigma,\lambda}_t(t_k,x)$, for all $(\theta,\sigma,\lambda)\in \Theta\times\Sigma\times\Lambda$ and $t\in [t_k,t_{k+1}]$, where $U^{\theta,\sigma,\lambda}_t(t_k,x):=(D_tY_{t_{k+1}}^{\theta,\sigma,\lambda}(t_k,x))^{-1}$.

Using the Malliavin calculus integration by parts formula on $[t_k,t_{k+1}]$, and the independence between $W$ and $(M,T_1^{t_k,t_{k+1}}, \ldots, T_j^{t_k,t_{k+1}})$, we get that
\begin{align*}
&\partial_{\beta}\widetilde{\E}\left[{\bf{1}}_{\{\widetilde{J}_{j,k},s_1,\ldots,s_j\}}f(Y_{t_{k+1}}^{\theta,\sigma,\lambda}(t_k,x))\right]=\widetilde{\E}\left[{\bf{1}}_{\{\widetilde{J}_{j,k},s_1,\ldots,s_j\}}f'(Y_{t_{k+1}}^{\theta,\sigma,\lambda}(t_k,x))\partial_{\beta}Y_{t_{k+1}}^{\theta,\sigma,\lambda}(t_k,x)\right]\\
&\qquad=\dfrac{1}{\Delta_n}\widetilde{\E}\left[\int_{t_k}^{t_{k+1}}D_t(f(Y_{t_{k+1}}^{\theta,\sigma,\lambda}(t_k,x)))U^{\theta,\sigma,\lambda}_t(t_k,x)\partial_{\beta}Y_{t_{k+1}}^{\theta,\sigma,\lambda}(t_k,x){\bf{1}}_{\{\widetilde{J}_{j,k},s_1,\ldots,s_j\}}dt\right]\\
&\qquad=\dfrac{1}{\Delta_n}\widetilde{\E}\left[{\bf{1}}_{\{\widetilde{J}_{j,k},s_1,\ldots,s_j\}}f(Y_{t_{k+1}}^{\theta,\sigma,\lambda}(t_k,x))\delta\left(\partial_{\beta}Y_{t_{k+1}}^{\theta,\sigma,\lambda}(t_k,x)U^{\theta,\sigma,\lambda}(t_k,x)\right)\right].
\end{align*}

On the other hand, using the stochastic flow property $Y_t^{\theta,\sigma,\lambda}=Y_t^{\theta,\sigma,\lambda}(s,Y_s^{\theta,\sigma,\lambda})$ for all $0\leq s\leq t$, and the Markov property of diffusion processes, we have that
\begin{equation*}\begin{split}
\partial_{\beta}\widetilde{\E}\left[{\bf{1}}_{\{\widetilde{J}_{j,k},s_1,\ldots,s_j\}}f(Y_{t_{k+1}}^{\theta,\sigma,\lambda}(t_k,x))\right]=\int_{\R}f(y)\partial_{\beta}q_{(j)}^{\theta,\sigma,\lambda}(\Delta_n,x,y;s_1,\ldots,s_j)e^{-\lambda\Delta_n}\frac{(\lambda\Delta_n)^j}{j!}\frac{j!}{\Delta_n^j}dy,
\end{split}
\end{equation*}
and
\begin{align*}
&\widetilde{\E}\left[{\bf{1}}_{\{\widetilde{J}_{j,k},s_1,\ldots,s_j\}}f(Y_{t_{k+1}}^{\theta,\sigma,\lambda}(t_k,x))\delta\left(\partial_{\beta}Y_{t_{k+1}}^{\theta,\sigma,\lambda}(t_k,x)U^{\theta,\sigma,\lambda}(t_k,x)\right)\right]\\
&=\int_{\R}f(y)\widetilde{\E}\left[\delta\left(\partial_{\beta}Y_{t_{k+1}}^{\theta,\sigma,\lambda}(t_k,x)U^{\theta,\sigma,\lambda}(t_k,x)\right)\big\vert Y_{t_{k+1}}^{\theta,\sigma,\lambda}=y,Y_{t_{k}}^{\theta,\sigma,\lambda}=x,\widetilde{J}_{j,k},s_1,\ldots,s_j\right]\\
&\qquad\times q_{(j)}^{\theta,\sigma,\lambda}(\Delta_n,x,y;s_1,\ldots,s_j)e^{-\lambda\Delta_n}\frac{(\lambda\Delta_n)^j}{j!}\frac{j!}{\Delta_n^j}dy.
\end{align*}
This shows the first equality. Moreover, using the same above arguments and proceeding similarly as in the proof of Proposition \ref{c2pro2}, we derive the other expressions.
\end{proof}

As a consequence, we have the following estimates.
\begin{lemma}\label{lemma14} Let $(\theta,\sigma,\lambda), (\overline{\theta},\overline{\sigma},\overline{\lambda})\in \Theta\times\Sigma\times\Lambda$, and set $\sigma_1:=\frac{\overline{\sigma}^2}{\overline{\theta}}(1-e^{-2\overline{\theta}\Delta_n})$ and $\sigma_2:=\frac{\sigma^2}{\theta}(1-e^{-2\theta\Delta_n})$. Then for any $p>1$, $p_1$ and $p_2$ conjugate with $p_1>1$ if $\sigma_2\geq\sigma_1$, and $p_1\in (1,\frac{\sigma_1}{\sigma_1-\sigma_2})$ if $\sigma_2<\sigma_1$, there exist constants $C, q>0$ such that for all $k \in \{0,...,n-1\}$ and $j\geq 1$, 
\begin{align*}
&\E\left[{\bf{1}}_{\{\widehat{J}_{j,k},s_1,\ldots,s_j\}}\left(\frac{\partial_{\theta}q_{(j)}^{\theta,\sigma,\lambda}}{q_{(j)}^{\theta,\sigma,\lambda}}(\Delta_n,X_{t_k}^{\overline{\theta},\overline{\sigma},\overline{\lambda}},X_{t_{k+1}}^{\overline{\theta},\overline{\sigma},\overline{\lambda}};s_1,\ldots,s_j)\right)^{p}\big\vert X_{t_{k}}^{\overline{\theta},\overline{\sigma},\overline{\lambda}}=x\right]\\
&\qquad\leq C\Delta_n^{\frac{p}{2}}e^{(\lambda-\overline{\lambda})\Delta_n}\left(\frac{\overline{\lambda}}{\lambda}\right)^j(e^{-\lambda\Delta_n}\lambda^j)^{\frac{1}{p_1}}\left(1+\vert x\vert^q\right),\\
&\E\left[{\bf{1}}_{\{\widehat{J}_{j,k},s_1,\ldots,s_j\}}\left(\frac{\partial_{\sigma}q_{(j)}^{\theta,\sigma,\lambda}}{q_{(j)}^{\theta,\sigma,\lambda}}(\Delta_n,X_{t_k}^{\overline{\theta},\overline{\sigma},\overline{\lambda}},X_{t_{k+1}}^{\overline{\theta},\overline{\sigma},\overline{\lambda}};s_1,\ldots,s_j)\right)^{p}\big\vert X_{t_{k}}^{\overline{\theta},\overline{\sigma},\overline{\lambda}}=x\right]\\
&\qquad\leq Ce^{(\lambda-\overline{\lambda})\Delta_n}\left(\frac{\overline{\lambda}}{\lambda}\right)^j(e^{-\lambda\Delta_n}\lambda^j)^{\frac{1}{p_1}}\left(1+\vert x\vert^q\right),\\
&\E\left[{\bf{1}}_{\{\widehat{J}_{j,k},s_1,\ldots,s_j\}}\left(\frac{\partial_{\lambda}q_{(j)}^{\theta,\sigma,\lambda}}{q_{(j)}^{\theta,\sigma,\lambda}}(\Delta_n,X_{t_k}^{\overline{\theta},\overline{\sigma},\overline{\lambda}},X_{t_{k+1}}^{\overline{\theta},\overline{\sigma},\overline{\lambda}};s_1,\ldots,s_j)\right)^{p}\big\vert X_{t_{k}}^{\overline{\theta},\overline{\sigma},\overline{\lambda}}=x\right]\\
&\qquad\leq C\Delta_n^{\frac{1}{p_2}}e^{(\lambda-\overline{\lambda})\Delta_n}\left(\frac{\overline{\lambda}}{\lambda}\right)^j(e^{-\lambda\Delta_n}\lambda^j)^{\frac{1}{p_1}},\\
&\E\left[{\bf{1}}_{\widehat{J}_{0,k}}\left(\frac{\partial_{\theta}q_{(0)}^{\theta,\sigma,\lambda}}{q_{(0)}^{\theta,\sigma,\lambda}}(\Delta_n,X_{t_k}^{\overline{\theta},\overline{\sigma},\overline{\lambda}},X_{t_{k+1}}^{\overline{\theta},\overline{\sigma},\overline{\lambda}})\right)^{p}\big\vert X_{t_{k}}^{\overline{\theta},\overline{\sigma},\overline{\lambda}}=x\right]\leq C\Delta_n^{\frac{p}{2}}e^{(\lambda-\overline{\lambda})\Delta_n}\left(1+\vert x\vert^q\right),\\
&\E\left[{\bf{1}}_{\widehat{J}_{0,k}}\left(\frac{\partial_{\sigma}q_{(0)}^{\theta,\sigma,\lambda}}{q_{(0)}^{\theta,\sigma,\lambda}}(\Delta_n,X_{t_k}^{\overline{\theta},\overline{\sigma},\overline{\lambda}},X_{t_{k+1}}^{\overline{\theta},\overline{\sigma},\overline{\lambda}})\right)^{p}\big\vert X_{t_{k}}^{\overline{\theta},\overline{\sigma},\overline{\lambda}}=x\right]\leq Ce^{(\lambda-\overline{\lambda})\Delta_n}\left(1+\vert x\vert^q\right),\\
&\E\left[{\bf{1}}_{\widehat{J}_{0,k}}\left(\frac{\partial_{\lambda}q_{(0)}^{\theta,\sigma,\lambda}}{q_{(0)}^{\theta,\sigma,\lambda}}(\Delta_n,X_{t_k}^{\overline{\theta},\overline{\sigma},\overline{\lambda}},X_{t_{k+1}}^{\overline{\theta},\overline{\sigma},\overline{\lambda}})\right)^{p}\big\vert X_{t_{k}}^{\overline{\theta},\overline{\sigma},\overline{\lambda}}=x\right]\leq C\Delta_n^{\frac{1}{p_2}}e^{(\lambda-\overline{\lambda})\Delta_n}.
\end{align*}
\end{lemma}
\begin{proof} Applying Lemma \ref{lemma13} and Jensen's inequality, we obtain that for any $j\geq 1$ and $p>1$,
\begin{align*}
&\E\left[{\bf{1}}_{\{\widehat{J}_{j,k},s_1,\ldots,s_j\}}\left(\frac{\partial_{\sigma}q_{(j)}^{\theta,\sigma,\lambda}}{q_{(j)}^{\theta,\sigma,\lambda}}(\Delta_n,X_{t_k}^{\overline{\theta},\overline{\sigma},\overline{\lambda}},X_{t_{k+1}}^{\overline{\theta},\overline{\sigma},\overline{\lambda}};s_1,\ldots,s_j)\right)^{p}\big\vert X_{t_{k}}^{\overline{\theta},\overline{\sigma},\overline{\lambda}}=x\right]=\E\bigg[{\bf{1}}_{\{\widehat{J}_{j,k},s_1,\ldots,s_j\}}\\
&\times\left(\dfrac{1}{\Delta_n}\widetilde{\E}_{x}^{\theta,\sigma,\lambda}\left[\delta\left(\partial_{\sigma}Y_{t_{k+1}}^{\theta,\sigma,\lambda}(t_k,x)U^{\theta,\sigma,\lambda}(t_k,x)\right)\big\vert Y_{t_{k+1}}^{\theta,\sigma,\lambda}=X_{t_{k+1}}^{\overline{\theta},\overline{\sigma},\overline{\lambda}},\widetilde{J}_{j,k},s_1,\ldots,s_j\right]\right)^p\big\vert X_{t_{k}}^{\overline{\theta},\overline{\sigma},\overline{\lambda}}=x\bigg]\\
&\leq\dfrac{1}{\Delta_n^p}\int_{\R}\widetilde{\E}_{x}^{\theta,\sigma,\lambda}\left[\left\vert\delta\left(\partial_{\sigma}Y_{t_{k+1}}^{\theta,\sigma,\lambda}(t_k,x)U^{\theta,\sigma,\lambda}(t_k,x)\right)\right\vert^p\big\vert Y_{t_{k+1}}^{\theta,\sigma,\lambda}=y,\widetilde{J}_{j,k},s_1,\ldots,s_j\right]\\
&\qquad\times q_{(j)}^{\overline{\theta},\overline{\sigma},\overline{\lambda}}(\Delta_n,x,y;s_1,\ldots,s_j)e^{-\overline{\lambda}\Delta_n}\frac{(\overline{\lambda}\Delta_n)^j}{j!}\frac{j!}{\Delta_n^j}dy\\
&=\dfrac{1}{\Delta_n^p}e^{(\lambda-\overline{\lambda})\Delta_n}\left(\frac{\overline{\lambda}}{\lambda}\right)^j\\
&\times\widetilde{\E}_{x}^{\theta,\sigma,\lambda}\left[\left\vert\delta\left(\partial_{\sigma}Y_{t_{k+1}}^{\theta,\sigma,\lambda}(t_k,x)U^{\theta,\sigma,\lambda}(t_k,x)\right)\right\vert^p{\bf{1}}_{\{\widetilde{J}_{j,k},s_1,\ldots,s_j\}}\dfrac{q_{(j)}^{\overline{\theta},\overline{\sigma},\overline{\lambda}}}{q_{(j)}^{\theta,\sigma,\lambda}}(\Delta_n,Y_{t_{k}}^{\theta,\sigma,\lambda},Y_{t_{k+1}}^{\theta,\sigma,\lambda};s_1,\ldots,s_j)\right].
\end{align*}

Next applying H\"older's inequality with $\frac{1}{p_1}+\frac{1}{p_2}=1$, Lemma \ref{lemma12} and \eqref{es5}, we get that for any $j\geq 1$, $p>1$, $p_1>1$ if $\sigma_2\geq\sigma_1$, and $p_1\in (1,\frac{\sigma_1}{\sigma_1-\sigma_2})$ if $\sigma_2<\sigma_1$,
\begin{align*}
&\E\left[{\bf{1}}_{\{\widehat{J}_{j,k},s_1,\ldots,s_j\}}\left(\frac{\partial_{\sigma}q_{(j)}^{\theta,\sigma,\lambda}}{q_{(j)}^{\theta,\sigma,\lambda}}(\Delta_n,X_{t_k}^{\overline{\theta},\overline{\sigma},\overline{\lambda}},X_{t_{k+1}}^{\overline{\theta},\overline{\sigma},\overline{\lambda}};s_1,\ldots,s_j)\right)^{p}\big\vert X_{t_{k}}^{\overline{\theta},\overline{\sigma},\overline{\lambda}}=x\right]\\
&\leq \dfrac{1}{\Delta_n^{p}}e^{(\lambda-\overline{\lambda})\Delta_n}\left(\frac{\overline{\lambda}}{\lambda}\right)^j\left(\widetilde{\E}_{x}^{\theta,\sigma,\lambda}\left[\left\vert\delta\left(\partial_{\sigma}Y_{t_{k+1}}^{\theta,\sigma,\lambda}(t_k,x)U^{\theta,\sigma,\lambda}(t_k,x)\right)\right\vert^{p p_2}\right]\right)^{\frac{1}{p_2}}\\
&\qquad\times \left(\widetilde{\E}_{x}^{\theta,\sigma,\lambda}\left[{\bf{1}}_{\{\widetilde{J}_{j,k},s_1,\ldots,s_j\}}\left(\dfrac{q_{(j)}^{\overline{\theta},\overline{\sigma},\overline{\lambda}}}{q_{(j)}^{\theta,\sigma,\lambda}}(\Delta_n,Y_{t_{k}}^{\theta,\sigma,\lambda},Y_{t_{k+1}}^{\theta,\sigma,\lambda};s_1,\ldots,s_j)\right)^{p_1}\right]\right)^{\frac{1}{p_1}}\\
&\leq Ce^{(\lambda-\overline{\lambda})\Delta_n}\left(\frac{\overline{\lambda}}{\lambda}\right)^j(e^{-\lambda\Delta_n}\lambda^j)^{\frac{1}{p_1}}\left(1+\vert x\vert^q\right),
\end{align*}
for some constants $C, q>0$. This concludes the second inequality. The proof of the others follows along the same lines and is thus omitted.
\end{proof}

As a consequence, we have the following crucial estimate.
\begin{lemma}\label{lemma15} For any random variable $Z$, $k\in \{0,...,n-1\}$ and $j\geq 1$, there exist constants $C, q>0$ such that for $n$ large enough, for any $\overline{q}_1, \overline{q}_2, \overline{q}_3>1$ conjugate, $\overline{p}_1>1$, $\widetilde{q}_1, \widetilde{q}_2, \widetilde{q}_3>1$ conjugate, $\widetilde{p}_1>1$, $q_1, q_2, q_3>1$ conjugate, and $p_1>1$, 
\begin{align*}
&\left\vert\E\left[Z{\bf{1}}_{\{\widehat{J}_{j,k},s_1,\ldots,s_j\}}\bigg(\frac{q_{(j)}^{\theta_0,\sigma_0,\lambda_0}}{q_{(j)}^{\theta_n,\sigma(\ell),\lambda_n}}(\Delta_n,X_{t_k}^{\theta_n,\sigma(\ell),\lambda_n},X^{\theta_n,\sigma(\ell),\lambda_n}_{t_{k+1}};s_1,\ldots,s_j)-1\bigg)\big\vert X_{t_k}^{\theta_n,\sigma(\ell),\lambda_n}=x\right]\right\vert\\
&\leq\dfrac{C}{\sqrt{n}}\left(1+\vert x\vert^q\right)\bigg\{\left(\E\left[\vert Z\vert^{\overline{q}_1}\vert X_{t_k}^{\theta_n,\sigma(\ell),\lambda_n}=x\right]\right)^{\frac{1}{\overline{q}_1}}\left(\frac{\lambda_n}{\lambda_0}\right)^{\frac{j}{\overline{q}_2}}(e^{-\lambda_0\Delta_n}\lambda_0^j)^{\frac{1}{\overline{p}_1 \overline{q}_2}}(e^{-\lambda_n\Delta_n}\lambda_n^j)^{\frac{1}{\overline{q}_3}}\\
&+\left(\E\left[\vert Z\vert^{\widetilde{q}_1}\vert X_{t_k}^{\theta_n,\sigma(\ell),\lambda_n}=x\right]\right)^{\frac{1}{\widetilde{q}_1}}\left(\frac{\lambda_n}{\lambda_0}\right)^{\frac{j}{\widetilde{q}_2}}(e^{-\lambda_0\Delta_n}\lambda_0^j)^{\frac{1}{\widetilde{p}_1 \widetilde{q}_2}}(e^{-\lambda_n\Delta_n}\lambda_n^j)^{\frac{1}{\widetilde{q}_3}}+\int_0^1\Big(\frac{\lambda_n}{\lambda_0+\frac{wh}{\sqrt{n\Delta_n}}}\Big)^{\frac{j}{q_2}}\\
&\times\left(\E\left[\vert Z\vert^{q_1}\vert X_{t_k}^{\theta_n,\sigma(\ell),\lambda_n}=x\right]\right)^{\frac{1}{q_1}}\Big(e^{-(\lambda_0+\frac{wh}{\sqrt{n\Delta_n}})\Delta_n}(\lambda_0+\frac{wh}{\sqrt{n\Delta_n}})^j\Big)^{\frac{1}{p_1 q_2}}(e^{-\lambda_n\Delta_n}\lambda_n^j)^{\frac{1}{q_3}}dh\bigg\},
\end{align*}
and
\begin{align*}
&\left\vert\E\left[Z{\bf{1}}_{\widehat{J}_{0,k}}\left(\frac{q_{(0)}^{\theta_0,\sigma_0,\lambda_0}}{q_{(0)}^{\theta_n,\sigma(\ell),\lambda_n}}(\Delta_n,X_{t_k}^{\theta_n,\sigma(\ell),\lambda_n},X^{\theta_n,\sigma(\ell),\lambda_n}_{t_{k+1}})-1\right)\big\vert X_{t_k}^{\theta_n,\sigma(\ell),\lambda_n}=x\right]\right\vert\\
&\leq \dfrac{C}{\sqrt{n}}\left(1+\vert x\vert^q\right)\bigg\{\left(\E\left[\vert Z\vert^{\overline{q}_1}\vert X_{t_k}^{\theta_n,\sigma(\ell),\lambda_n}=x\right]\right)^{\frac{1}{\overline{q}_1}}+\left(\E\left[\vert Z\vert^{\widetilde{q}_1}\vert X_{t_k}^{\theta_n,\sigma(\ell),\lambda_n}=x\right]\right)^{\frac{1}{\widetilde{q}_1}}\\
&\qquad+\left(\E\left[\vert Z\vert^{q_1}\vert X_{t_k}^{\theta_n,\sigma(\ell),\lambda_n}=x\right]\right)^{\frac{1}{q_1}}\bigg\}.
\end{align*}
\end{lemma}
\begin{proof} 
We write $\frac{q_{(j)}^{\theta_0,\sigma_0,\lambda_0}}{q_{(j)}^{\theta_n,\sigma(\ell),\lambda_n}}\equiv\frac{q_{(j)}^{\theta_0,\sigma_0,\lambda_0}}{q_{(j)}^{\theta_n,\sigma(\ell),\lambda_n}}(\Delta_n,X_{t_k}^{\theta_n,\sigma(\ell),\lambda_n},X^{\theta_n,\sigma(\ell),\lambda_n}_{t_{k+1}};s_1,\ldots,s_j)$. Observe that
\begin{align*}
&\frac{q_{(j)}^{\theta_0,\sigma_0,\lambda_0}}{q_{(j)}^{\theta_n,\sigma(\ell),\lambda_n}}-1=\frac{q_{(j)}^{\theta_0,\sigma_0,\lambda_0}-q_{(j)}^{\theta_0,\sigma(\ell),\lambda_0}+q_{(j)}^{\theta_0,\sigma(\ell),\lambda_0}-q_{(j)}^{\theta_n,\sigma(\ell),\lambda_0}+q_{(j)}^{\theta_n,\sigma(\ell),\lambda_0}-q_{(j)}^{\theta_n,\sigma(\ell),\lambda_n}}{q_{(j)}^{\theta_n,\sigma(\ell),\lambda_n}}\\
&=-\int_0^1\left(\frac{\ell v}{\sqrt{n}}\frac{\partial_{\sigma}q_{(j)}^{\theta_0,\sigma_0+\frac{\ell vh}{\sqrt{n}},\lambda_0}}{q_{(j)}^{\theta_n,\sigma(\ell),\lambda_n}}+\frac{u}{\sqrt{n\Delta_n}}\frac{\partial_{\theta}q_{(j)}^{\theta_0+\frac{uh}{\sqrt{n\Delta_n}},\sigma(\ell),\lambda_0}}{q_{(j)}^{\theta_n,\sigma(\ell),\lambda_n}}+\frac{w}{\sqrt{n\Delta_n}}\frac{\partial_{\lambda}q_{(j)}^{\theta_n,\sigma(\ell),\lambda_0+\frac{wh}{\sqrt{n\Delta_n}}}}{q_{(j)}^{\theta_n,\sigma(\ell),\lambda_n}}\right)dh.
\end{align*}
This implies that
\begin{align*}
\left\vert\E\left[Z{\bf{1}}_{\{\widehat{J}_{j,k},s_1,\ldots,s_j\}}\left(\frac{q_{(j)}^{\theta_0,\sigma_0,\lambda_0}}{q_{(j)}^{\theta_n,\sigma(\ell),\lambda_n}}-1\right)\big\vert X_{t_k}^{\theta_n,\sigma(\ell),\lambda_n}=x\right]\right\vert\leq S_1+S_2+S_3,
\end{align*}
where
\begin{align*}
&S_1:=\frac{\vert v\vert}{\sqrt{n}}\int_0^1\left\vert\E\left[Z{\bf{1}}_{\{\widehat{J}_{j,k},s_1,\ldots,s_j\}}\frac{\partial_{\sigma}q_{(j)}^{\theta_0,\sigma_0+\frac{\ell vh}{\sqrt{n}},\lambda_0}}{q_{(j)}^{\theta_0,\sigma_0+\frac{\ell vh}{\sqrt{n}},\lambda_0}} \frac{q_{(j)}^{\theta_0,\sigma_0+\frac{\ell vh}{\sqrt{n}},\lambda_0}}{q_{(j)}^{\theta_n,\sigma(\ell),\lambda_n}}\big\vert X_{t_k}^{\theta_n,\sigma(\ell),\lambda_n}=x\right]\right\vert dh,\\
&S_2:=\frac{\vert u\vert}{\sqrt{n\Delta_n}}\int_0^1\left\vert\E\left[Z{\bf{1}}_{\{\widehat{J}_{j,k},s_1,\ldots,s_j\}}\frac{\partial_{\theta}q_{(j)}^{\theta_0+\frac{uh}{\sqrt{n\Delta_n}},\sigma(\ell),\lambda_0}}{q_{(j)}^{\theta_0+\frac{uh}{\sqrt{n\Delta_n}},\sigma(\ell),\lambda_0}} \frac{q_{(j)}^{\theta_0+\frac{uh}{\sqrt{n\Delta_n}},\sigma(\ell),\lambda_0}}{q_{(j)}^{\theta_n,\sigma(\ell),\lambda_n}}\big\vert X_{t_k}^{\theta_n,\sigma(\ell),\lambda_n}=x\right]\right\vert dh,\\
&S_3:=\frac{\vert w\vert}{\sqrt{n\Delta_n}}\int_0^1\left\vert\E\left[Z{\bf{1}}_{\{\widehat{J}_{j,k},s_1,\ldots,s_j\}}\frac{\partial_{\lambda}q_{(j)}^{\theta_n,\sigma(\ell),\lambda_0+\frac{wh}{\sqrt{n\Delta_n}}}}{q_{(j)}^{\theta_n,\sigma(\ell),\lambda_0+\frac{wh}{\sqrt{n\Delta_n}}}} \frac{q_{(j)}^{\theta_n,\sigma(\ell),\lambda_0+\frac{wh}{\sqrt{n\Delta_n}}}}{q_{(j)}^{\theta_n,\sigma(\ell),\lambda_n}}\big\vert X_{t_k}^{\theta_n,\sigma(\ell),\lambda_n}=x\right]\right\vert dh.
\end{align*}
First, using H\"older's inequality with $\overline{q}_1, \overline{q}_2, \overline{q}_3>1$ conjugate, Lemmas \ref{lemma12} and \ref{lemma14} where notice that $\theta_n$ tends to $\theta_0$, and $\sigma_0+\frac{\ell vh}{\sqrt{n}}$ and $\sigma(\ell)$ tend to $\sigma_0$ as $n\to\infty$, we get that for $n$ large enough, and for any $\overline{p}_1>1$,
\begin{align*}
&S_1\leq \frac{\vert v\vert}{\sqrt{n}}\int_0^1\left(\E\left[\vert Z\vert^{\overline{q}_1}\vert X_{t_k}^{\theta_n,\sigma(\ell),\lambda_n}=x\right]\right)^{\frac{1}{\overline{q}_1}}\\
&\qquad\times\left(\E\left[{\bf{1}}_{\{\widehat{J}_{j,k},s_1,\ldots,s_j\}}\left(\frac{\partial_{\sigma}q_{(j)}^{\theta_0,\sigma_0+\frac{\ell vh}{\sqrt{n}},\lambda_0}}{q_{(j)}^{\theta_0,\sigma_0+\frac{\ell vh}{\sqrt{n}},\lambda_0}}\right)^{\overline{q}_2}\big\vert X_{t_k}^{\theta_n,\sigma(\ell),\lambda_n}=x\right]\right)^{\frac{1}{\overline{q}_2}}\\
&\qquad\times \left(\E\left[{\bf{1}}_{\{\widehat{J}_{j,k},s_1,\ldots,s_j\}}\left(\frac{q_{(j)}^{\theta_0,\sigma_0+\frac{\ell vh}{\sqrt{n}},\lambda_0}}{q_{(j)}^{\theta_n,\sigma(\ell),\lambda_n}}\right)^{\overline{q}_3}\big\vert X_{t_k}^{\theta_n,\sigma(\ell),\lambda_n}=x\right]\right)^{\frac{1}{\overline{q}_3}}dh\\
&\leq C\frac{\vert v\vert}{\sqrt{n}}(\E[\vert Z\vert^{\overline{q}_1}\vert X_{t_k}^{\theta_n,\sigma(\ell),\lambda_n}=x])^{\frac{1}{\overline{q}_1}}e^{-\frac{w}{\overline{q}_2}\sqrt{\frac{\Delta_n}{n}}}\big(\frac{\lambda_n}{\lambda_0}\big)^{\frac{j}{\overline{q}_2}}(e^{-\lambda_0\Delta_n}\lambda_0^j)^{\frac{1}{\overline{p}_1 \overline{q}_2}}(e^{-\lambda_n\Delta_n}\lambda_n^j)^{\frac{1}{\overline{q}_3}}\left(1+\vert x\vert^q\right).
\end{align*}
Next, using H\"older's inequality with $\widetilde{q}_1, \widetilde{q}_2, \widetilde{q}_3>1$ conjugate, Lemmas \ref{lemma12} and \ref{lemma14} where notice that $\theta_0+\frac{uh}{\sqrt{n\Delta_n}}$ and $\theta_n$ tend to $\theta_0$ as $n\to\infty$, we get that for $n$ large enough, and any $\widetilde{p}_1>1$,
\begin{align*}
&S_2\leq \frac{\vert u\vert}{\sqrt{n\Delta_n}}\int_0^1\left(\E\left[\vert Z\vert^{\widetilde{q}_1}\vert X_{t_k}^{\theta_n,\sigma(\ell),\lambda_n}=x\right]\right)^{\frac{1}{\widetilde{q}_1}}\\
&\qquad\times\left(\E\left[{\bf{1}}_{\{\widehat{J}_{j,k},s_1,\ldots,s_j\}}\left(\frac{\partial_{\theta}q_{(j)}^{\theta_0+\frac{uh}{\sqrt{n\Delta_n}},\sigma(\ell),\lambda_0}}{q_{(j)}^{\theta_0+\frac{uh}{\sqrt{n\Delta_n}},\sigma(\ell),\lambda_0}}\right)^{\widetilde{q}_2}\big\vert X_{t_k}^{\theta_n,\sigma(\ell),\lambda_n}=x\right]\right)^{\frac{1}{\widetilde{q}_2}}\\
&\qquad\times \left(\E\left[{\bf{1}}_{\{\widehat{J}_{j,k},s_1,\ldots,s_j\}}\left(\frac{q_{(j)}^{\theta_0+\frac{uh}{\sqrt{n\Delta_n}},\sigma(\ell),\lambda_0}}{q_{(j)}^{\theta_n,\sigma(\ell),\lambda_n}}\right)^{\widetilde{q}_3}\big\vert X_{t_k}^{\theta_n,\sigma(\ell),\lambda_n}=x\right]\right)^{\frac{1}{\widetilde{q}_3}}dh\\
&\leq C\frac{\vert u\vert}{\sqrt{n}}(\E[\vert Z\vert^{\widetilde{q}_1}\vert X_{t_k}^{\theta_n,\sigma(\ell),\lambda_n}=x])^{\frac{1}{\widetilde{q}_1}}e^{-\frac{w}{\widetilde{q}_2}\sqrt{\frac{\Delta_n}{n}}}\big(\frac{\lambda_n}{\lambda_0}\big)^{\frac{j}{\widetilde{q}_2}}(e^{-\lambda_0\Delta_n}\lambda_0^j)^{\frac{1}{\widetilde{p}_1 \widetilde{q}_2}}(e^{-\lambda_n\Delta_n}\lambda_n^j)^{\frac{1}{\widetilde{q}_3}}\left(1+\vert x\vert^q\right).
\end{align*}
Finally, using H\"older's inequality with $q_1, q_2, q_3>1$ conjugate, Lemmas \ref{lemma12} and \ref{lemma14} where notice that $\sigma_1=\sigma_2=\frac{\sigma(\ell)^2}{\theta_n}(1-e^{-2\theta_n\Delta_n})$, we get that for any $p_1, p_2>1$ conjugate,
\begin{align*}
S_3&\leq \frac{\vert w\vert}{\sqrt{n\Delta_n}}\int_0^1\left(\E\left[\vert Z\vert^{q_1}\vert X_{t_k}^{\theta_n,\sigma(\ell),\lambda_n}=x\right]\right)^{\frac{1}{q_1}}\\
&\qquad\times\left(\E\left[{\bf{1}}_{\{\widehat{J}_{j,k},s_1,\ldots,s_j\}}\left(\frac{\partial_{\lambda}q_{(j)}^{\theta_n,\sigma(\ell),\lambda_0+\frac{wh}{\sqrt{n\Delta_n}}}}{q_{(j)}^{\theta_n,\sigma(\ell),\lambda_0+\frac{wh}{\sqrt{n\Delta_n}}}}\right)^{q_2}\big\vert X_{t_k}^{\theta_n,\sigma(\ell),\lambda_n}=x\right]\right)^{\frac{1}{q_2}}\\
&\qquad\times \left(\E\left[{\bf{1}}_{\{\widehat{J}_{j,k},s_1,\ldots,s_j\}}\left(\frac{q_{(j)}^{\theta_n,\sigma(\ell),\lambda_0+\frac{wh}{\sqrt{n\Delta_n}}}}{q_{(j)}^{\theta_n,\sigma(\ell),\lambda_n}}\right)^{q_3}\big\vert X_{t_k}^{\theta_n,\sigma(\ell),\lambda_n}=x\right]\right)^{\frac{1}{q_3}}dh\\
&\leq C\Delta_n^{\frac{1}{p_2 q_2}-\frac{1}{2}}\frac{\vert w\vert}{\sqrt{n}}\int_0^1\left(\E\left[\vert Z\vert^{q_1}\vert X_{t_k}^{\theta_n,\sigma(\ell),\lambda_n}=x\right]\right)^{\frac{1}{q_1}}e^{\frac{1}{q_2}(\frac{wh}{\sqrt{n\Delta_n}}-\frac{w}{\sqrt{n\Delta_n}})}\Big(\frac{\lambda_n}{\lambda_0+\frac{wh}{\sqrt{n\Delta_n}}}\Big)^{\frac{j}{q_2}}\\
&\qquad\times\Big(e^{-(\lambda_0+\frac{wh}{\sqrt{n\Delta_n}})\Delta_n}(\lambda_0+\frac{wh}{\sqrt{n\Delta_n}})^j\Big)^{\frac{1}{p_1 q_2}}(e^{-\lambda_n\Delta_n}\lambda_n^j)^{\frac{1}{q_3}}dh\\
&\leq C\frac{\vert w\vert}{\sqrt{n}}\int_0^1\left(\E\left[\vert Z\vert^{q_1}\vert X_{t_k}^{\theta_n,\sigma(\ell),\lambda_n}=x\right]\right)^{\frac{1}{q_1}}e^{\frac{1}{q_2}(\frac{wh}{\sqrt{n\Delta_n}}-\frac{w}{\sqrt{n\Delta_n}})}\Big(\frac{\lambda_n}{\lambda_0+\frac{wh}{\sqrt{n\Delta_n}}}\Big)^{\frac{j}{q_2}}\\
&\qquad\times\Big(e^{-(\lambda_0+\frac{wh}{\sqrt{n\Delta_n}})\Delta_n}(\lambda_0+\frac{wh}{\sqrt{n\Delta_n}})^j\Big)^{\frac{1}{p_1 q_2}}(e^{-\lambda_n\Delta_n}\lambda_n^j)^{\frac{1}{q_3}}dh,
\end{align*}
where $p_2$ and $q_2$ are chosen in order that $p_2q_2<2$. This concludes the first inequality. The second inequality is proceeded similarly. Thus, the result follows.
\end{proof}

\subsection{Large deviation type estimates}
\label{largeest}

For all $p\geq 1$, $k \in \{0,...,n-1\}$, and $i\in\{0,1\}$, we set $\widehat{J}_{\geq 2,k}:=\{\Delta N_{t_{k+1}}\geq 2\}$, $\widetilde{J}_{\geq 2,k}:=\{\Delta M_{t_{k+1}}\geq 2\}$, and
\begin{align*}
M_{i,p}^{\theta,\sigma,\lambda}:&=\E\left[{\bf 1}_{\widehat{J}_{i,k}}\left((\Delta N_{t_{k+1}})^p-\widetilde{\E}_{X_{t_k}}^{\theta,\sigma,\lambda}\left[(\Delta M_{t_{k+1}})^p\big\vert Y_{t_{k+1}}^{\theta,\sigma,\lambda}=X_{t_{k+1}}\right]\right)^2\big\vert X_{t_k}\right],\\
M_{\geq 2,p}^{\theta,\sigma,\lambda}:&=\E\left[{\bf 1}_{\widehat{J}_{\geq 2,k}}\left((\Delta N_{t_{k+1}})^p-\widetilde{\E}_{X_{t_k}}^{\theta,\sigma,\lambda}\left[(\Delta M_{t_{k+1}})^p\big\vert Y_{t_{k+1}}^{\theta,\sigma,\lambda}=X_{t_{k+1}}\right]\right)^2\big\vert X_{t_k}\right].
\end{align*}
For jump diffusions \cite{KNT14, KNT15}, a large deviation principle is used by conditioning on the jump structure to derive the large deviation type estimates in \cite[Lemma 2.4]{KNT14} and \cite[Lemma 5.4]{KNT15}. For the O-U process \eqref{c2eq1}, the following analogue large deviation type estimates hold.  
\begin{lemma}\label{deviation} Let $(\theta,\sigma,\lambda)\in \Theta\times \Sigma \times\Lambda$ such that $\vert\theta_0-\theta\vert+\vert\lambda_0-\lambda\vert\leq\frac{C_0}{\sqrt{n\Delta_n}}$, $\vert\sigma_0-\sigma\vert\leq\frac{C_0}{\sqrt{n}}$, for some constant $C_0>0$. Then for all $p\geq 1$, $k \in \{0,...,n-1\}$, and $n$ large enough, there exist constants $C, C_1>0$ such that for any $\alpha\in(0,\frac{1}{2})$, $p_1>1$, $q_1>1$ with $p_1q_1<2$, and $\mu_1\in(1,2)$,
\begin{align}
&M_{0,p}^{\theta,\sigma,\lambda}\leq C\left(e^{-C_1\Delta_n^{2\alpha-1}}+\Delta_n^{\frac{2}{p_1q_1}}\right),\label{m0}\\
&M_{1,p}^{\theta,\sigma,\lambda}\leq C\left(e^{-C_1\Delta_n^{2\alpha-1}}+\Delta_n^{\frac{2}{p_1q_1}}\right),\label{m1}\\
&M_{\geq 2,p}^{\theta,\sigma,\lambda}\leq C\Delta_n^{\frac{2}{\mu_1}}.\label{m2}
\end{align}
\end{lemma}
\begin{proof}
We start showing (\ref{m0}). Multiplying the random variable inside the conditional expectation of $M_{0,p}^{\theta,\sigma,\lambda}$ by ${\bf 1}_{\widetilde{J}_{0,k}} +{\bf 1}_{\widetilde{J}_{1,k}}+ {\bf 1}_{\widetilde{J}_{\geq 2,k}}$ and applying Jensen's inequality, we get that $M_{0,p}^{\theta,\sigma,\lambda}\leq 2(M_{0,1,p}^{\theta,\sigma,\lambda}+M_{0,2,p}^{\theta,\sigma,\lambda})$, where 
\begin{align*}
M_{0,1,p}^{\theta,\sigma,\lambda}:&=\E\left[{\bf 1}_{\widehat{J}_{0,k}}\widetilde{\E}_{X_{t_k}}^{\theta,\sigma,\lambda}\left[{\bf 1}_{\widetilde{J}_{1,k}}\big\vert Y_{t_{k+1}}^{\theta,\sigma,\lambda}=X_{t_{k+1}}\right]\big\vert X_{t_k}\right],\\
M_{0,2,p}^{\theta,\sigma,\lambda}:&=\E\left[{\bf 1}_{\widehat{J}_{0,k}}\widetilde{\E}_{X_{t_k}}^{\theta,\sigma,\lambda}\left[{\bf 1}_{\widetilde{J}_{\geq 2,k}}(\Delta M_{t_{k+1}})^{2p}\big\vert Y_{t_{k+1}}^{\theta,\sigma,\lambda}=X_{t_{k+1}}\right]\big\vert X_{t_k}\right].
\end{align*}
Proceeding as in Lemma \ref{change} to change measures, applying H\"older's inequality with $\frac{1}{p_1}+\frac{1}{p_2}=1$, Jensen's inequality, and Lemma \ref{lemma12} where notice that $\theta\to\theta_0$, $\sigma\to\sigma_0$ as $n\to\infty$, we get 
\begin{align*}
M_{0,1,p}^{\theta,\sigma,\lambda}&=\frac{e^{\lambda\Delta_n}}{e^{\lambda_0\Delta_n}}\E\left[{\bf 1}_{\widehat{J}_{0,k}}\frac{q_{(0)}^{\theta_0,\sigma_0,\lambda_0}}{q_{(0)}^{\theta,\sigma,\lambda}}(\Delta_n,X_{t_k}^{\theta,\sigma,\lambda},X_{t_{k+1}}^{\theta,\sigma,\lambda})\widetilde{\E}_{X_{t_k}}^{\theta,\sigma,\lambda}[{\bf 1}_{\widetilde{J}_{1,k}}\big\vert Y_{t_{k+1}}^{\theta,\sigma,\lambda}=X_{t_{k+1}}^{\theta,\sigma,\lambda}]\big\vert X_{t_k}^{\theta,\sigma,\lambda}=X_{t_k}\right]\\
&\leq e^{(\lambda-\lambda_0)\Delta_n}\left(\E\left[{\bf 1}_{\widehat{J}_{0,k}}\left(\frac{q_{(0)}^{\theta_0,\sigma_0,\lambda_0}}{q_{(0)}^{\theta,\sigma,\lambda}}(\Delta_n,X_{t_k}^{\theta,\sigma,\lambda},X_{t_{k+1}}^{\theta,\sigma,\lambda})\right)^{p_1}\big\vert X_{t_k}^{\theta,\sigma,\lambda}=X_{t_k}\right]\right)^{\frac{1}{p_1}}\\
&\qquad\times\left(\E\left[{\bf 1}_{\widehat{J}_{0,k}}\widetilde{\E}_{X_{t_k}}^{\theta,\sigma,\lambda}\left[{\bf 1}_{\widetilde{J}_{1,k}}\big\vert Y_{t_{k+1}}^{\theta,\sigma,\lambda}=X_{t_{k+1}}^{\theta,\sigma,\lambda}\right]\big\vert X_{t_k}^{\theta,\sigma,\lambda}=X_{t_k}\right]\right)^{\frac{1}{p_2}}\\
&\leq C\left(M_{0,1,1,p}^{\theta,\sigma,\lambda}\right)^{\frac{1}{p_2}},
\end{align*}
for some constant $C>0$ and $n$ large enough. Here, using Bayes' formula, we get that
\begin{align*}
M_{0,1,1,p}^{\theta,\sigma,\lambda}:&=\E\left[{\bf 1}_{\widehat{J}_{0,k}}\widetilde{\E}_{X_{t_k}}^{\theta,\sigma,\lambda}\left[{\bf 1}_{\widetilde{J}_{1,k}}\big\vert Y_{t_{k+1}}^{\theta,\sigma,\lambda}=X_{t_{k+1}}^{\theta,\sigma,\lambda}\right]\big\vert X_{t_k}^{\theta,\sigma,\lambda}=X_{t_k}\right]\\
&=\E\left[{\bf 1}_{\widehat{J}_{0,k}}\frac{q_{(1)}^{\theta,\sigma,\lambda}(\Delta_n,X_{t_k}^{\theta,\sigma,\lambda},X_{t_{k+1}}^{\theta,\sigma,\lambda})e^{-\lambda\Delta_n}\lambda\Delta_n}{p^{\theta,\sigma,\lambda}(\Delta_n,X_{t_k}^{\theta,\sigma,\lambda},X_{t_{k+1}}^{\theta,\sigma,\lambda})}\big\vert X_{t_k}^{\theta,\sigma,\lambda}=X_{t_k}\right]\\
&=\int_{\R}\frac{q_{(1)}^{\theta,\sigma,\lambda}(\Delta_n,X_{t_k},y)e^{-\lambda\Delta_n}\lambda\Delta_n}{p^{\theta,\sigma,\lambda}(\Delta_n,X_{t_k},y)}q_{(0)}^{\theta,\sigma,\lambda}(\Delta_n,X_{t_k},y)e^{-\lambda\Delta_n}dy.
\end{align*}

We next divide the $dy$ integral into the subdomains $J_1:=\{y\in\R:\vert y-X_{t_k}e^{-\theta\Delta_n}+\frac{\lambda}{\theta}(1-e^{-\theta \Delta_n})\vert>\Delta_n^{\alpha}\}$ and $J_2:=\{y\in\R:\vert y-X_{t_k}e^{-\theta\Delta_n}+\frac{\lambda}{\theta}(1-e^{-\theta \Delta_n})\vert\leq\Delta_n^{\alpha}\}$, where $\alpha\in(0,\frac{1}{2})$, and call each integral $M_{0,1,1,1,p}^{\theta,\sigma,\lambda}$ and $M_{0,1,1,2,p}^{\theta,\sigma,\lambda}$. We start bounding $M_{0,1,1,1,p}^{\theta,\sigma,\lambda}$. By \eqref{density},
\begin{equation}\label{de1}
p^{\theta,\sigma,\lambda}(\Delta_n,X_{t_k},y)\geq q_{(1)}^{\theta,\sigma,\lambda}(\Delta_n,X_{t_k},y)e^{-\lambda\Delta_n}\lambda\Delta_n.
\end{equation}
Then using \eqref{q0}, the equality $e^{-\vert x\vert^2}=e^{-\frac{\vert x\vert^2}{2}}e^{-\frac{\vert x\vert^2}{2}}$, valid for all $x\in\R$, and $1-e^{-2\theta\Delta_n}\leq 2\theta\Delta_n$, 
\begin{align*}
M_{0,1,1,1,p}^{\theta,\sigma,\lambda}&\leq\int_{J_1}q_{(0)}^{\theta,\sigma,\lambda}(\Delta_n,X_{t_k},y)dy\\
&=\int_{J_1}\dfrac{1}{\sqrt{\frac{\pi}{\theta}\sigma^2(1-e^{-2\theta\Delta_n})}}\exp\left\{-\dfrac{\left(y-X_{t_k}e^{-\theta\Delta_n}+\frac{\lambda}{\theta}(1-e^{-\theta\Delta_n})\right)^2}{\frac{1}{\theta}\sigma^2(1-e^{-2\theta\Delta_n})}\right\}dy\\
&\leq e^{-\frac{1}{4\sigma^2}\Delta_n^{2\alpha-1}}\int_{J_1}\dfrac{1}{\sqrt{\frac{\pi}{\theta}\sigma^2(1-e^{-2\theta\Delta_n})}}\exp\left\{-\dfrac{\left(y-X_{t_k}e^{-\theta\Delta_n}+\frac{\lambda}{\theta}(1-e^{-\theta\Delta_n})\right)^2}{\frac{2}{\theta}\sigma^2(1-e^{-2\theta\Delta_n})}\right\}dy\\
&\leq Ce^{-\frac{1}{4\sigma^2}\Delta_n^{2\alpha-1}},
\end{align*}
for some constant $C>0$, since the $dy$ integral is Gaussian and thus finite. Next, by \eqref{density},
\begin{equation}\label{de2}
p^{\theta,\sigma,\lambda}(\Delta_n,X_{t_k},y)\geq q_{(0)}^{\theta,\sigma,\lambda}(\Delta_n,X_{t_k},y)e^{-\lambda\Delta_n}.
\end{equation}
Then using \eqref{qjf}, \eqref{qj}, and Fubini's theorem,
\begin{align*}
&M_{0,1,1,2,p}^{\theta,\sigma,\lambda}\leq e^{-\lambda\Delta_n}\lambda\Delta_n\int_{J_2}q_{(1)}^{\theta,\sigma,\lambda}(\Delta_n,X_{t_k},y)dy\\
&\leq \int_{t_k}^{t_{k+1}}\int_{J_2}\dfrac{\lambda}{\sqrt{\frac{\pi}{\theta}\sigma^2(1-e^{-2\theta\Delta_n})}}\exp\left\{-\dfrac{\left(y-X_{t_k}e^{-\theta\Delta_n}+\frac{\lambda}{\theta}(1-e^{-\theta\Delta_n})-e^{-\theta (t_{k+1}-s_1)}\right)^2}{\frac{1}{\theta}\sigma^2(1-e^{-2\theta\Delta_n})}\right\}dyds_1.
\end{align*}

Notice that $e^{-\theta\Delta_n}\leq e^{-\theta (t_{k+1}-s_1)}\leq 1$, for  $s_1\in[t_k,t_{k+1}]$. Thus, $e^{-\theta (t_{k+1}-s_1)}\geq \frac{1}{2}$ for $n$ large enough. Moreover, on $J_2$ we have $\vert y-X_{t_k}e^{-\theta\Delta_n}+\frac{\lambda}{\theta}(1-e^{-\theta \Delta_n})\vert\leq\Delta_n^{\alpha}$. Therefore, using the inequality $\vert u+v\vert^2\geq \frac{\vert u\vert^2}{2}-\vert v\vert^2$, valid for all $u, v\in\R$, together with  $\alpha\in(0,\frac{1}{2})$, we deduce 
\begin{align*}
&\left\vert y-X_{t_k}e^{-\theta\Delta_n}+\frac{\lambda}{\theta}(1-e^{-\theta\Delta_n})-e^{-\theta (t_{k+1}-s_1)}\right\vert^2\\
&\geq\dfrac{e^{-2\theta (t_{k+1}-s_1)}}{2}-\left\vert y-X_{t_k}e^{-\theta\Delta_n}+\frac{\lambda}{\theta}(1-e^{-\theta\Delta_n})\right\vert^2\geq\dfrac{1}{8}-\Delta_n^{2\alpha}\geq \dfrac{1}{10},
\end{align*}
for $n$ large enough. Therefore, using again the equality $e^{-\vert x\vert^2}=e^{-\frac{\vert x\vert^2}{2}}e^{-\frac{\vert x\vert^2}{2}}$, valid for all $x\in\R$, and the fact that $1-e^{-2\theta\Delta_n}\leq 2\theta\Delta_n$, we obtain that for $n$ large enough,
\begin{align*}
M_{0,1,1,2,p}^{\theta,\sigma,\lambda}&\leq e^{-\frac{1}{40\sigma^2}\Delta_n^{-1}}\int_{t_k}^{t_{k+1}}\int_{J_2}\dfrac{\lambda}{\sqrt{\frac{\pi}{\theta}\sigma^2(1-e^{-2\theta\Delta_n})}}\\
&\qquad\times\exp\left\{-\dfrac{\left(y-X_{t_k}e^{-\theta\Delta_n}+\frac{\lambda}{\theta}(1-e^{-\theta\Delta_n})-e^{-\theta (t_{k+1}-s_1)}\right)^2}{\frac{2}{\theta}\sigma^2(1-e^{-2\theta\Delta_n})}\right\}dyds_1\\
&\leq Ce^{-\frac{1}{40\sigma^2}\Delta_n^{-1}},
\end{align*}
for some constant $C>0$, since the $dy$ integral is Gaussian and thus finite. Hence, 
\begin{align*}
M_{0,1,1,p}^{\theta,\sigma,\lambda}\leq Ce^{-\frac{1}{40\sigma^2}\Delta_n^{2\alpha-1}},
\end{align*}
This implies that for any $\alpha\in(0,\frac{1}{2})$, $n$ large enough, and for some constants $C, C_1>0$,
\begin{align}\label{m01}
M_{0,1,p}^{\theta,\sigma,\lambda}\leq Ce^{-C_1\Delta_n^{2\alpha-1}},
\end{align}

Next, we set $g(X_{t_{k+1}}):=\widetilde{\E}_{X_{t_k}}^{\theta,\sigma,\lambda}\left[{\bf 1}_{\widetilde{J}_{\geq 2,k}}(\Delta M_{t_{k+1}})^{2p}\big\vert Y_{t_{k+1}}^{\theta,\sigma,\lambda}=X_{t_{k+1}}\right]$. Applying Lemmas \ref{change} and \ref{lemma12}, and H\"older's inequality with $\frac{1}{p_1}+\frac{1}{p_2}=1$, we obtain that for $n$ large enough,
\begin{align*}
&M_{0,2,p}^{\theta,\sigma,\lambda}\leq\E\left[g(X_{t_{k+1}})\big\vert X_{t_k}\right]=e^{(\lambda-\lambda_0)\Delta_n}\bigg\{\E\left[g(X_{t_{k+1}}^{\theta,\sigma,\lambda}){\bf{1}}_{\widehat{J}_{0,k}}\frac{q_{(0)}^{\theta_0,\sigma_0,\lambda_0}}{q_{(0)}^{\theta,\sigma,\lambda}}\big\vert X_{t_k}^{\theta,\sigma,\lambda}=X_{t_k}\right]\\
&\qquad+\sum_{j=1}^{\infty}\left(\frac{\lambda_0}{\lambda}\right)^j\int_{\Sigma_k^j}\E\left[g(X_{t_{k+1}}^{\theta,\sigma,\lambda}){\bf{1}}_{\{\widehat{J}_{j,k},s_1,\ldots,s_j\}}\frac{q_{(j)}^{\theta_0,\sigma_0,\lambda_0}}{q_{(j)}^{\theta,\sigma,\lambda}}\big\vert X_{t_k}^{\theta,\sigma,\lambda}=X_{t_k}\right]ds_1\cdots ds_j\bigg\}\\
&\leq C\left(\E\left[\left\vert g(X_{t_{k+1}}^{\theta,\sigma,\lambda})\right\vert^{p_1}\big\vert X_{t_k}^{\theta,\sigma,\lambda}=X_{t_k}\right]\right)^{\frac{1}{p_1}}\bigg\{\left(\E\left[{\bf{1}}_{\widehat{J}_{0,k}}\left(\frac{q_{(0)}^{\theta_0,\sigma_0,\lambda_0}}{q_{(0)}^{\theta,\sigma,\lambda}}\right)^{p_2}\big\vert X_{t_k}^{\theta,\sigma,\lambda}=X_{t_k}\right]\right)^{\frac{1}{p_2}}\\
&\qquad+\sum_{j=1}^{\infty}\left(\frac{\lambda_0}{\lambda}\right)^j\int_{\Sigma_k^j}\left(\E\left[{\bf{1}}_{\{\widehat{J}_{j,k},s_1,\ldots,s_j\}}\left(\frac{q_{(j)}^{\theta_0,\sigma_0,\lambda_0}}{q_{(j)}^{\theta,\sigma,\lambda}}\right)^{p_2}\big\vert X_{t_k}^{\theta,\sigma,\lambda}=X_{t_k}\right]\right)^{\frac{1}{p_2}}ds_1\cdots ds_j\bigg\}\\
&\leq C\left(\E\left[\left\vert g(X_{t_{k+1}}^{\theta,\sigma,\lambda})\right\vert^{p_1}\big\vert X_{t_k}^{\theta,\sigma,\lambda}=X_{t_k}\right]\right)^{\frac{1}{p_1}}\left\{1+\sum_{j=1}^{\infty}\left(\frac{\lambda_0}{\lambda}\right)^j\frac{\Delta_n^j}{j!}(e^{-\lambda\Delta_n}\lambda^j)^{\frac{1}{p_2}}\right\}\\
&\leq C\left(\P(\widetilde{J}_{\geq 2,k}\vert Y_{t_{k}}^{\theta,\sigma,\lambda}=X_{t_{k}})\right)^{\frac{1}{p_1q_1}}\left(\E\left[(\Delta M_{t_{k+1}})^{2pp_1q_2}\vert Y_{t_{k}}^{\theta,\sigma,\lambda}=X_{t_{k}}\right]\right)^{\frac{1}{p_1q_2}}\\
&\leq C\Delta_n^{\frac{2}{p_1q_1}}\Delta_n^{\frac{1}{p_1q_2}}\leq C\Delta_n^{\frac{2}{p_1q_1}},
\end{align*}
where we have used the finiteness of the sum w.r.t. $j$, the H\"older's inequality with $\frac{1}{q_1}+\frac{1}{q_2}=1$, and chosen $p_1, q_1$ in order that $p_1q_1<2$. This, together with \eqref{m01}, concludes (\ref{m0}). Here
\begin{align*}
\frac{q_{(0)}^{\theta_0,\sigma_0,\lambda_0}}{q_{(0)}^{\theta,\sigma,\lambda}}\equiv\frac{q_{(0)}^{\theta_0,\sigma_0,\lambda_0}}{q_{(0)}^{\theta,\sigma,\lambda}}(\Delta_n,X_{t_k}^{\theta,\sigma,\lambda},X^{\theta,\sigma,\lambda}_{t_{k+1}}), \
\frac{q_{(j)}^{\theta_0,\sigma_0,\lambda_0}}{q_{(j)}^{\theta,\sigma,\lambda}}\equiv\frac{q_{(j)}^{\theta_0,\sigma_0,\lambda_0}}{q_{(j)}^{\theta,\sigma,\lambda}}(\Delta_n,X_{t_k}^{\theta,\sigma,\lambda},X^{\theta,\sigma,\lambda}_{t_{k+1}};s_1,\ldots,s_j).
\end{align*}

We next show (\ref{m1}). As for the term $M_{0,p}^{\theta,\sigma,\lambda}$, multiplying the random variable inside the conditional expectation of $M_{1,p}^{\theta,\sigma,\lambda}$ by ${\bf 1}_{\widetilde{J}_{0,k}} +{\bf 1}_{\widetilde{J}_{1,k}}+ {\bf 1}_{\widetilde{J}_{\geq 2,k}}$, using ${\bf 1}_{\widetilde{J}_{1,k}^c}=1-{\bf 1}_{\widetilde{J}_{1,k}}$ and Jensen's inequality, we have that $M_{1,p}^{\theta,\sigma,\lambda}\leq 2(M_{1,1,p}^{\theta,\sigma,\lambda}+M_{1,2,p}^{\theta,\sigma,\lambda})$, where
\begin{align*}
M_{1,1,p}^{\theta,\sigma,\lambda}:&=\E\left[{\bf 1}_{\widehat{J}_{1,k}}\widetilde{\E}_{X_{t_k}}^{\theta,\sigma,\lambda}\left[{\bf 1}_{\widetilde{J}_{1,k}^c}\big\vert Y_{t_{k+1}}^{\theta,\sigma,\lambda}=X_{t_{k+1}}\right]\big\vert X_{t_k}\right],\\
M_{1,2,p}^{\theta,\sigma,\lambda}:&=\E\left[{\bf 1}_{\widehat{J}_{1,k}}\widetilde{\E}_{X_{t_k}}^{\theta,\sigma,\lambda}\left[{\bf 1}_{\widetilde{J}_{\geq 2,k}}(\Delta M_{t_{k+1}})^{2p}\big\vert Y_{t_{k+1}}^{\theta,\sigma,\lambda}=X_{t_{k+1}}\right]\big\vert X_{t_k}\right].
\end{align*}
Observe that $M_{1,1,p}^{\theta,\sigma,\lambda}=M_{1,1,0,p}^{\theta,\sigma,\lambda}+M_{1,1,2,p}^{\theta,\sigma,\lambda}$, where
\begin{align*}
M_{1,1,0,p}^{\theta,\sigma,\lambda}:&=\E\left[{\bf 1}_{\widehat{J}_{1,k}}\widetilde{\E}_{X_{t_k}}^{\theta,\sigma,\lambda}\left[{\bf 1}_{\widetilde{J}_{0,k}}\big\vert Y_{t_{k+1}}^{\theta,\sigma,\lambda}=X_{t_{k+1}}\right]\big\vert X_{t_k}\right],\\
M_{1,1,2,p}^{\theta,\sigma,\lambda}:&=\E\left[{\bf 1}_{\widehat{J}_{1,k}}\widetilde{\E}_{X_{t_k}}^{\theta,\sigma,\lambda}\left[{\bf 1}_{\widetilde{J}_{\geq 2,k}}\big\vert Y_{t_{k+1}}^{\theta,\sigma,\lambda}=X_{t_{k+1}}\right]\big\vert X_{t_k}\right].
\end{align*}
Proceeding as in Lemma \ref{change} to change the measures, applying H\"older's inequality with $\frac{1}{p_1}+\frac{1}{p_2}=1$, Jensen's inequality, and Lemma \ref{lemma12} where notice that $\theta\to\theta_0$, $\sigma\to\sigma_0$ as $n\to\infty$, we get that for $n$ large enough,
\begin{align*}
M_{1,1,0,p}^{\theta,\sigma,\lambda}&=e^{(\lambda-\lambda_0)\Delta_n}\frac{\lambda_0}{\lambda}\int_{t_k}^{t_{k+1}}\E\bigg[{\bf 1}_{\{\widehat{J}_{1,k},s_1\}}\frac{q_{(1)}^{\theta_0,\sigma_0,\lambda_0}}{q_{(1)}^{\theta,\sigma,\lambda}}(\Delta_n,X_{t_k}^{\theta,\sigma,\lambda},X_{t_{k+1}}^{\theta,\sigma,\lambda};s_1)\\
&\qquad\times\widetilde{\E}_{X_{t_k}}^{\theta,\sigma,\lambda}\left[{\bf 1}_{\widetilde{J}_{0,k}}\big\vert Y_{t_{k+1}}^{\theta,\sigma,\lambda}=X_{t_{k+1}}^{\theta,\sigma,\lambda}\right]\big\vert X_{t_k}^{\theta,\sigma,\lambda}=X_{t_k}\bigg]ds_1\\
&\leq C\left(M_{1,1,0,1,p}^{\theta,\sigma,\lambda}\right)^{\frac{1}{p_2}},
\end{align*}
for some constant $C>0$. Here, using Bayes' formula, we get that
\begin{align*}
M_{1,1,0,1,p}^{\theta,\sigma,\lambda}:&=\E\left[{\bf 1}_{\widehat{J}_{1,k}}\widetilde{\E}_{X_{t_k}}^{\theta,\sigma,\lambda}\left[{\bf 1}_{\widetilde{J}_{0,k}}\big\vert Y_{t_{k+1}}^{\theta,\sigma,\lambda}=X_{t_{k+1}}^{\theta,\sigma,\lambda}\right]\big\vert X_{t_k}^{\theta,\sigma,\lambda}=X_{t_k}\right]\\
&=\int_{\R}\frac{q_{(0)}^{\theta,\sigma,\lambda}(\Delta_n,X_{t_k},y)e^{-\lambda\Delta_n}}{p^{\theta,\sigma,\lambda}(\Delta_n,X_{t_k},y)}q_{(1)}^{\theta,\sigma,\lambda}(\Delta_n,X_{t_k},y)e^{-\lambda\Delta_n}\lambda\Delta_ndy.
\end{align*}
Again we divide the $dy$ integral into the subdomains $J_1:=\{y\in\R:\vert y-X_{t_k}e^{-\theta\Delta_n}+\frac{\lambda}{\theta}(1-e^{-\theta \Delta_n})\vert>\Delta_n^{\alpha}\}$ and $J_2:=\{y\in\R:\vert y-X_{t_k}e^{-\theta\Delta_n}+\frac{\lambda}{\theta}(1-e^{-\theta \Delta_n})\vert\leq\Delta_n^{\alpha}\}$, where $\alpha\in(0,\frac{1}{2})$, and call each integral $M_{1,1,0,1,1,p}^{\theta,\sigma,\lambda}$ and $M_{1,1,0,1,2,p}^{\theta,\sigma,\lambda}$. In the same way the term $M_{0,1,1,1,p}^{\theta,\sigma,\lambda}$ was treated, using \eqref{de1}, \eqref{q0}, the equality $e^{-\vert x\vert^2}=e^{-\frac{\vert x\vert^2}{2}}e^{-\frac{\vert x\vert^2}{2}}$, valid for all $x\in\R$, we get that for any $\alpha\in(0,\frac{1}{2})$, $n$ large enough, and for some constant $C>0$,
 \begin{align*}
M_{1,1,0,1,1,p}^{\theta,\sigma,\lambda}\leq\int_{J_1}q_{(0)}^{\theta,\sigma,\lambda}(\Delta_n,X_{t_k},y)dy\leq Ce^{-\frac{1}{4\sigma^2}\Delta_n^{2\alpha-1}}.
\end{align*}

Next, as for the term $M_{0,1,1,2,p}^{\theta,\sigma,\lambda}$, using \eqref{de2}, \eqref{qjf} and \eqref{qj}, we get that for $n$ large enough,
\begin{align*}
M_{1,1,0,1,2,p}^{\theta,\sigma,\lambda}\leq e^{-\lambda\Delta_n}\lambda\Delta_n\int_{J_2}q_{(1)}^{\theta,\sigma,\lambda}(\Delta_n,X_{t_k},y)dy\leq Ce^{-\frac{1}{40\sigma^2}\Delta_n^{-1}},
\end{align*}
for some constant $C>0$. Therefore, we have shown that $M_{1,1,0,p}^{\theta,\sigma,\lambda}$ satisfies \eqref{m01}.

Proceeding as for the term $M_{0,2,p}^{\theta,\sigma,\lambda}$, we have that $M_{1,2,p}^{\theta,\sigma,\lambda}+M_{1,1,2,p}^{\theta,\sigma,\lambda}\leq C\Delta_n^{\frac{2}{p_1q_1}}$, for all $p_1>1$, $q_1>1$ with $p_1q_1<2$. This concludes (\ref{m1}). Finally, using H\"older's inequality with $\frac{1}{\mu_1}+\frac{1}{\mu_2}=1$ and $\mu_1\in(1,2)$, and $\P(\widehat{J}_{\geq 2,k}\vert X_{t_k})\leq (\lambda\Delta_n)^2$, it is easy to check that $M_{\geq 2,p}^{\theta,\sigma,\lambda}\leq C\Delta_n^{\frac{2}{\mu_1}}$, for some constant $C>0$, which concludes (\ref{m2}). Thus, the result follows.
\end{proof}

\subsection{Change of measures}
The following technical lemma is used several times in Lemmas \ref{lemma5} and \ref{main}.
\begin{lemma}\label{c2lemma8} For all $k \in \{0,...,n-1\}$,  
\begin{equation*}\begin{split}
\E\left[\widetilde{\E}_{X_{t_k}}^{\theta_n,\sigma_0,\lambda(\ell)}\left[\widetilde{M}_{t_{k+1}}^{\lambda(\ell)}-\widetilde{M}_{t_k}^{\lambda(\ell)}\big\vert Y_{t_{k+1}}^{\theta_n,\sigma_0,\lambda(\ell)}=X_{t_{k+1}}\right]\big\vert X_{t_k}\right]=-\dfrac{\ell w}{\sqrt{n\Delta_n}}\Delta_n.
\end{split}
\end{equation*}
\end{lemma}
\begin{proof}
Using Girsanov's theorem in Lemma \ref{c2Girsanov2}, the independence between $\widetilde{M}_{t_{k+1}}^{\lambda(\ell)}-\widetilde{M}_{t_k}^{\lambda(\ell)}$ and $\frac{d\widehat{\P}}{d \widehat{Q}_k^{\theta_n,\lambda(\ell),\theta_0,\lambda_0,\sigma_0}}$, and the fact that $\E[\widetilde{M}_{t_{k+1}}^{\lambda(\ell)}-\widetilde{M}_{t_k}^{\lambda(\ell)}]=-\frac{\ell w}{\sqrt{n\Delta_n}}\Delta_n$, we obtain that
\begin{align*}
&\E\left[\widetilde{\E}_{X_{t_k}}^{\theta_n,\sigma_0,\lambda(\ell)}\left[\widetilde{M}_{t_{k+1}}^{\lambda(\ell)}-\widetilde{M}_{t_k}^{\lambda(\ell)}\big\vert Y_{t_{k+1}}^{\theta_n,\sigma_0,\lambda(\ell)}=X_{t_{k+1}}\right]\big\vert X_{t_k}\right]\\
&=\E_{\widehat{Q}_k^{\theta_n,\lambda(\ell),\theta_0,\lambda_0,\sigma_0}}\left[\left(\widetilde{M}_{t_{k+1}}^{\lambda(\ell)}-\widetilde{M}_{t_k}^{\lambda(\ell)}\right)\dfrac{d\widehat{\P}}{d \widehat{Q}_k^{\theta_n,\lambda(\ell),\theta_0,\lambda_0,\sigma_0}}\big\vert X_{t_k}\right]\\
&=\E\left[\widetilde{M}_{t_{k+1}}^{\lambda(\ell)}-\widetilde{M}_{t_k}^{\lambda(\ell)}\right]\E_{\widehat{Q}_k^{\theta_n,\lambda(\ell),\theta_0,\lambda_0,\sigma_0}}\left[\dfrac{d\widehat{\P}}{d \widehat{Q}_k^{\theta_n,\lambda(\ell),\theta_0,\lambda_0,\sigma_0}}\big\vert X_{t_k}\right]=-\dfrac{\ell w}{\sqrt{n\Delta_n}}\Delta_n,
\end{align*}
since the second expectation is equal to $1$. Thus, the result follows.
\end{proof}

\noindent{\bf Acknowledgements.} The author would like to thank Professors Arturo Kohatsu-Higa and Eulalia Nualart for fruitful discussions on the subject, and the anonymous referee for all the comments and suggestions that helped to improve the paper. The author acknowledges the financial support from JST-CREST project.


\begin{thebibliography}{00}

\bibitem{AJ07} A{\"{\i}}t-Sahalia, Y. and Jacod, J. (2007), Volatility estimators for discretely sampled {L}\'evy processes, {\it Ann. Statist.}, {\bf 35}(1), 355-392.

\bibitem{BN01} Barndorff-Nielsen, O. E. and Shephard, N. (2001), Non-Gaussian Ornstein-Uhlenbeck-based models and some of their uses in financial economics, {\it J. R.  Statist. Soc. B}, {\bf 63}, 167-241.

\bibitem{CDG14} Cl\'ement, E., Delattre, S. and Gloter, A. (2014), Asymptotic lower bounds in estimating jumps, {\it Bernoulli}, {\bf 20}(3), 1059-1096.

\bibitem{CG15} Cl\'ement, E. and Gloter, A. (2015), Local Asymptotic Mixed Normality property for discretely observed stochastic differential equations driven by stable L\'evy processes, {\it Stochastic Process. Appl.}, {\bf 125}, 2316-2352.

\bibitem{DJ06} Davis, M. H. A. and Johansson, M. P. (2006), Malliavin Monte Carlo Greeks for jump diffusions, {\it Stochastic Process. Appl.}, {\bf 116}, 101-129.

\bibitem{G01} Gobet, E. (2001), Local asymptotic mixed normality property for elliptic diffusions: a Malliavin calculus approach, {\it Bernoulli}, {\bf 7}, 899-912.

\bibitem{G02} Gobet, E. (2002), LAN property for ergodic diffusions with discrete observations, {\it Ann. I. H. Poincar\'e}, {\bf 38}, 711-737.

\bibitem{Haj72} H{\'a}jek, J. (1972), Local asymptotic minimax and admissibility in estimation, {\it Proceedings of the {S}ixth {B}erkeley {S}ymposium on {M}athematical {S}tatistics and {P}robability, {V}ol. {I}: {T}heory of statistics}, 175--194.

\bibitem{J11} Jacod, J. (2011), {\it Statistics and high frequency data}. Lecture Notes in The Fourth European Summer School in Financial Mathematics.

\bibitem{JP82} Jeganathan, P. (1982), On the asymptotic theory of estimation when the limit of the log-likelihood ratios is mixed normal, {\it Sankhy\=a Ser. A}, {\bf 44}(2), 173--212.

\bibitem{K13} Kawai, R. (2013), Local Asymptotic Normality Property for Ornstein-Uhlenbeck Processes with Jumps Under Discrete Sampling, {\it J Theor Probab}, {\bf 26}, 932-967.

\bibitem{KNT14} Kohatsu-Higa, A., Nualart, E. and Tran, N.K. (2014), LAN property for a simple L\'evy process, {\it C. R. Acad. Sci. Paris, Ser. I}, {\bf 352}(10), 859-864.

\bibitem{KNT15} Kohatsu-Higa, A., Nualart, E. and Tran, N.K. (2015), LAN property for an ergodic diffusion with jumps, Preprint: \url{http://arxiv.org/pdf/1506.00776v1.pdf}

\bibitem{LC60} Le~Cam, L. (1960), Locally asymptotically normal families of distributions, {\it Univ. California, Publ. Statist}, {\bf 3}, 37-98.

\bibitem{CY90} Le~Cam, L. and Lo Yang, G. (1990), {\it Asymptotics in statistics: Some basic concepts}, Springer Series in Statistics. Springer-Verlag, New York.

\bibitem{M14} Mai, H. (2014), Efficient maximum likelihood estimation for L\'evy-driven Ornstein-Uhlenbeck processes, {\it Bernoulli}, {\bf 20}(2), 919-957.

\bibitem{N} Nualart, D. (2006), {\it The Malliavin Calculus and Related Topics}, Second Edition, Springer.

\bibitem{P08} Petrou, E. (2008), Malliavin {C}alculus in {L}\'evy spaces and {A}pplications to {F}inance, {\it Electron. J. Probab.}, {\bf 13}, 852-879.

\bibitem{SK} Sato, K. (1999), {\it L\'evy Processes and Infinitely Divisible Distributions}, Cambridge Univ. Press, Cambridge.

\bibitem{SY06} Shimizu, Y. and Yoshida, N. (2006), Estimation of Parameters for Diffusion Processes with Jumps from Discrete Observations, {\it Stat. Inference Stoch. Process.}, {\bf 9}(3), 227-277.
\end{thebibliography}
\end{document}